\documentclass[11pt]{article}
\usepackage{setspace}%,subcaption}
\usepackage{amsmath,hyperref}
\usepackage{lingmacros}
\usepackage{amsfonts}
\usepackage{amsthm}
\usepackage{amsmath}
\usepackage{amscd}
\usepackage[latin2]{inputenc}
\usepackage{t1enc}
\usepackage[mathscr]{eucal}
\usepackage{indentfirst}
\usepackage{graphicx}
\usepackage{graphics}
\usepackage{pict2e}
\numberwithin{equation}{section}
\usepackage[margin=1.9cm]{geometry}
\usepackage{epstopdf}
\usepackage[T1]{fontenc}
\usepackage{amssymb}
\usepackage[normalem]{ulem}
\usepackage{comment}
\usepackage{xcolor}
\useunder{\uline}{\ul}{}

 \theoremstyle{plain}
\newtheorem{Th}{Theorem}[section]
\newtheorem{Lemma}[Th]{Lemma}

\newtheorem{Prop}[Th]{Proposition}

\newtheorem{assumption}{Assumption}
\theoremstyle{definition}
\newtheorem{Def}[Th]{Definition}

\newtheorem{Rem}{Remark}
\newtheorem{?}[Th]{Problem}

\newcommand\relphantom[1]{\mathrel{\phantom{#1}}}

\newcommand{\E}{\mathbb{E}}

\newcommand\tm{\tilde{m}}

\newcommand\tb{\tilde{b}}
\newcommand\tc{\tilde{c}}
\newcommand{\ts}{\tilde{s}}

\newcommand{\tX}{\tilde{X}}

\newcommand\mI{\mathcal{I}}
\newcommand\mD{\mathcal{D}}
\newcommand\mC{\mathcal{C}}
\newcommand\mX{\mathcal{X}}
\newcommand\mY{\mathcal{Y}}
\newcommand\mZ{\mathcal{Z}}

\newcommand\mV{\mathcal{V}}
\newcommand\mU{\mathcal{U}}
\newcommand{\mS}{\mathcal{S}}
\newcommand{\mT}{\mathcal{T}}

\newcommand{\hf}{\hat{f}}

\newcommand{\hQ}{\hat{Q}}

\newcommand{\hG}{\hat{G}}
\newcommand{\hH}{\hat{H}}

\newcommand{\bba}{\bar{a}}
\newcommand{\bbi}{\bar{i}}
\newcommand{\bbj}{\bar{j}}

\newcommand{\bx}{\mathbf{x}}
\newcommand{\by}{\mathbf{y}}

\newcommand{\bQ}{\mathbf{Q}}
\newcommand{\bR}{\mathbf{R}}

\newcommand{\bW}{\mathbf{W}}

\newcommand{\bX}{\mathbf{X}}
\newcommand{\bY}{\mathbf{Y}}

\allowdisplaybreaks

\begin{document}

\title{Tracy-Widom law for the extreme eigenvalues of large signal-plus-noise matrices}
\author{Zhixiang Zhang and Guangming Pan \vspace{0.5cm}\\
School of Physical and Mathematical Sciences \\
 Nanyang Technological University\\
 Singapore}
\date{}

\maketitle

\begin{abstract}
Let $\bY =\bR+\bX$ be an $M\times N$ matrix, where $\bR$ is a rectangular diagonal matrix and $\bX$ consists of $i.i.d.$ entries. This is a signal-plus-noise type model. Its signal matrix could be full rank, which is rarely studied in literature compared with the low rank cases. This paper is to study the extreme eigenvalues of $\bY\bY^*$. We show that under the high dimensional setting ($M/N\rightarrow c\in(0,1]$) and some regularity conditions on $\bR$ the rescaled extreme eigenvalue converges in distribution to Tracy-Widom distribution ( $TW_1$).

\end{abstract}

\section{Introduction}
Consider a signal-plus-noise data matrix. %has been popular in the last few decades.  
 It takes the form of \begin{eqnarray}\label{splusnmo}
\bY= \bR+ \bX,
\end{eqnarray} where $\bR$ is the signal matrix and $\bX$ is the noise matrix.  This model is popular in many fields such as machine learning \cite{yang2016rate}, matrix denoising \cite{nadakuditi2014optshrink} or signal processing \cite{vallet2012improved}.  A lot of research has been devoted to the spectral properties of signal-plus-nose type matrices.  Such a signal-plus-noise model without a finite rank structure on $\bR$ has been investigated in several papers. For example, \cite{dozier2007onempirical} studied the limiting spectral distribution(LSD) of  $\bY\bY^*$ and \cite{dozier2007empirical} provided analytic properties of the stieltjes transform of LSD. \cite{banna2020clt} established a central limit theorem for the linear spectral statistics of $\bY\bY^*$.

There are also a lot of work focusing on the cases with low rank signals. For example,  \cite{loubaton2011almost} derived the almost localization of the spiked eigenvalues, \cite{ding2020high} obtained the convergent limits and rates of the leading eigenvalues and eigenvectors, \cite{bao2018singular} showed the distributions of the principal singular vectors and singular subspaces. More works can be found in \cite{hachem2007deterministic,benaych2011eigenvalues,vallet2012improved,cape2019signal}.

\indent However, it may be the case that the signals are not low rank matrices, such as the direction of arrival (DOA) estimation \cite{mani2010direction}. Extreme eigenvalues are of primary interest in statistics due to their roles in PCA or factor analysis. Hence in this paper, we consider the extreme eigenvalues of $\bY\bY^*$ for the case when $\bR$ is rectangular diagonal (not necessarily low rank) and $\bX$ consists of i.i.d entries.  We show that under suitable conditions on $\bR$ the rescaled extreme eigenvalues of $\bY\bY^*$ converge in distribution to the celebrated Tracy-Widom distribution.

%Spectral properties of $\bY$ with low-rank assumptions on the $\bR$ and  different covariance structures of $\bX$ have been investigated widely with fruitful results, for example,  \cite{hachem2007deterministic,benaych2011eigenvalues,loubaton2011almost,bao2018singular, vallet2012improved, ding2020high}.

%A lot of research has been devoted to study spectral properties of signal-plus-nose type matrices $\bY\bY^*$ with $\bR$ being low rank. For example,  \cite{loubaton2011almost} studied the almost localization of $\bY\bY^*$, \cite{ding2020high} derived the convergent limits and rates of leading eigenvalues and eigenvector statistics, \cite{bao2018singular} obtained the singular vector and singular subspace distribution, see also \cite{hachem2007deterministic,benaych2011eigenvalues,vallet2012improved,cape2019signal}

%From the perspective of random matrix theory, a phase transition phenomenon occurres frequently for extreme eigenvalues of different random matrix models. One may refer to \cite{baik2005phase}.
%Without attempting to be comprehensive, one can refer to... for instance.

%Most current theoretical results on studying full rank type of information-plus-noise type matrices $\bY\bY^*$ requires that the noise matrices contain i.i.d. entries.
 Tracy widom distribution has been widely established for different types of random matrices. Tracy and Widom \cite{tracy1994level, tracy1996orthogonal} firstly derived the distribution of extreme eigenvalue for GOE and GUE and named it Tracy-Widom distribution.  Johnstone established the Tracy-Widom distribution for the standard Wishart matrices. These works utilize the joint distribution of eigenvalues of studying model. In the last decade, significant development has been made on universality property of random matrices theory, such as \cite{tao2010random} on Wigner matrices and  \cite{erdHos2012rigidity} on sample covariance matrices. These results show that the the limiting behaviours of the eigenvalue statistics do not depend on the specific distribution of matrix entries, also referred to as universality property.   A necessary and sufficient condition to guarantee the Tracy-Widom distribution is provided in \cite{lee2014necessary} for Wigner matrices and \cite{ding2018necessary} for sample covariance matrices.  Extensive research including \cite{lee2015edge, han2016tracy, fan2017tracy, bao2019canonical} have signified that extreme eigenvalues of a large group of random matrices have Tracy-Widom distributions. In addition, Tracy-Widom distribution has been used frequently in hypothesis testing, such as testing  the number of factors \cite{onatski2009testing}, testing the covariance structure of covariance \cite{han2016tracy}, testing linear independence between high dimensional vectors \cite{bao2019canonical}.
% In the last decade, the dynamic approach has drawn huge attention to researchers.

The method used in this paper follows from \cite{lee2015edge, lee2016tracy}. A dynamic signal matrix is constructed to link the target matrix with a sample covariance matrix, and then we make a continuous Green function comparison over the time flow. The analysis requires a local law as an ingredient, that is, an estimation of the difference between a Green function of the target matrix to a fixed limit uniformly on a domain of the complex plane. We establish an entrywise local law and apply it to the proof of the Tracy-Widom distribution. We mention that usually the usefulness of the local law of a random matrix model is not restricted to study of the edge behaviour.  It is a powerful tool to study rigidity of eigenvalues,  universality of  correlation function, or eigenvector distributions, including but not limited to these, see the summary in \cite{benaych2016lectures, knowles2017anisotropic}.  Once the local law for our model is obtained,  the remaining challenge is to track the rate of change of Green functions over time flow accurately. Optical theorems provide some cancel mechanisms to obtain derivative of Green functions with respect to time parameter with a desired order.

We remark that a direct application of our results yields the Tracy-Widom distribution for general $\bR$ (not necessarily diagonal) with gaussian $\bX$. We believe from the universality property of random matrices that the Tracy-Widom law holds for general $\bX$ by using a Green function comparison strategy similar to \cite{erdHos2012rigidity} and \cite{bao2015universality}. It is also possible to obtain a necessary and sufficient condition on the entries of $\bX$  by using method in \cite{lee2014necessary}. We will pursue these in the coming work. The results in this paper also lead to the study of a spiked model that $\bR$ has some spiked singular values. Spiked sample eigenvalues may carry valuable information about the structure of signals which is worth to study.

 We also notice that a model similar to ours was studied in very recent works \cite{ding2020edge, ding2020tracy}. However, they considered the case when the variance of the entries of noise matrix tends to $0$, which is very different from our case.

The rest of the paper is organized as follows. In Section 2 we state our main results regarding to the Tracy-Widom distribution of extreme eigenvalues and the local law of $\bY\bY^*$.  Section 3 includes some preliminaries. In Section 4 we rescale the target model and prove the Tracy-Widom distribution by providing the standard Green function comparison argument. In Section 5 we interpolate the target matrix with the sample covariance matrix over a time flow and analyze the rate of change of Green functions assuming the validity of local law and Optical theorems, which complete the proof of Green function comparison in Section 4. In Section 6 we prove the local law. And in the Section 7, we derive the Optical theorems which involve tedious calculations.

 %The usefulness of local law is not restricted to study of edge behaviour.  It is a powerful tool to study rigidity of eigenvalues,  universality of  correlation function, or eigenvector distributions, including but not limited to these.
%Some conventions: We use $C$ to denote generic constants that do not dependent on N.  For $a< b$, we set $[\![a, b]\!]:= [a,b]\cap \mathbb{Z}$.The indicator function of event $\Xi$ is denoted by $1(\Xi)$. The big O and small o notations  $O(\cdot)$ and $o(\cdot)$  are in their standard meaning.

\section{Main Results}
In this section, we state the main results of the paper. First we give the result on the Tracy-Widom distribution of the extreme eigenvalues of the signal plus noise matrices. Then we state the local law for the signal plus noise model, which serves as the main technical input in the proof of Tracy-Widom distribution.
\subsection{Tracy widom law for signal plus noise matrices}
We consider a matrix $\bY$ of the form \eqref{splusnmo} under the high dimensional setting, where $\bR$ is an $M\times N$ deterministic matrix contains non-zero entries $d_1\geq d_2 \geq \cdots \geq d_M\geq 0$ at main-diagonal position, i.e.  \begin{equation}\bR=\left( \begin{array} { c c c c c } { d_ { 1 } }
 & { \dots } & { 0 } & { \cdots } & { 0 } \\ { 0 } & { \ddots } & { 0 } & { \cdots } & { 0 } \\ { 0 } & { \cdots } & { d_ {M} } & { \cdots } & { 0 } \end{array} \right),
\end{equation}
and $\bX$ contains $i.i.d.$ entries.  To be more specific, we assume the following.
 \begin{assumption}\label{assump1}We assume that $\bX=(x_{ij}) $ is an $M\times N$ matrix, whose entries $\{x_{ij}:1\leq i \leq M, 1\leq j \leq N\}$ are independent real random variables satisfying $$Ex_{ij}=0, \quad E|x_{ij}|^2 =\frac{1}{N}.$$
Moreover, we assume that for all $p\in \mathbb{N}$, there is $C_p$ such that $$E|\sqrt{N}x_{ij}|^p\leq C_p.$$\end{assumption}
 \begin{assumption}\label{assump2}
 $c_N:=M/N \rightarrow c \leq 1$.
 \end{assumption}

Let $\bQ:=\bY\bY^*=(\bR+\bX)(\bR+\bX)^*$. Denote the empirical spectral distribution(ESD) of  $\bR\bR^*$ by $$\hat{\rho} :=\frac{1}{M}\sum_{i=1}^M  \delta_{d_i^2},$$ where $\delta_x$ is the Dirac measure at point $x$.
Use $\mu_1\geq \mu_2, \cdots, \geq \mu_M$ to denote the ordered eigenvalues of $\bQ$. The ESD of $\bQ$ is $$F^{\bQ}:=\frac{1}{M}\sum_{i=1}^M  \delta_{\mu_i}.$$
The Stieltjes transform of $F^{\bQ}$ is \begin{eqnarray}
s_N(z)=s_{\bQ}(z):=\int \frac{dF^{\bQ}(t)}{t-z}.
\end{eqnarray}
If $\hat{\rho} \rightarrow \rho$ and $M/N\rightarrow c$, it was shown in \cite{dozier2007onempirical} that $F^{\bQ}$ converges to a deterministic distribution $F^{c,\rho}$  whose stieltjes transform $s_0(z)$ is defined through a self-consistent equation: \begin{eqnarray}\label{lslaw}s_0 = \int  \frac{d\rho(t)}{\frac{t}{1+c s_0}-z(1+cs_0)+(1-c)}, \quad s_0\in \mathcal{C}^+, \Im(zs_0)\geq 0.\end{eqnarray}
 In the following, we need the non-asymptotic version of $s_0$, denoted by $s$, which is the unique solution in $\mathcal{C}^+$ satisfying \eqref{lslaw} where we replace $c$ and $\rho$ by $c_N$ and $\hat{\rho}$ respectively.  According to Theorem 2.1 in \cite{dozier2007empirical},  $\lim_{z\rightarrow E, z\in \mathcal{C}_+}s(z)$ exists for $E\in \mathcal{R}/\{0\}$, so the definition of $s(z)$ can be extended to $z\in \mathcal{C}_+\cup \mathcal{R}/\{0\} $, and we still denote it by $s$.  By the well known inverse formula of the Stieltjes transform,  we know that the non asymptotic distribution $F^{c,\rho}$ has a continuous density $\rho_0$  with \begin{eqnarray}
\rho_0 (E)=\lim_{\eta\rightarrow 0} \frac{1}{\pi}\Im m(E+i \eta)=\frac{1}{\pi}\Im m(E),  \quad E\in \mathcal{R}/\{0\}.
\end{eqnarray}
 % For notational simplicity, we still use $s$ in the following paper to denote the non-asymptotic version of stieltjes transform.
For convenience, we also need to introduce the $N\times N$ conjugate matrix $\tilde{\bQ}=\bY^*\bY$ which shares the same non-zero eigenvalues of $\bQ$.
We have that $s_{\tilde{\bQ}}=-\frac{1-c_N}{z}+c_N s_{\bQ}.$
 % $\tilde{F}^{c,\rho}$ to denote the non asymptotic version of limit of $F^{\tilde{\bQ}_N}$, and
If we use $\ts$ to denote the non-asymptotic version of limit of $s_{\tilde{\bQ}}$, it is easy to verify that
\begin{eqnarray}
\ts = -\frac{1-c}{z}+cs.
\end{eqnarray}

We introduce the following notations that were used in \cite{loubaton2011almost}, but omit the subscript $N$ for ease of notation:
\begin{eqnarray}\label{Linkfun}
\begin{aligned}
&f(w) =\frac{1}{M} Tr(\bR \bR^*-w I_M)^{-1},\\
&\phi(w)=w(1-cf(w))^2+(1-c)(1-cf(w)),\\
&w(z)=z(1+cs(z))^2-(1-c)(1+cs(z)).
\end{aligned}
\end{eqnarray}
Let $\xi_r$ be the largest solution to $\phi^{\prime} (w)=0$ and satisfy $1-cf(w)>0$.  Denote \begin{eqnarray}\label{defedge}\lambda_r = \phi(\xi_r).\end{eqnarray}
 By Theorem 3 in \cite{loubaton2011almost}, $\lambda_r$  is the rightmost boundary of support of $F^{c,\rho}$.  By Lemma 2 in \cite{loubaton2011almost}, we also have $\xi_r=\lim_{\eta\rightarrow 0}w(\lambda_r+i\eta)=w(\lambda_r).$
Denote \begin{eqnarray}\label{defb}b(z)=1+cs(z),  \quad \text{and}  \quad  b\equiv b(\lambda_r)=1+cs(\lambda_r).\end{eqnarray}
Then according to the last equation in \eqref{Linkfun}, we have the relation \begin{eqnarray}\label{defxi}
\xi_r=\lambda_r b^2-(1-c)b.
\end{eqnarray}
From \eqref{lslaw}, and according to Theorem 2.1 in \cite{dozier2007onempirical}, we can infer that \begin{eqnarray}\label{relaofbxi}
1-cf(\xi_r)=\frac{1}{b}.
\end{eqnarray}
Since \begin{eqnarray}\phi'(w)=(1-cf(w))^2-2cw(1-cf(w))f'(w)-c(1-c)f'(w),\end{eqnarray}
by the fact that $\phi'(\xi_r)=0$ together with \eqref{defb}, \eqref{defxi} and \eqref{relaofbxi},
 we have \begin{eqnarray}\label{firstorderrelation}
%\frac{1}{N}\sum_{i=1}^M\frac{\lambda_r b^2+d_i^2}{(d_i^2-\xi_r)^2}=1
\frac{1}{N}\sum_{i=1}^M\frac{b(\lambda_r b^2+\xi_r)}{(d_i^2-\xi_r)^2}=1.
\end{eqnarray}

To ensure that there is a square root behaviour near the edge $\lambda_r$, we need the following assumption.

\begin{assumption} \label{assump3}There exists positive constant $\hat{c}$, such that $$\limsup_{N\rightarrow \infty}d_1^2-w(\lambda_r) < -\hat{c}.$$
\end{assumption}
\begin{Rem} We consider two examples to see that this assumption is reasonable.\\
Example 1. If $d>0$ and $d_i=d$ for all $i$, then by some calculations, we find that $\phi'(w) = 0$ is equivalent to  $g(w):=-w^3+3d^2 w^2-(3d^4-2cd^2-c)w+d^6-2cd^4+(2c^2-c)d^2=0$. Observing that $g(d^2)=2c^2d^2>0$ and $g(+\infty)<0$, by the mean value theorem, there exists at least one $w>d^2$, s.t. $\phi'(w)=0$. Denote the largest one by $\xi_r$, and it is easy to verify that $1-cf(\xi_r)>0$ since $f(\xi_r)<0$. Therefore, we conclude that there exists a positive constant $\hat{c}$, such that $d^2-\xi_r<\hat{c}.$ \\
Example 2. If $c=1$, and  $RR^*=\mbox{diag}(2,2-1/(p-1),2-2/(p-1),\cdots, 1+1/(p-1), 1)$, then the limiting spectral distribution of $RR^*$ is a uniform distribution on the interval $[1,2]$. We find the largest solution to $\phi'(w)=0$ is approximately 3.89 and $1-f(3.89)>0$. The Assumption 2 also holds obviously.
\end{Rem}

Recall that $b=b(\lambda_r)=1+cs(\lambda_r)$.
Denote $\gamma_0$ to be the solution of the following equation,
\begin{eqnarray}\label{defgamma}
\frac{1}{\gamma_0^3}\frac{1}{N}\sum_{i=1}^M\frac{b^2}{(d_i^2-\xi_r)^2}=-\frac{1}{b^3}-\frac{1}{N}\sum_{i=1}^M\frac{(2\lambda_r b-(1-c))^2}{(d_i^2-\xi_r)^3}-\frac{1}{N}\sum_{i=1}^M\frac{\lambda_r}{(d_i^2-\xi_r)^2}.
\end{eqnarray}
We are ready to give the main theorem of this paper.
\begin{Th}\label{Th1twlaw}
Suppose that Assumptions 1,2,3 hold. Let $\mu_1$ be the largest eigenvalue of $\bQ$. Then there exists $\gamma_0$ defined in \eqref{defgamma} such that the distribution of the rescaled eigenvalue of $\bQ$ converges to the  type-1 Tracy-Widom distribution, i.e.,
\begin{eqnarray}
\lim_{N\rightarrow \infty} P\left(\gamma_0 N^{2/3}(\mu_1 - \lambda_r)\leq t\right) = F_1(t).
\end{eqnarray}
\end{Th}
\begin{Rem}
If one considers $\tilde{\bQ}=(\bR+\sigma \bX)(\bR+\sigma \bX)^*$ for constant $\sigma>0$,   the Tracy-Widom distribution still holds by some modification on Assumption 3. Notice that we can write $\tilde{\bQ} = \sigma^2 (\tilde{\bR}+\bX) (\tilde{\bR}+\bX)^*$ where $\tilde{\bR}=\sigma^{-1}\bR$. Therefore by replacing $\bR$ with $\tilde{\bR}$ and  redefining $f, \phi, w, \xi_r, \lambda_r, b$ from \eqref{Linkfun} to \eqref{defb}, under an assumption similar to Assumption 3, the Tracy-Widom law still holds for the largest eigenvalue of $\tilde{\bQ}$ with a different limit and a rescaling constant.
\end{Rem}
We use a simple example to illustrate this result. Consider $\bR= I_{M\times M}$. Then we find $\xi_r = 3$, $\lambda_r = 27/4$ and $\gamma_0=(27^2/16)^{-1/3}$. We plot the histogram of $ \gamma_0 N^{2/3}(\mu_1 - \lambda_r)$ for 10000 generated matrices in Figure \ref{fig1}, we see that the simulation results fit the Tracy-Widom law very well.

\begin{figure}[]
    \includegraphics[width=10cm]{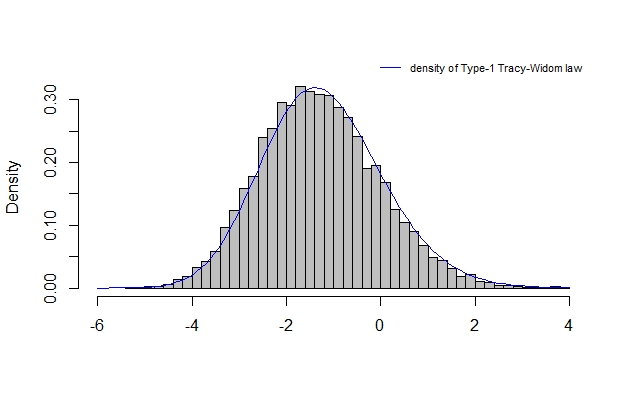}\centering
    \caption{The histogram of centralized and rescaled largest eigenvalues for $(I+X)(I+X)^*$, with $M=N=500$ and 10000 matrices.}\label{fig1}
\end{figure}

\subsection{Local law for signal plus noise matrices}
The proof of Theorem \ref{Th1twlaw} requires  accurate estimation of Green function entries of $\bQ$, known as local law. This is formulated by introducing linearization matrices \begin{eqnarray}\label{HRplusX}
H(z) := \left(\begin{array}{cc}-zI_M & {\bY} \\ {\bY^{*}} & {- I_{N}}\end{array}\right),
\end{eqnarray}
and \begin{eqnarray}\label{defG346}
G(z) := H^{-1}(z)
\end{eqnarray}

We introduce notations to label the entries of $H$ and $G$.
Define index sets
$$\mI_M:=[\![1,M]\!] , \quad \mI_N:=[\![M+1,M+N]\!],\quad \mI := \mI_M\cup \mI_N = [\![1,M+N]\!].$$
We consistently use Latin letters, e.g. $i,j$, to denote indices in $\mI_M$, and Greek letters, e.g. $\mu,\nu$, to denote indices in $\mI_N$. %Sometimes we will use $s,t \in \mI$.
Then we label the indices of $\bY$ according to  \begin{eqnarray}\bY= (Y_{i\mu}: i \in \mI_M, \mu \in \mI_N), \end{eqnarray}
and $\bX$ is relabelled accordingly.
 Note that when $\mu=i+M$, $Y_{i\mu}=d_i+x_{i\mu}.$  We use \begin{eqnarray}
\bbi:=i+M
\end{eqnarray}
to denote such special Greek letter index corresponding to each $i\in \mI_M$, i.e, we have $Y
_{i\bbi}=d_i+x_{i\bbi}$. Furthermore, we see in  Theorem \ref{Lolaw} below that the Green function entries $G_{i\bbi}$ has a non-zero limit. We shall call $(i,\bbi)$ as index pair. \\
%We also call $(i,\mu_i)$ or $(\mu, i_{\mu})$ as index pairif $\mu_i = i+M \in \mI_N$, or $i_{\mu}=\mu-M\in \mI_M$. We can also say that ``$\mu_i$ is paired with $i$ '', or ``$i_{\mu}$ is paired with $\mu$''.
%One significant difference between our result with local law on sample covariance matrices is that the entries of Green function $G$  located at indexed pairs  do not approximate to 0.  For example, $G_{1\mu_1}$ and $G_{M+1,i_{M+1}}$ are of constant order.

Let $T\subset \mI$.  We define the minor $H^{(T)}:=(H_{st}:s, t\in \mI \setminus S)$ by removing all the rows and columns of $H$ indexed by $T$. Note that we keep the original values of the matrix indices. We also write $G^{(T)}=(H^{(T)})^{-1}$.
For sufficiently small positive $\tau, \tc$, we define the domain of spectral parameter $z$ by
\begin{eqnarray}
\mathcal{D}(\tc,\tau)\equiv\mathcal{D}(\tc,\tau,N):=\{z = E+i\eta \in \mathcal{C}^{+},|z|\geq\tau,E_+ - \tc \leq E\leq \tau^{-1},N^{-1+\tau}\leq\eta\leq\tau^{-1}\}.
\end{eqnarray}
Recall $b(z)$ defined in \eqref{defb}. Let \begin{eqnarray}\label{deftb890}\tb(z) = z b(z) -(1-c).\end{eqnarray}
We are ready to provide the local law.
\begin{Th}\label{Lolaw}(Strong local law) Under the assumptions 1,2,3, for any sufficiently small $\epsilon>0$ and sufficiently large $D>0$, with probability $1-O(N^{-D})$ we have
\begin{eqnarray}\label{Lolaw1}
\bigcap_{z\in \mD(\tc,\tau)}\Big\{ |s_N(z)-s(z)|\leq \frac{N^\epsilon}{N\eta} \Big\},
\end{eqnarray} and \begin{eqnarray}\begin{aligned}\label{Lolaw2}
\bigcap_{z\in \mD(\tc,\tau)}\Big\{ & \max_{i\in \mI_M}|G_{ii}-\frac{b(z)}{d_i^2-w(z)}|+  \max_{i \in \mI_M}|G_{\bbi\bbi}-\frac{\tb(z)}{d_i^2-w(z)}|+\max_{i\in \mI_M}|G_{i \bbi}-\frac{d_i}{d_i^2-w(z)}|\\&+\max_{\mu \in [\![2M+1,M+N]\!]}|G_{\mu\mu}+\frac{1}{b(z)}| +\max_{s\neq t, (s,t)  \text{is not index pair}}|G_{st}|\leq N^\epsilon\Big(\sqrt{\frac{\Im s(z)}{N\eta}}+\frac{1}{N\eta}\Big)\Big\}.
\end{aligned}\end{eqnarray}
\end{Th}

 \begin{Prop}\label{crare}(Convergence rate at right edge)
Recall the largest eigenvalue of matrix $\bQ$ is denoted by $\mu_1.$ For any sufficiently small $\epsilon>0$ and sufficiently large $D>0$, we have \begin{eqnarray}
|\mu_1-\lambda_r| \leq N^{2/3+\epsilon},
\end{eqnarray}
holds with probability $1-O(N^{-D})$.
\end{Prop}

\section{Preliminaries}
We use notations introduced in \cite{erdHos2013averaging} which provide a simple and systematic way to control the orders of random variables or their expectations,  and that will be used throughout this paper.
\begin{Def}\label{defprec}
Let $$X =\{ X^{(N)}(u): N\in \mathcal{N}, u\in U^{(N)}\}, Y =\{ Y^{(N)}(u): N\in \mathcal{N}, u\in U^{(N)}\}$$  be two families of nonnegative random variables and where $U^{(N)}$ is a possibly N-dependent parameter set. We say that $X$ is stochastically dominated by $Y$, uniformly in $u$, if for all fixed small $\epsilon>0$, and large $D>0$,
$$\sup_{u\in U^{(N)}} P\left(X^{(N)}(u)>N^\epsilon Y^{(N)}(u)\right) \leq N^{-D},$$ for all large $N\geq N(\epsilon,D),$ and we denote this by $X\prec Y$ or $X=O(Y)$. Moreover, if for some complex valued family $X$ we have $|X|\prec Y$, we still write $X\prec Y$ or $X=O(Y)$.
\end{Def}
%We use $O(\cdot)$ and $o(\cdot)$ in its standard meaning
We say that an event $\Xi$ holds with high probability if $1\prec 1(\Xi)$, i.e., $P(\Xi)=1-O(N^{-D})$ for any large constant $D>0$.

\begin{Rem}
We have following arithmetic rules for $\prec$.
Suppose that $X(u,v)\prec Y(u,v)$  uniformly in $u \in U$ and $v\in V$. If $|V|=O(N^C)$ for some constant $C$, then $\sum_u X(u,v)\prec \sum_u Y(u,v)$. Suppose that $X_1(u)\prec Y_1(u)$ uniformly in $u\in U$ and $X_2(u)\prec Y_2(u)$ uniformly in $u\in U$, then $X_1(u)X_2(u)\prec Y_1(u)Y_2(u)$ uniformly in $u\in U$.  These two properties can be proved by a simple union bound. Furthermore, if $\Gamma:=\Gamma(u)$ is deterministic and  $\Gamma(u)\geq N^{-C}$ for some $C>0$,  let $X=X(u)$  satisfies $X\prec \Gamma$ and $EX^2 \leq N^{C'}$, then $EX\prec \Gamma$. This follows from an application of Chebyshevis inequality.
\end{Rem}
 We also need the following large deviation estimates taken from \cite{erdHos2013local}. \begin{Lemma}\label{ldb} (Large deviation bounds). Let $\left(X_{i}\right),\left(Y_{i}\right),\left(a_{i j}\right),$ and $\left(b_{i}\right)$ be independent families of random variables, where $N \in \mathbb{N}$ and $i, j=1, \ldots, N .$ Suppose that all entries $X_{i}$ and $Y_{i}$ are independent and satisfy the conditions
$\mathbb{E} X=0, \E |X|^2=1, \E|X|^{p} \leq C_{p}$ for all p with some constants $C_p$.\\
(i) Suppose that $\left(\sum_{i}\left|b_{i}\right|^{2}\right)^{1 / 2} \prec \Psi .$ Then $\sum_{i} b_{i} X_{i} \prec \Psi.$\\
(ii) Suppose that $\left(\sum_{i \neq j}\left|a_{i j}\right|^{2}\right)^{1 / 2} \prec \Psi .$ Then $\sum_{i \neq j} a_{i j} X_{i} X_{j} \prec \Psi.$\\
(iii) Suppose that $\left(\sum_{i, j}\left|a_{i j}\right|^{2}\right)^{1 / 2} \prec \Psi .$ Then $\sum_{i, j} a_{i j} X_{i} Y_{j} \prec \Psi.$
\end{Lemma}
 %From Schur formula, \begin{eqnarray}
%G(z) =\left(\begin{array}{cc}G_M(z) & {G_M(z) \bY} \\ {\bY^{*}G_M(z)} & {G_N(z)}\end{array}\right),
%\end{eqnarray}
%where $G_M(z)= (\bY\bY^{*}-z I_M)^{-1}$ and $G_N(z)=(\bY^*\bY/z-I_N)^{-1}$.
We record the following Green function identities, see Lemma 4.4 in  \cite{knowles2017anisotropic} for instance. % The $G^{(T)}_M$ can be defined similar to $G^{(T)}$
\begin{Lemma}\label{gfi}(Green function identities) Let $G \equiv G(z)$ be defined in \eqref{defG346}.We have\\
(i) For any $i\in \mI_M$,and $\mu \in \mI_N$, we have \begin{eqnarray}\label{gfidiag}\begin{aligned}
&G_{ii}=\frac{1}{-z-(\bY G^{(i)}\bY^*)_{ii}},\\
&G_{\mu\mu}=\frac{1}{-1-(\bY^* G^{(\mu)}\bY)_{\mu\mu}}.
\end{aligned}
\end{eqnarray}
(ii)For $i\neq j \in \mI_M$
\begin{eqnarray}\label{gfiij}
G_{ij}=-G_{ii}(\bY G^{(i)})_{ij}=-G_{jj}(G^{(j)}\bY^*)_{ij}=G_{ii}G_{jj}^{(i)}(\bY G^{(ij)}\bY^*)_{ij}.
\end{eqnarray}
For $\mu\neq \nu \in \mI_N$
\begin{eqnarray}\label{gfimunu}
G_{\mu\nu}=-G_{\mu\mu}(\bY^* G^{(\mu)})_{\mu\nu}=-G_{\nu\nu}(G^{(\nu)}\bY)_{\mu\nu}=G_{\mu\mu}G^{(\mu)}_{\nu\nu}(\bY^* G^{(\mu\nu)}\bY)_{\mu\nu}.
\end{eqnarray}
(iii)
For $i\in \mI_M$ and $\mu \in \mI_N$, we have
\begin{eqnarray}\label{gfiimu}\begin{aligned}
&G_{i\mu}=-G_{\mu\mu}(G^{(\mu)}\bY)_{i\mu}=-G_{ii}(\bY G^{(i)})_{i\mu},\\
&G_{\mu i}=-G_{\mu\mu}(\bY^* G^{(\mu)})_{\mu i}  =-G_{ii}(G^{(i)}\bY^*)_{\mu i}.
\end{aligned}
\end{eqnarray}
(iv) For $r\in \mI$ and $s,t \in  \mI \backslash \{k\}$, we have
\begin{eqnarray}\label{gfiminus}
G_{st}^{(r)}=G_{st}-\frac{G_{sr}G_{rt}}{G_{rr}}.
\end{eqnarray}
(v) All of the identities above hold if we replace $G$ by $G^{(T)}$.
\end{Lemma}

\section{Green function comparison and Proof of Theorem \ref{Th1twlaw}}
The main technical ingredient to obtain Theorem \ref{Th1twlaw} is pursuing a continuous Green function comparison near the rightmost edge. This has been successfully used to derive the Tracy Widom law for deformed Wigner matrix \cite{lee2015edge} and real sample covariance matrix with general population covariance structure \cite{lee2016tracy}.  In Subsection \ref{subsec41} we rescale $\bQ$ and state the local law of rescaled matrix for later use. In Subsection \ref{subsec42} we give the Green function comparison result and prove Theorem \ref{Th1twlaw}.
\subsection{Rescaling of the model}\label{subsec41}
We need to rescale the matrix $\bQ$ such that the eigenvalue gap of the renormalized matrix $\hat{\bQ}=\gamma \bQ$ has same eigenvalue gaps predicted by the Tracy-Widom distribution. The rescaling factor $\gamma$ for our target matrix is actually defined by $\gamma_0$ through equation \eqref{defgamma}.
We now set \begin{eqnarray}\label{ResQ1}\hat{\bY} = \gamma_0^{1/2} (\bR+\bX), \quad \hat{\bQ}:= \hat{\bY}  \hat{\bY}^*.\end{eqnarray}
The rescaled linearization matrix is  defined by
 \begin{eqnarray}\label{defH21}
\hH(z) := \left(\begin{array}{cc}-z I_M & \hat{\bY}  \\  \hat{\bY}^{*} & {- I_{N}}\end{array}\right)
\end{eqnarray}
Let
\begin{eqnarray}\label{Gre21}
\hG(z) := \hH^{-1}(z)=\left(\begin{array}{cc}G_1 & {G_1 \hat{\bY}} \\ {\hat{\bY}^{*}G_1} & {G_2}\end{array}\right),
\end{eqnarray}
where $G_1(z)= (\hat{\bY}\hat{\bY}^{*}-z I_M)^{-1}$ and $G_2=(\hat{\bY}^*\hat{\bY}/z-I_N)^{-1}$. The averaged Green function of $\hat{\bQ}$ is \begin{eqnarray}\label{avegrefun}m_N(z):=\frac{1}{M} Tr G_1(z).\end{eqnarray}
%Since we need the information matrix $R$ change with time parameter $t$ continuously in the proof, the rescaling factor also evolves with time $t$, so we use $\gamma:=\gamma(t)$ to denote rescaling factor of $Q(t)$.
Denote the stieltjes transform of limiting spectral distribution of $\hat{\bQ}$ by $m=m(z)$, which is the unique solution of equation
\begin{eqnarray}\label{defm}
m(z)= \frac{1}{M} \sum_{i=1}^M  \frac{1}{\frac{\gamma_0 d_i^2}{1+c \gamma_0 m(z)}-z(1+c \gamma_0 m(z))+\gamma_0 (1-c)}, \quad  m \in \mathcal{C}^+.
\end{eqnarray}
$m(z)$ can be extended to $R/\{0\}$.
Define \begin{eqnarray}\label{Linkfunc2}\begin{aligned}
&\hf(\hat{w}) =\frac{1}{M} Tr(\gamma_0 \bR \bR^*-\hat{w} I_M)^{-1}\\
&\hat{\phi}(w)=w(1-c\gamma_0 \hf(\hat{w}))^2+\gamma_0
(1-c)(1-c\gamma_0 \hf(\hat{w}))\\
&\hat{w}(z)=z(1+c\gamma_0 m(z))^2-\gamma_0(1-c) (1+c\gamma_0 m(z))
\end{aligned}\end{eqnarray}
We denote by $\xi$  the largest solution of $\hat{\phi'}(w)=0$.  Denote the edge for $\hQ$ by $E_+$, satisfying $E_+ = \hat{\phi}(\xi) $, which parallels the relation in \eqref{defedge}.
Note that $\gamma_0 m(E_+)= s(\lambda_r)$. Recall  $b=b(\lambda_r)$ defined in $\eqref{defb}$, we also have
\begin{eqnarray}\label{rescb}
b=1+c\gamma_0 m(E_+),
\end{eqnarray}and
\begin{eqnarray}\label{rescxi}
\xi =  E_+ b^2-\gamma_0 (1-c) b=\gamma_0 \xi_r.
\end{eqnarray}
Denote by $\tm$ the Stieltjes transform of limiting spectral distribution of $\hat{\bY}^*\hat{\bY}$.  Define \begin{eqnarray}\label{deftb}
\tb = E_+[1+\gamma_0 \tm(E_+)].
\end{eqnarray}
It is easy to see that $\tb = E_+ b-\gamma_0(1-c)$ from the fact that  $\tm(z) =-(1-c)/z+cm$. It follows easily that \begin{eqnarray}\label{xibtb}\xi =b\tb.\end{eqnarray}

Now we can rewrite equation \eqref{firstorderrelation} and \eqref{defgamma} as
\begin{eqnarray}\label{R1}%\frac{1}{N}\sum_{i=1}^N \frac{\gamma_0 E_+b^2+\gamma_0^2 d_i^2}{(\gamma_0 d_i^2-\xi)^2}=1,
\frac{1}{N}\sum_{i=1}^M \frac{\gamma_0  b(E_+ b^2+\xi)}{(\gamma_0 d_i^2-\xi)^2}=1,
\end{eqnarray}
and \begin{eqnarray}\label{R2}
\frac{1}{\gamma_0^2}\frac{1}{N}\sum_{i=1}^M\frac{b^2}{(\gamma_0 d_i^2-\xi)^2}=-\frac{1}{\gamma_0 b^3}-\frac{1}{N}\sum_{i=1}^M\frac{(2 E_+ b- \gamma_0(1-c))^2}{(\gamma_0 d_i^2-\xi)^3}-\frac{1}{N}\sum_{i=1}^M\frac{E_+}{(\gamma_0 d_i^2-\xi)^2}.
\end{eqnarray}
The equation \eqref{R1} can be understood as multiplying $\gamma_0^2$ on both numerator and denominator of the LHS of \eqref{firstorderrelation}. For \eqref{R2}, it is obtained by dividing $\gamma_0$ on both sides of \eqref{defgamma}.
%These two equations   play crucial roles in analysing the

Next, we collect some estimates on the Green functions of $\hat{\bQ}$, which will be used later in Section \ref{secgfflowprof42} . For $\tau > 0$, and sufficiently small positive constant $\tc$, define the domain
\begin{eqnarray}
\hat{\mathcal{D}}(\tc,\tau)\equiv \hat{\mathcal{D}}(\tc,\tau,N):=\{z\in \mathcal{C}_{+},|z|\geq\tau,E_+ - \tc \leq E\leq \tau^{-1},N^{-1+\tau}\leq\eta\leq\tau^{-1}\}.
\end{eqnarray}
\begin{Lemma}\label{gfcon}
With slight abuse of notations, we let $b(z)=1+\gamma_0 cm(z)$ and  $\tb(z)= zb(z)-\gamma_0(1-c)$. Under assumptions 1,2,3, we have that
%\begin{eqnarray}\bigcap_{z\in \mD(\tc,\tau)}\Big\{ |m_N(z)-m(z)|\leq \frac{N^\epsilon}{N\eta} \Big\}\end{eqnarray} holds with high probability and
%\begin{eqnarray}\begin{aligned}&\bigcap_{z\in \mD (\tc,\tau)}\Big\{ \max_{i \in \mI_M}| \hG_{ii}-\frac{b}{\gamma_0 d_i^2-\xi}|+\max_{i \in \mI_M}|\hG_{\bbi\bbi}-\frac{\tb}{\gamma_0 d_i^2-\xi}|+\Big\{ \max_{i\in \mI}|\hG_{ii}^{(\bbi)}+\frac{1}{\tb}|+\max_{i\in \mI} |\hG_{\bbi\bbi}^{(i)} + \frac{1}{b}|\\& +\max_{\mu \in [\![2M+1,M+N]\!]}|\hG_{\mu\mu}+\frac{1}{b}|+\max_{i\in \mI_M}|G_{i \bbi}-\frac{\gamma_0^{1/2}d_i}{\gamma_0 d_i^2-\xi}|+\max_{s\neq t, (s,t)  \text{is not index pair}}|\hG_{st}|\leq N^\epsilon\Big(\sqrt{\frac{\Im m(z)}{N\eta}}+\frac{1}{N\eta}\Big)\Big\}\end{aligned}\end{eqnarray}
\begin{eqnarray}
|m_N(z)-m(z)|\prec \frac{1}{N\eta},
\end{eqnarray}  and
\begin{eqnarray}\begin{aligned}
& \max_{i \in \mI_M}| \hG_{ii}-\frac{b(z)}{\gamma_0 d_i^2-\hat{w}(z)}|+\max_{i \in \mI_M}|\hG_{\bbi\bbi}-\frac{\tb(z)}{\gamma_0 d_i^2-\hat{w}(z)}|+ \max_{i\in \mI_M}|\hG_{ii}^{(\bbi)}+\frac{1}{\tb(z)}|+\max_{i\in \mI_M} |\hG_{\bbi\bbi}^{(i)} + \frac{1}{b(z)}|\\& +\max_{\mu \in [\![2M+1,M+N]\!]}|\hG_{\mu\mu}+\frac{1}{b(z)}|+\max_{i\in \mI_M}|\hG_{i \bbi}-\frac{\gamma_0^{1/2}d_i}{\gamma_0 d_i^2-\hat{w}(z)}|+\max_{s\neq t, (s,t)  \text{is not index pair}}|\hG_{st}|\prec \sqrt{\frac{\Im m(z)}{N\eta}}+\frac{1}{N\eta}
\end{aligned}
\end{eqnarray}
uniformly in $z \in \hat{\mD}(\tc,\tau)$
\end{Lemma}
\begin{proof}
The proof is similar to that of Theorem \ref{Lolaw} without essential difference.
\end{proof}

\subsection{Green function comparison}\label{subsec42}
We introduce rescaled sample covariance matrices \begin{eqnarray}\label{defW}\bW:= c^{1/6}(1+\sqrt{c})^{-4/3}XX^*,\end{eqnarray} and let $$m_\bW(z) = \frac{1}{N}Tr(\bW-z)^{-1}.$$
Recall the renormalized matrix $\hat{\bQ}$ defined in \eqref{ResQ1}. The following Green function comparison proposition is the key technical input for the proof of  Theorem \ref{Th1twlaw}.
\begin{Prop}\label{LinComp}
Let $\epsilon>0$ and set $\eta = N^{-2/3-\epsilon}.$ Denote the right edge of $\bW$  by $M_+$. Let $E_1$ and $E_2$ satisfy \begin{eqnarray}
|E_1|,  |E_2|\leq N^{-2/3+\epsilon}.
\end{eqnarray}
Let  $K: \mathbb {R} \rightarrow \mathbb {R}$ be a real function satisfying
$$\sup _{x}\left|K^{(k)}(x)\right| /(|x|+1)^{C} \leq C, \quad k=0,1,2,3,4.$$
Then there exists a constant $\phi>0$ such that for any sufficiently large $N$ and for any sufficiently small $\epsilon > 0$, we have \begin{eqnarray}
\left| EK \left(N\int_{E_1}^{E_2} \Im m_{\hat{\bQ}}(x+E_++i\eta ) \right)-EK \left(N\int_{E_1}^{E_2} \Im m_{\bW}(x+M_++i\eta) \right) \right |\leq N^{-\phi}.
\end{eqnarray}
\end{Prop}
\begin{Rem}
Proposition \ref{LinComp} can be easily extended to a general form, see Remark 5.1 in \cite{lee2015edge}, or Remark 4.2 in \cite{lee2016tracy}.\\
\textsc{\textbf{Proof of Theorem \ref{Th1twlaw}}}.  We need to assume the validity of Proposition \ref{LinComp} which is proved in the next section,  and require a similar result to Lemma 3.4 in \cite{lee2016tracy} using Lemma \ref{gfcon}. Then the proof is similar to that of Theorem 2.4 in \cite{lee2016tracy}, or see the proof of Theorem 1 in \cite{pillai2014universality}.
\end{Rem}

\section{Green function flow and Proof of Proposition \ref{LinComp} }\label{secgfflowprof42}
In this section we prove Proposition \ref{LinComp}.  The strategy is to consider a continuous interpolation between $\bW$ and $\hat{\bQ}$ by introducing a dynamic signal matrix $\bR$ with time parameter $t$ and track rate of change of Green functions over the time flow precisely. The key Lemmas are Lemma \ref{E-derivofG-formu} and \ref{imcancel} which estimate derivatives of Green functions accurate enough.  \\
\indent We use shorthand notation $$\sum_i^{(T)} :=\sum_{\substack{i=1\\i\notin T}}^M , \quad \quad \sum_\mu^{(T)} :=\sum_{\substack{\mu=M+1\\\mu\notin T}}^{M+N}.$$
If $T=\{i\}$, we abbreviate $(i)=(\{i\})$; similarly, when $T=\{a,b\}$, write $(ab)=(\{a,b\})$.
If there are both Latin letter indices and Greek letter indices, we use shorthand notation \begin{eqnarray}
\sum_{i,\mu,\nu}^{(\bbi)}:=\sum_{i}\sum_{\mu,\nu}^{(\bbi)},\quad \quad \sum_{i,j,\mu,\nu}^{(\bbi)}:=\sum_{i,j}\sum_{\mu,\nu}^{(\bbi)}
\end{eqnarray}
\subsection{Green Function flow}\label{subsec51}
We interpolate the deterministic signal matrix $\bR$ and zero matrix by defining
\begin{equation}\bR(t)=\left( \begin{array} { c c c c c } { d_ {1}(t) }
 & { \dots } & { 0 } & { \cdots } & { 0 } \\ { 0 } & { \ddots } & { 0 } & { \cdots } & { 0 } \\ { 0 } & { \cdots } & { d_ {M}(t) } & { \cdots } & { 0 } \end{array} \right),
\end{equation} where \begin{eqnarray}\label{evofdi}d_i(t)=e^{-t/2} d_i,  \quad i=1,\cdots, M.   \quad (t\geq 0).
\end{eqnarray}
We can define $\lambda_r(t)$, $b(t)$ and $\xi_r(t)$ according to  \eqref{defedge}, \eqref{defb} and \eqref{defxi} respectively. The only difference is that in current case, $R(t)$ evolves with time parameter $t$. Denote $\gamma(t) $ be the solution of  \begin{eqnarray}\label{defevogamma}
\frac{1}{\gamma(t)^3}\frac{1}{N}\sum_{i=1}^N\frac{b(t)^2}{(d_i(t)^2-\xi_r(t))^2}=-\frac{1}{b(t)^3}-\frac{1}{N}\sum_{i=1}^N\frac{(2\lambda_r(t) b(t)-(1-c))^2}{(d_i(t)^2-\xi_r(t))^3}-\frac{1}{N}\sum_{i=1}^N\frac{\lambda_r(t)}{(d_i(t)^2-\xi_r(t))^2}.
\end{eqnarray}

Let \begin{eqnarray}
\hat{\bY}(t)= \gamma^{1/2}(t) (\bR(t) + \bX),  \quad  \hat{\bQ}(t)=\hat{\bY}(t)\hat{\bY}(t)^*.
\end{eqnarray}
Now we consider the rescaled linearization matrix defined by
 \begin{eqnarray}\label{defH22}
\hH(z,t) := \left(\begin{array}{cc}-z I_M & \hat{\bY}(t) \\ {\hat{\bY}(t)^{*}} & {- I_{N}}\end{array}\right),
\end{eqnarray}
The resolvent
\begin{eqnarray}\label{Gre22}
\hG(z,t) := \hH(z,t)^{-1}=\left(\begin{array}{cc}G_1(z,t) & {G_1(z,t)\hat{\bY}(t)} \\ {\hat{\bY}(t)^{*}G_1(z,t)} & {G_2(z,t)}\end{array}\right),
\end{eqnarray}
where $G_1(z,t)= [\hat{\bY}(t)\hat{\bY}(t)^{*}-z I_M]^{-1}$ and $G_2=[\hat{\bY}(t)^*\hat{\bY}(t)/z-I_N]^{-1}$.

Now the matrix $\hat{\bQ}(t)$ evolves with time parameter $t$, moreover $\hat{\bQ}(0) = \hat{\bQ}$ where $\hQ$ is defined in \eqref{ResQ1} and $\hat{\bQ}(\infty)=\bW$. Recall the argument from \eqref{Linkfunc2} to \eqref{rescxi}. By replacing $d_i$ and $\gamma_0$ with $d_i(t)$ and $\gamma(t)$ in \eqref{Linkfunc2}, we can similarly
define the edge $E_+(t)$  that varies with $t$, so as the $b(t), \xi(t)$ and $\tb(t)$ appeared in \eqref{rescb}, \eqref{rescxi} and \eqref{deftb}.  With slight abuse of notations, we still use notation $E_+, b, \xi, \tb$ frequently and omit the dependence on $t$ in most cases later for simplicity.  Also for ease of notation, we use notion $\bY$, $G$ instead of $\hat{\bY}$, $\hG$ to stand for rescaled data matrix and Green function of rescaled signal plus noise matrix throughout the rest of  this section, so as in Section \ref{secoptth}. We remark that with these notations,  Lemma \ref{gfi} still holds for the rescaled model without any difference.
%We should bear in mind that the quantities $z$, $\gamma$ and $R$ appeared above are all $t$-dependent, and we omit $t$ for simplicity. More specifically, we still use function defined in \eqref{Linkfun} formally, but since $d_i, i=1,\cdots, M$ in function $f$ vary with $t$, $\phi$ and $w$ also depend on $t$. Therefore, if we define the $\xi_r$ to be the largest solution of $\phi'(w)=0$, this $\xi_r$ depends on $t$. Similarly, define $\lambda_r$ as in \eqref{defedge},$b$ as in \eqref{defb} and $\gamma$ through equation \eqref{defgamma}. Denote the edge for $\tQ(t)$ by $E_+(t)$, and it is easy to conclude that $E_+(t)=\gamma\lambda_r$.
Choose the domain of spectral parameter $z$ as \begin{eqnarray}\label{domainofz}
\mathcal{E}_\epsilon(t)=\{z = E_+(t)+y+i\eta\in \mC^+: y\in [-N^{-2/3+\epsilon},N^{-2/3+\epsilon}], \eta =N^{-2/3-\epsilon}\}.
\end{eqnarray}
The averaged Green function $m_N(t,z)$ (defined in \eqref{avegrefun} for $\hat{\bQ}$ is defined by
\begin{eqnarray}
m_N(t,z):=\frac{1}{M}tr(\hat{\bQ}(t) -zI)^{-1}
\end{eqnarray}
Recalling Lemma \ref{gfcon}, we have
\begin{eqnarray*}
|m_N(t,z)-m(t,z) |\prec \Psi,
\end{eqnarray*}
uniformly in $\mathcal{E}_{\epsilon}(t), t\geq 0$, where
\begin{eqnarray}\label{defPsi}
\Psi:=N^{-1/3+C'\epsilon},
\end{eqnarray} for some small constant $C'$ independent of $N,\epsilon$ and $t$. Furthermore, we have $m(t,z)-m(E_+) \prec \Psi $ for $z \in \mathcal{E}_{\epsilon}(t)$, see Lemma \ref{mYZorder} below. Therefore, \begin{eqnarray}
m_N(t,z)-m(E_+)\prec \Psi \quad \text{for $z$ uniformly in} \; \mathcal{E}_{\epsilon}(t), t\geq 0
\end{eqnarray}
Similarly, we can obtain
\begin{eqnarray}\label{0011a}\begin{aligned}
& \max_{i \in \mI_M}| G_{ii}-\frac{b}{\gamma d_i^2-\xi}|+\max_{i \in \mI_M}|G_{\bbi\bbi}-\frac{\tb}{\gamma d_i^2-\xi}|+ \max_{i\in \mI_M}|G_{ii}^{(\bbi)}+\frac{1}{\tb}|+\max_{i\in \mI_M} |G_{\bbi\bbi}^{(i)} + \frac{1}{b}|\\& +\max_{\mu \in [\![2M+1,M+N]\!]}|G_{\mu\mu}+\frac{1}{b}|+\max_{i\in \mI_M}|G_{i \bbi}-\frac{\gamma^{1/2}d_i}{\gamma d_i^2-\xi}|+\max_{s\neq t, (s,t)  \text{is not index pair}}|G_{st}|\prec \Psi.
\end{aligned}
\end{eqnarray}

\subsection{Proof of Proposition \ref{LinComp}}\label{sebsec52}
In this subsection, we prove Proposition \ref{LinComp} by using an accurate estimate on the imaginary part of derivatives of Green function with respect to time parameter $t$, see \eqref{90i4o}.  \\
%which is accurate enough to prove Proposition \ref{LinComp}.\\

\noindent \textsc{\textbf{Proof of Proposition \ref{LinComp}}.} Similar to the proof of Proposition 4.1 in \cite{lee2016tracy},   we also only consider the case $K' \equiv 1$ for simplicity.
Using Lemma \ref{E-derivofG-formu} together with Lemma \ref{imcancel} below, we find \begin{eqnarray}\label{90i4o}
\E \frac{1}{N}\Im \sum_i \frac{\partial G_{ii}}{\partial t}=O(N^{1/2}\Psi^3).
\end{eqnarray}
Integrating on both sides of \eqref{90i4o} from $t=0$ to $t=2\log N$, we get \begin{eqnarray}\label{maog3}
\left| \E N\int_{E_1}^{E_2} \Im m_N(x+E_+ + i\eta)|_{t=0}dx - \E N\int \Im m_N(x+E_+ +i\eta)|_{t=2\log N}dx \right|\leq  N^{-1/6+C\epsilon},
\end{eqnarray}
for some constant $C>0$.\\
\indent From \eqref{evofdi}, at $t=\infty$, $d_i=0$ for $i=1,\cdots M$. Therefore, according to definition of $\xi_r$,$\lambda_r$, $b$ and $\gamma$ (see \eqref{Linkfun}-\eqref{defgamma}), we have
$$\xi_r(\infty)=\sqrt{c}, \quad \lambda_r=(1+\sqrt{c})^2, \quad b(\infty)=\frac{1}{1+\sqrt{c}}, \quad  \gamma(\infty)=(1+\sqrt{c})^{-4/3} c^{1/6}.$$
Let $t_0 = 2 \log N$. It can be easily shown that $\gamma(t_0)=\gamma(\infty)+O(N^{-2})$, and $z(t_0)=z(\infty)+O(N^{-2})$. Using matrix identity $A^{-1}-B^{-1}=-A^{-1}(A-B)B^{-1}$ and \eqref{0011a}, we have $$|G_{ii}(t_0)-G_{ii}(\infty)| = \left| -\sum_{a,b}G_{ia}(t_0)\left(\hat{\bQ}_{ab}(t_0)-\hat{\bQ}_{ab}(\infty)\right) G_{bi}\right|\prec N^{-2/3}.$$
It follows that \begin{eqnarray}\label{maog4}
\left| \E N\int_{E_1}^{E_2} \Im m_N(x+E_+ + i\eta)|_{t=2\log N}dx - \E N\int \Im m_N(x+E_+ +i\eta)|_{t=\infty}dx \right|\leq  N^{-1/3+C\epsilon},
\end{eqnarray}for some constant $C>0$.
Note that $m_N(x+E_+ +i\eta)|_{t=\infty}=m_{W}(x+M_+ +i\eta)$, we conclude the proof from \eqref{maog3} and \eqref{maog4}.
\subsection{Derivatives of Green functions}\label{subsec53}
%Recall the definitions of $H$ in \eqref{defH} and $G$ in \eqref{Gre21}.Let $T\subset \mI$.  We define the minor $H^{(T)}:=(H_{st}:s, t\in \mI \setminus S)$ by removing all the rows and columns of $H$ indexed by $T$. Note that we keep the original values of the matrix indices. We also write $G^{(T)}=(H^{(T)})^{-1}$. Moreover, we use shorthand notation $$\sum_i^{(T)} :=\sum_{\substack{i=1\\i\notin T}}^M , \quad \quad \sum_\mu^{(T)} :=\sum_{\substack{\mu=M+1\\\mu\notin T}}^{M+N}.$$
%If $T=\{i\}$, we abbreviate $(i)=(\{i\})$; similarly, when $T=\{a,b\}$, write $(ab)=(\{a,b\})$.
%If there are both Latin letter indices and Greek letter indices, we use shorthand notation \begin{eqnarray}\sum_{i,\mu,\nu}^{(\bbi)}:=\sum_{i}\sum_{\mu,\nu}^{(\bbi)},\quad \quad \sum_{i,j,\mu,\nu}^{(\bbi)}:=\sum_{i,j}\sum_{\mu,\nu}^{(\bbi)}\end{eqnarray}
%In the following, let $Y:=\gamma^{1/2} (R+X)$.

We now consider the derivative of $G_{ii}$ with respect to the time parameter $t$ for $z\in \mathcal{E}_\epsilon(t)$, see \eqref{domainofz}. In the following, we denote by $\dot{z}$ the derivative of $z$ with respect to time parameter $t$, and similarly denote $\dot{d}_i$ and $\dot{\gamma}$, see \eqref{evofdi} and \eqref{defevogamma}. A direct calculation yields $\dot{d_i}=-\frac{d_i(t)}{2}$. Then we have
\begin{eqnarray}\label{derivofdiag}\begin{aligned}\frac{\partial{G_{ii}}}{\partial t} &= \dot{z}\sum_{j} G_{ij} G_{ji} - 2\sum_{j,\mu} G_{ij} \frac{\partial Y_{j\mu}}{\partial t} G_{\mu i}\\&=\dot{z}\sum_{j} G_{ij} G_{ji} +\gamma^{1/2} \sum_j G_{ij} d_j G_{\bbj i}-\gamma^{-1/2}\dot{\gamma} \sum_{j,\mu} G_{ij} (X_{j\mu}+R_{j\mu}) G_{\mu i}\\&=\dot{z}\sum_{j} G_{ij} G_{ji} +\sum G_{i\bbj} G_{\bbj i}-\gamma^{1/2}\sum_{j,p} G_{ip} x_{p\bbj}G_{\bbj i}
\\&\relphantom{EEE}-\gamma^{-1}\dot{\gamma}\sum_{j} G_{i\bbj} G_{\bbj i}+\gamma^{-1/2}\dot{\gamma}\sum_{j,p} G_{ip}x_{p\bbj}G_{\bbj i}-\gamma^{-1/2}\dot{\gamma} \sum_{j,\mu} G_{ij} x_{j\mu} G_{\mu i},
\end{aligned}\end{eqnarray}
where in the last step, we use $G_{i\bbj} = \gamma^{1/2} G_{ij} d_j+ \gamma^{1/2}\sum_{p} G_{ip} x_{p\bbj},$ which can be obtained from  \eqref{Gre21}. To obtain the desired order of the expectation of  $\E \frac{1}{N}\Im \sum_i \frac{\partial G_{ii}}{\partial t}$ as in \eqref{90i4o},  we start from calculating expectation of right side of \eqref{derigii} above and ignore those terms with order $O(N^{1/2}\Psi^3)$. \\

To handle the third term and fourth terms above, we need the following lemma which was used in \cite{khorunzhy1996asymptotic}.
\begin{Lemma} Let $\ell \in \mathbb{N}$ be fixed and assume that  $f \in C^{\ell+1}(\mathbb{R})$. Let $\xi$ be a centered random variable
with finite first $\ell+2$ moments. Then we have the expansion
$$
\E(\xi f(\xi))=\sum_{k=1}^{\ell} \frac{\kappa_{k+1}(\xi)}{k !} \E \left(f^{(k)}(\xi)\right)+C_{\ell} \E\left(|\xi|^{\ell+2}\right) \sup _{|t| \in R}\left|f^{(\ell+1)}(t)\right|
$$
where  $\kappa_{k}(\xi), k \in \mathbb{N}$, are the cumulants of $\xi.$
\end{Lemma}

Notice that under Assumption \ref{assump1}, $\kappa_2(x_{ij})= 1/N$, and $\kappa_p(x_{ij})\leq C_p'/N^{p/2}$ for some $C_p'$ independent of $N$. With above lemma, we see that the expectation of the last term in  \eqref{derivofdiag} is \begin{eqnarray}\label{k3term1}
\begin{aligned}
\E\gamma^{-1/2}\sum_{j,\mu} G_{ij} x_{j\mu}G_{\mu i}&=-\frac{1}{N}\sum_{j,\mu} (2\E G_{ij}G_{j\mu}G_{\mu i}+\E G_{i\mu}G_{jj}G_{\mu i}+\E G_{ij}G_{\mu \mu}G_{ji})\\ &\relphantom{EEEE}  +\kappa_3 \gamma^{1/2} \sum_{j,\mu} 3 \E G_{ij}G_{jj}G_{\mu\mu}G_{\mu i} + O(N^{1/2}\Psi^3),
\end{aligned}
\end{eqnarray}
where we use the fact that those terms like $\sum_{j,\mu}\frac{\kappa_3}{N^{3/2}}G_{i\mu} G_{jj}G_{j\mu}G_{\mu i}$   is of $O(N^{1/2} \Psi^3)$, and $\sum_{j,\mu} \frac{\kappa_3}{N^{3/2}}G_{ij}G_{j\mu}G_{\mu j}G_{ji}$  is of $O(N^{1/2}\Psi^4)$.
For the third term $\frac{\kappa_3}{N^{3/2}}\sum 6 \E G_{ij}G_{jj}G_{\mu\mu}G_{\mu i},$ we may suppose $j\neq i$, otherwise it is of order $O(N^{-1/2}\Psi)$ which is negligible. We have the following lemma to control the order of $\E G_{ij}G_{jj}G_{\mu\mu}G_{\mu i}$.
\begin{Lemma}\label{EGiijjmumumui}If $j\neq i$,
\begin{eqnarray}\E G_{ij}G_{jj}G_{\mu\mu}G_{\mu i}=O(\Psi^3).\end{eqnarray}
\end{Lemma}
\begin{proof}
Using \eqref{0011a} and \eqref{gfiminus}, we see that \begin{eqnarray}\label{4q3oa}
\E G_{ij}G_{jj}G_{\mu\mu}G_{\mu i}=\frac{b\tb}{(\gamma d_a^2-\xi)^2}\E G_{ij}G_{\mu i} + O(\Psi^3)=\frac{b\tb}{(\gamma d_a^2-\xi)^2}\E G_{ij}G_{\mu i}^{(j)} + O(\Psi^3).
\end{eqnarray}
Using resolvent identity \eqref{gfiij}, we have $$G_{ij}=G_{ii}G_{jj}^{(i)}\sum_{\alpha \beta} Y_{i\alpha}G_{\alpha\beta}^{(ij)}Y_{j\beta}=\gamma\left(G_{ii}G_{jj}^{(i)}d_i d_j G_{\bbi\bbj}^{(ij)}+\sum_\alpha d_j x_{i\alpha} G_{\alpha\bbj}^{(ij)}+\sum_\beta d_i G_{\bbi\beta}^{(ij)}x_{j\beta}+\sum_{\alpha,\beta} x_{i\alpha}G_{\alpha\beta}^{(ij)}x_{j\beta}\right). $$
Substituting above into \eqref{4q3oa} and omitting $\gamma $ since we only care about order of terms, we see that the non vanishing term is \begin{eqnarray*}\E G_{ii}G_{jj}^{(i)}d_i d_j G_{\bbi\bbj}^{(ij)}G_{\mu i}^{(j)}+\E \sum_\alpha d_j x_{i\alpha} G_{\alpha\bbj}^{(ij)}G_{\mu i}^{(j)}.
\end{eqnarray*}
We have that \begin{eqnarray}\begin{aligned}
\E G_{ii}G_{jj}^{(i)}d_i d_j G_{\bbi\bbj}G_{\mu i}^{(j)}&=\E \frac{b\tb}{(\gamma d_a^2-\xi)^2}d_i d_j G_{\bbi\bbj}^{(ij)}G_{\mu i}^{(j)}+O(\Psi^3)\\&=-\E \frac{b\tb}{(\gamma d_a^2-\xi)^2}d_i d_j G_{\bbj\bbj}^{(ij)}\sum_{k} x_{k\bbj} G_{\bbi k}^{(ij\bbj)}G_{\mu i}^{(j)}+O(\Psi^3)\\&=\E \frac{\tb}{(\gamma d_a^2-\xi)^2}d_i d_j \sum_{k} x_{k\bbj} G_{\bbi k}^{(ij\bbj)}G_{\mu i}^{(j\bbj)}+O(\Psi^3)=O(\Psi^3)
\end{aligned}\end{eqnarray}
and \begin{eqnarray}\begin{aligned}
\E \sum_{\alpha} d_j x_{i\alpha} G_{\alpha\bbj}^{(ij)}G_{\mu i}^{(j)}&=-\sum_{\alpha} \E d_j x_{i\alpha} G_{\alpha\bbj}^{(ij)} G_{ii}^{(j)}\sum_{\nu} x_{i\nu}G_{\mu\nu}^{(ij)}=-\frac{1}{N}\frac{b}{\gamma d_i^2-\xi}\sum_{\alpha} \E d_j G_{\alpha\bbj}^{(ij)}G_{\mu\alpha}^{(ij)}+O(\Psi^3)\\&=-\frac{1}{N}\frac{b}{\gamma d_i^2-\xi}\sum_{\alpha} \E d_j G_{\bbj\bbj}^{(ij)}x_{k\bbj}G_{\alpha k}^{(ij\bbj)}G_{\mu\alpha}^{(ij)}+O(\Psi^3) = O(\Psi^3).
\end{aligned}
\end{eqnarray}
Therefore we conclude this lemma.
\end{proof}

We now turn back to \eqref{k3term1}. Using Lemma \ref{EGiijjmumumui}, we have \begin{eqnarray}\label{9a8gaf}
\E\gamma^{-1/2}\sum_{j,\mu} G_{ij} x_{j\mu}G_{\mu i}=-\frac{1}{N}\sum_{j,\mu} (2\E G_{ij}G_{j\mu}G_{\mu i}+\E G_{i\mu}G_{jj}G_{\mu i}+\E G_{ij}G_{\mu \mu}G_{ji})+O(N^{1/2}\Psi^3).
\end{eqnarray}
Similarly,  we have \begin{eqnarray}\label{k3term2}
\E \sum_{j,p} G_{ip} x_{p\bbj}G_{\bbj i}=-\frac{\gamma^{1/2}}{N} \sum _{j,p}\left(2\E G_{ip}G_{p\bbj}G_{\bbj i} + \E G_{i\bbj}G_{pp}G_{\bbj i} +  \E G_{ip}G_{\bbj\bbj} G_{pi} \right)+O(N^{1/2}\Psi^3).
\end{eqnarray}
Plugging \eqref{9a8gaf} and \eqref{k3term2} into \eqref{derivofdiag}, we obtain that \begin{eqnarray}\label{derigii}\begin{aligned}
\E\frac{\partial G_{ii}}{\partial t} =& \dot{z}\sum_{j}\E G_{ij} G_{ji} +\sum_j \E G_{i\bbj} G_{\bbj i}+\frac{\gamma}{N}  \sum _{j,p}\left(2\E G_{ip}G_{p\bbj}G_{\bbj i} + \E G_{i\bbj}G_{pp}G_{\bbj i} +  \E G_{ip}G_{\bbj\bbj} G_{pi} \right)\\&-\gamma^{-1}\dot{\gamma}\sum_{j} \E G_{i\bbj} G_{\bbj i}-\frac{\dot{\gamma}}{N}\sum _{j,p}\left(2\E G_{ip}G_{p\bbj}G_{\bbj i} +  \E G_{i\bbj}G_{pp}G_{\bbj i} +  \E G_{ip}G_{\bbj\bbj} G_{pi} \right)\\&+\frac{\dot{\gamma}}{N} \sum_{j,\mu} (2\E G_{ij}G_{j\mu}G_{\mu i}+\E G_{i\mu}G_{jj}G_{\mu i}+\E G_{ij}G_{\mu \mu}G_{ji})+O(N^{1/2}\Psi^3)
\\&=\dot{z}\sum_{j}\E G_{ij} G_{ji} +\sum_j \E G_{i\bbj} G_{\bbj i}+\frac{\gamma}{N}  \sum _{j,p}\left(2\E G_{ip}G_{p\bbj}G_{\bbj i} + \E G_{i\bbj}G_{pp}G_{\bbj i} +  \E G_{ip}G_{\bbj\bbj} G_{pi} \right)\\&-\gamma^{-1}\dot{\gamma}\sum_{j} \E G_{i\bbj} G_{\bbj i}+\frac{\dot{\gamma}}{N} \sum_{\mu>2M, j} \left( 2\E G_{ij}G_{j\mu}G_{\mu i}+\E G_{i\mu}G_{jj}G_{\mu i}+\E G_{ij}G_{\mu \mu}G_{ji}\right)
+O(N^{1/2}\Psi^3)
\end{aligned}
\end{eqnarray}
Actually we have a simpler expression of  $\E \frac{1}{N}\Im \sum_i \frac{\partial G_{ii}}{\partial t}$ given in Lemma \ref{E-derivofG-formu} below. We need the following notations.
Let \begin{eqnarray}\label{Xset1}\begin{aligned}
&X_{22}=\frac{1}{N}\sum_{r}G_{ir}G_{ri}, & \quad   & X_{32}=(cm(E_+)-cm_N)\frac{1}{N}\sum_{r} G_{ir}G_{ri},\\
&X_{33}=\frac{1}{N^2}\sum G_{ir}G_{rs}G_{si},&  &  X_{42}=(cm(E_+)-cm_N)^2\frac{1}{N}\sum G_{ir}G_{ri},\\&X_{43}= (cm(E_+)-cm_N)\frac{1}{N^2}\sum  G_{ir}G_{rs}G_{si},&  &X_{44}=\frac{1}{N^3}\sum_{r,s,t}G_{ir}G_{rs}G_{st}G_{ti},\\
&\tX_{44}=\frac{1}{N^3}\sum G_{ir}G_{ri}G_{st}G_{ts}.
\end{aligned}
\end{eqnarray}
Observe that the above $z$-dependent random variables satisfy
\begin{eqnarray*}
X_{22}=O(\Psi^2),\quad  X_{32}, X_{33}= O(\Psi^3), \quad X_{42}, X_{43}, X_{44}, \tX_{44}=O(\Psi^4).
\end{eqnarray*}
Some variations of above random variables are given by
\begin{eqnarray}\label{Xset2}\begin{aligned}
&X_{22}'=\frac{1}{N}\sum_{\alpha} G_{i\alpha}G_{\alpha i},& \quad & X_{32}'=(cm(E_+)-cm_N)\frac{1}{N}\sum_{\alpha} G_{i\alpha}G_{\alpha i},\\ &X_{33}'=\frac{1}{N^2}\sum G_{i\alpha}G_{\alpha k}G_{ki},&   &X_{33}''=\frac{1}{N^2}\sum G_{i\alpha}G_{\alpha\beta}G_{\beta i}, \\
&X_{42}'=(cm(E_+)-cm_N)^2\frac{1}{N}\sum G_{i\alpha}G_{\alpha i}, &  &X_{43}'= (cm(E_+)-cm_N)\frac{1}{N^2}\sum G_{i\alpha}G_{\alpha s}G_{si}\\
&X_{43}''=  (cm(E_+)-cm_N)  \frac{1}{N^2}\sum G_{i\alpha}G_{\alpha \beta}G_{\beta i},&  &X_{44}'=\frac{1}{N^3}\sum_{\alpha,s,t}G_{i\alpha}G_{\alpha s}G_{st}G_{ti},\\
&X_{44}''=\frac{1}{N^3}\sum_{\alpha,\beta,t}G_{i\alpha}G_{\alpha \beta}G_{\beta t}G_{ti},&  &X_{44}'''=\frac{1}{N^3}\sum_{\alpha,\beta,\mu}G_{i\alpha}G_{\alpha\beta}G_{\beta \mu}G_{\mu i}
\end{aligned}\end{eqnarray}

\begin{Lemma}\label{E-derivofG-formu}
Let  \begin{eqnarray}\label{defmathcalX}\begin{aligned}&\mX_3 = \sum_i ( X_{32}-X_{33}), \quad \mathfrak{X}_{4}=(cm-cm_N) +(E_+-z) \sum_i  X_{22},\\  &\mX_4=  \sum_i (3 X_{42}- 6X_{43}+12X_{44}+3 \tX_{44})
\end{aligned}\end{eqnarray}
Then we have
\begin{eqnarray}\E \frac{1}{N} \sum_i \frac{\partial G_{ii}}{\partial t}= C_0+ C_1 \E \mX_{3}+ C_2 \E \mathfrak{X}_{4} + C_3 \E \mX_4+O(N^{1/2}\Psi^3)
\end{eqnarray}
where $C_0, C_1, C_2$ and $C_3$  are explicitly given by
\begin{eqnarray}\begin{aligned}
&C_0=(\frac{h}{\gamma E_+}-\frac{1}{\gamma})(b-\frac{\dot{\gamma}}{\gamma}), \quad C_1 =\frac{2\tb b}{\gamma E_+}-2\gamma E_+-\frac{2\tb\dot{\gamma}}{\gamma^2 E_+},\\
&C_2 =\frac{2b^2(b\tb-\gamma^2 E_+^2)}{\gamma^2 E_+ h}-\gamma cm -\frac{\gamma^3(1-c)}{b^3}-\frac{2b(b\tb-\gamma^2 E_+^2)}{\gamma^3 E_+ h}\dot{\gamma}+\frac{\gamma^2(1-c)}{b^3}\dot{\gamma},\\
&C_3=\left(\frac{\gamma^3 h^4 \tb \varphi_4}{E_+} -\frac{3\tb}{h}-\frac{3\gamma^2 \tb}{b^3}-\frac{3\gamma^2 E_+\tb}{b^2 h}-\frac{E_+  b^2+\gamma^2 E_+^2 }{b h}+ \frac{h \gamma^3(1-c)}{E_+ b^4}-\frac{\gamma^3(1-c)}{b^4} \right)(b-\frac{\dot{\gamma}}{\gamma})\\&\relphantom{EE}+\dot{\gamma}\frac{\gamma^3 cm(1-c)}{b^4},
\end{aligned}
\end{eqnarray}
and where $\varphi_4$ is defined in \eqref{notatvarPhi}.
\end{Lemma}
\begin{proof}
 The first step is to rewrite each term in \eqref{derigii} as linear combination of $\E X_{22}, \E X_{32},\E X_{33}, \E (cm-cm_N), (E_+-z) \E X_{22}, \E X_{42}, \E X_{43}, \E X_{44}, \E \tX_{44}$. Although there exist linear relations within these terms, for example, see \eqref{optfrommnx22}, we do not pursue the cancellation between them in this step. Notice that the first term is already expressed by $\E X_{22}$. By a simple fact that $\sum_j \E G_{i\bbj}G_{\bbj i}=N\E X_{22}'-\sum_{\alpha>2M} E G_{i\alpha}G_{\alpha i} $, the second term can be rewritten  according to \eqref{X22'2} and \eqref{X22'vsimple} below.  Using \eqref{X33l2Mv} and \eqref{optfromX33} below,  we find the third term is \begin{eqnarray*}\begin{aligned}\frac{\gamma}{N} \sum _{j,p}2\E G_{ip}G_{p\bbj}G_{\bbj i}&=2 \gamma N\E X_{33}'-\frac{2\gamma}{N}\sum _{\alpha>2M, p}\E G_{ip}G_{p\alpha}G_{\alpha i}=2\gamma (E_+ -\frac{\gamma(1-c)}{b^2}) N\E X_{33}. \end{aligned}\end{eqnarray*}
To handle the fourth term, observe that
 \begin{eqnarray}\frac{\gamma}{N}  \sum _{j,p} \E G_{i\bbj}G_{pp}G_{\bbj i}=\gamma cm \sum_j \E G_{i\bbj}G_{\bbj i}- \gamma \sum_{j} \E(cm-cm_N)G_{i\bbj}G_{\bbj i}. \end{eqnarray}
Since we have  \begin{eqnarray}\begin{aligned}
&\frac{1}{N}\sum_{j} \E(cm-cm_N)G_{i\bbj}G_{\bbj i}\\&=\E X_{32}'-\frac{1}{N}\sum_{\alpha>2M} \E (cm-cm_N)G_{i\alpha}G_{\alpha i}\\&=(E_+-\frac{\gamma(1-c)}{b^2})\E X_{32}+\frac{b}{\gamma d_i^2-\xi}\E \frac{1}{N}(cm-cm_N) + \frac{\gamma^2(1-c)}{b^3}\E (-2X_{42}+2X_{43}-2X_{44})+O(\Psi^5),
\end{aligned}\end{eqnarray}
where in the second step we use \eqref{x32'l2Mv} and \eqref{x32'vsimple}, the fourth term can be transformed.
Next we deal with the fifth term. We have the following relation that can be verified without difficulty, \begin{eqnarray}\label{relofmtm}
E_+\tm(E_+)-\frac{1}{N}\sum G_{\mu\mu}=E_+(cm(E_+)-\frac{1}{N}\sum G_{jj})+(E_+-z)cm(E_+) + O(\Psi^3).
\end{eqnarray}
Using this we find
\begin{eqnarray}\begin{aligned}
\frac{\gamma}{N^2}\sum_{p,j} \E G_{ip}G_{\bbj\bbj} G_{pi}&=-\gamma E_+ \E X_{32}-\gamma(E_+ -z)cm(E_+)\E X_{22}
-\frac{\gamma}{N^2}\sum_{\mu>2M} G_{ip}G_{\mu\mu}G_{pi}+\gamma E_+\tm(E_+)\E X_{22}\\
&=-\gamma E_+ \E X_{32}-\gamma(E_+ -z)cm(E_+)\E X_{22}
+\gamma E_+\tm(E_+)\E X_{22}\\
&+\frac{\gamma(1-c)}{b}\E X_{22}+\frac{\gamma^2(1-c)}{b^2}\E X_{32}+\frac{\gamma^3(1-c)}{b^3}\E (X_{42}+4 X_{44}+\tX_{44})+O(\Psi^5),
\end{aligned}\end{eqnarray}
where in the second equality we use  \eqref{x32'l2Mv}.
The last four terms can be transformed in a similar manner to above terms, so we do not provide details here.% \begin{eqnarray}\begin{aligned}\sum_{j} \E(cm-cm_N)G_{i\bbj}G_{\bbj i}&=\gamma \varphi_2 \E X_{32}' + \gamma \tb^2 \varphi_1 \E X_{32} + (\Phi_2+\frac{\gamma^2(1-c)}{b^3})\E(-2X_{42}+2X_{43}-2X_{44})\\&\relphantom{EEE}+\E(cm-cm_N)G_{i\bbi}^2\\&=(E_+ -\frac{\gamma(1-c)}{b^2})\E X_{32}+\frac{\gamma \varphi_2 b}{\gamma d_i^2-\xi}\E(cm-cm_N)+ (\Phi_2+\frac{\gamma^2(1-c)}{b^3})\\ &\relphantom{EE}\times\E (-2X_{42}+2X_{43}-2X_{44})+\frac{\gamma d_i^2}{\gamma d_i^2-\xi}\E (cm-cm_N)+O(N \Psi^5).\end{aligned}\end{eqnarray}

In the next step, using above expressions of terms in \eqref{derigii}, summing over index $i$ and dividing by $N$, we summarize the following. \\
(1). The terms do not contain expectations of Green entries are $(\frac{h}{\gamma E_+}-\frac{1}{\gamma})(b-\frac{\dot{\gamma}}{\gamma})$.\\
(2). With the help of \eqref{derivofE} and \eqref{R1g},  the coefficient of $\sum_i \E X_{22}$ is 0.\\
(3). With the help of \eqref{X22'2} and \eqref{X22'vsimple}, both the coefficients of $\sum_i \E X_{32}$ and $-\sum_i \E X_{33}$ are $$\frac{2\tb b}{\gamma E_+}-2\gamma E_+-\frac{2\tb\dot{\gamma}}{\gamma^2 E_+}.$$
(4). With the help of \eqref{optfromx32} and \eqref{optfrommnx22},
both the coefficients of $\E (cm-cm_N) $ and $(E_+-z)\sum_i \E X_{22}$ are
%\begin{eqnarray}\frac{1}{N}\sum_a \frac{-2\gamma^2 d_a^2 b^2(E_+ b+\tb)^2}{(\gamma d_a^2-\xi)^3 E_+}-\gamma cm-\frac{\gamma^3(1-c)}{b^3}-\gamma^{-1}\dot{\gamma}\frac{1}{N}\sum_a \frac{-2\gamma ^2 d_a^2 b (E_+ b+\tb)^2}{(\gamma d_a^2-\xi)^3 E_+}+\dot{\gamma} \frac{\gamma^2(1-c)}{b^3}\end{eqnarray}
\begin{eqnarray}
-\frac{2b^2(\gamma^2 E_+^2-b\tb)}{\gamma^2 E_+ h}-\gamma cm-\frac{\gamma^3(1-c)}{b^3}+\frac{2b(\gamma^2 E_+^2-b\tb)}{\gamma^3 E_+ h}\dot{\gamma}+\dot{\gamma} \frac{\gamma^2(1-c)}{b^3}
\end{eqnarray}
(5). With the help of \eqref{X22'2} and \eqref{X22'vsimple}, the coefficient of  $ \sum_i \E(3 X_{42}- 6X_{43}+12X_{44}+3 \tX_{44})$ is \begin{eqnarray}\begin{aligned}
&\left\{\frac{\gamma^3 h^4 \tb \varphi_4}{E_+} -\frac{3\tb}{h}-\frac{3\gamma^2 \tb}{b^3}-\frac{3\gamma^2 E_+\tb}{b^2 h}-\frac{E_+  b^2+\gamma^2 E_+^2 }{b h}+ \frac{h \gamma^3(1-c)}{E_+ b^4} -\frac{\gamma^3(1-c)}{b^4}\right\}(b-\frac{\dot{\gamma}}{\gamma})\\&+\dot{\gamma}cm \frac{\gamma^3(1-c)}{b^4}.
\end{aligned}\end{eqnarray}
We conclude this lemma from above.
\end{proof}

\begin{Lemma}\label{imcancel}
\begin{eqnarray}\label{cancelrule}\Im  \left(C_1 \E \mX_{3}+ C_2 \E \mathfrak{X}_{4} + C_3 \E \mX_4 \right)=O(N^{1/2}\Psi^3)\end{eqnarray}
\end{Lemma}
\begin{proof}
Subtracting ($C_1 \gamma /2$)-times imaginary part of \eqref{optfromX22} from \eqref{cancelrule}, we find the coefficient of $\Im \E \mX_3$ cancels, the coefficient of $\Im \E \mathfrak{X}_4$ is \begin{eqnarray}\label{9b3k8}-\gamma cm -\frac{\gamma^3(1-c)}{b^3}+\frac{2b E_+}{\gamma h}\dot{\gamma}+\frac{\gamma^2(1-c)}{b^3}\dot{\gamma},\end{eqnarray}
and the coefficient of $\Im \E \mX_4$ is $C_3-\frac{C_1 \gamma}{2}\theta_4$, which is explicitly given by \begin{eqnarray}\label{9bn62}\begin{aligned}
&\left(\frac{\tb}{h}+\frac{\gamma^2 \tb}{b^3}+\frac{\gamma^2 E_+ \tb}{b^2 h}-\frac{E_+ b}{h}-\frac{\gamma^2 E_+^2}{b h}\right)(b-\frac{\dot{\gamma}}{\gamma})-\frac{\gamma^2 E_+ \tb}{b h}+\frac{\gamma^3 h (1-c)}{E_+ b^3}-\frac{\gamma^3 (1-c)}{b^3}\\
&\relphantom{EE} +\gamma^5 h^4 E_+ \varphi_4- \frac{4\gamma^2 E_+^2}{h}-\frac{4\gamma^4 E_+^2}{b^3}-\frac{3 \gamma^4 E_+^3}{b^2 h}+\frac{\gamma E_+ \tb}{b^2 h}\dot{\gamma}.
\end{aligned}\end{eqnarray}
From  \eqref{optfromx32} and \eqref{optfrommnx22}, we find \begin{eqnarray*}\E \mathfrak{X}_4= -\E \mX_4  +\sum_i( - 4\E X_{43}+6\E X_{44}+2\E \tX_{44})\end{eqnarray*}
It follows from \eqref{optfromX33} that \begin{eqnarray} \E \mathfrak{X}_4 =   -\E \mX_4  +O(N^{1/2}\Psi^3).\end{eqnarray}
Now we find that to prove \eqref{cancelrule}, it suffices to show that \begin{eqnarray}-\eqref{9b3k8}+\eqref{9bn62}=0.\end{eqnarray}
Replacing $\gamma^5 h^4 E_+ \varphi_4$ in \eqref{9bn62} with $$\gamma^5 h^4 E_+ \varphi_4 = \frac{\dot{\gamma}}{\gamma}+\gamma^3 h^3\left(\frac{4E_+^2}{\gamma h^4}+\frac{2\gamma E_+^3}{b^2 h^4}+\frac{2\gamma E_+^2}{b^3 h^3}-\frac{2 b\tb}{\gamma^3 h^4}+\frac{\gamma E_+}{h^2 b^4}+\frac{1}{\gamma^3 h^3} \right)$$ which is obtained from \eqref{deriofgamma}, we find  the coefficient of $\dot{\gamma}$ vanishes, and the other terms also cancel. So we conclude this lemma.
\end{proof}

\begin{Lemma}
For the $b$ and $\gamma $ defined in ..., we have
\begin{eqnarray}\label{deriofb}
\dot{b}=\gamma^2 E_+ +b(b-1)
\end{eqnarray}
\begin{eqnarray}\label{deriofgamma}
\dot{\gamma}=\gamma^6 h^4 E_+ \varphi_4-\gamma^4 h^3\left(\frac{4E_+^2}{\gamma h^4}+\frac{2\gamma E_+^3}{b^2 h^4}+\frac{2\gamma E_+^2}{b^3 h^3}-\frac{2 b\tb}{\gamma^3 h^4}+\frac{\gamma E_+}{h^2 b^4}+\frac{1}{\gamma^3 h^3} \right)
\end{eqnarray}
\end{Lemma}
%Denote \begin{eqnarray}\varphi_0:=\frac{1}{N}\sum \frac{1}{\gamma d_i^2-\xi}\end{eqnarray}

\begin{proof}
Recalling the definition of $\hf$ and $E_+$ in \eqref{Linkfunc2} and the argument below it, we have  \begin{eqnarray}E_+ = \xi (1-c\gamma \hf(\xi))^2+\gamma(1-c)(1-c\gamma \hf(\xi)) .\end{eqnarray}
Then using the fact that $1-c\gamma \hf(\xi) = 1/b$ which can be deduced from \eqref{defm}, \eqref{rescb} and \eqref{rescxi}, we find
\begin{eqnarray}\label{320ag}
\dot{z}= \dot{E}_+= \dot{\xi}\frac{1}{b^2}-\frac{2\xi (\gamma c\dot{\hf}+c \hf \dot{\gamma})}{b}+\dot{\gamma}\frac{1-c}{b}-\gamma (1-c)c \gamma \dot{\hf}-\gamma (1-c) c \hf\dot{\gamma}.
\end{eqnarray}
Actually we can calculate \begin{eqnarray}\label{n6h7a}
c\dot{\hf}(\xi) = -\frac{\dot{\gamma}\psi_2}{\gamma}+\psi_2+\varphi_2 \dot{\xi}
\end{eqnarray}
Plugging \eqref{n6h7a} back into \eqref{320ag}, we find that the coefficient for $\dot{\xi}$ is 0, and get \begin{eqnarray}\label{derivofE}
\dot{z}=\dot{E}_+=-\gamma (E_+b+\tb)\psi_2+\frac{E_+}{\gamma}\dot{\gamma}=-(E_+b+\tb) + E_+ +\frac{E_+}{\gamma}\dot{\gamma}.
\end{eqnarray}
Then we compute the derivative of $\xi$  and $h$ for use later, which are\begin{eqnarray}\label{deriofxi}
\dot{\xi}=\frac{b\tb}{\gamma}\dot{\gamma}+h\dot{b}+E_+ b^2-h b^2,
\end{eqnarray}
and \begin{eqnarray}\label{deriofh}
\dot{h}=2 \dot{E}_+ b+2 E_+\dot{b}-\dot{\gamma}(1-c)=2 E_+ b-2 h b +2 E_+ \dot{b}+\frac{h}{\gamma}\dot{\gamma}.
\end{eqnarray}

We differentiate the equation $\varphi_2= 1/\gamma b^2 h$ in \eqref{relofvarphi} with respect to $t$ to get \begin{eqnarray}
\frac{2}{N}\sum_{i=1}\frac{\dot{\gamma} d_i^2-\gamma d_i^2 -\dot{\xi}}{(\gamma d_i^2-\xi)^3} = \frac{\dot{\gamma}}{\gamma^2 b^2 h}+\frac{2\dot{b}}{\gamma b^3 h}+\frac{\dot{h}}{\gamma b^2 h^2}.
\end{eqnarray}
Plugging \eqref{deriofxi} and \eqref{deriofh} into  the above and we find the coefficient of $\dot{\gamma}$ cancels, and we can obtain that
\begin{eqnarray}\label{deriofbinproof}
\dot{b}=\gamma^2 E_+ +b(b-1)
\end{eqnarray}
Similarly, by differentiating the expression of $\varphi_3$ in \eqref{relofvarphi} with respect to $t$, we find
\begin{eqnarray}\begin{aligned}
&\frac{3}{N}\sum_i \frac{\dot{\gamma} d_i^2-\gamma d_i^2 -\dot{\xi}}{(\gamma d_i^2-\xi)^4}+\frac{3\dot{\gamma}}{\gamma^4 h^3}+\frac{3\dot{h}}{\gamma^3 h^4}+\frac{\dot{\gamma}}{\gamma^2 b^3 h^2}+\frac{3\dot{b}}{\gamma b^4 h^2}\\
&\relphantom{EEE}+\frac{2 \dot{h}}{\gamma b^3 h^3}+\frac{E_+\dot{\gamma}}{\gamma^2 b^2 h^3}+\frac{2 E_+ \dot{b}}{\gamma b^3 h^3}+\frac{3 E_+ \dot{h}}{\gamma b^2 h^4}-\frac{\dot{E}_+}{\gamma b^2 h^3}=0.
\end{aligned}\end{eqnarray}
Substituting \eqref{derivofE},\eqref{deriofxi},\eqref{deriofh} and \eqref{deriofbinproof} into above, we obtain the derivative of $\dot{\gamma}$ given by \eqref{deriofgamma} through tedious calculations.

%After dividing $\gamma$ on both sides of \eqref{R1g}, and taking derivative with respect to $t$, we get \begin{eqnarray}
%\gamma^{-2}\dot{\gamma}+\frac{1}{N}\sum_{i} \frac{\dot{E}_+ b^2 +2E_+ b \dot{b}+\dot{\gamma}d_i^2-\gamma d_i^2}{(\gamma d_i^2-\xi)^2}-\frac{2}{N}\sum_{i}\frac{(E_+ b^2+\gamma d_i^2) (\dot{\gamma}d_i^2-\gamma d_i^2-\dot{\xi})}{(\gamma d_i^2-\xi)^3}=0.
%\end{eqnarray}
%Substituting \eqref{derivofE} and \eqref{deriofxi} into above equation, and merging terms with respect to $\dot{b}$ and $\dot{\gamma}$, we find
\end{proof}

\section{Proof of Theorem \ref{Lolaw} and Proposition \ref{crare}}
\subsection{Weak local law}
In this subsection, we derive a weak local law, see Theorem \ref{Wkloc} below, served as the first step for the proof of Theorem \ref{Lolaw}.
\begin{Th}\label{Wkloc}Under the Assumptions 1,2,3, for any sufficiently small $\epsilon>0$ and large $D>0$, we have  \begin{eqnarray}
\bigcap_{z\in \mD(\tc,\tau)}\Big\{
|s_N(z)-s(z)|+ \max_{s\neq t, (s,t)  \text{is not index pair}}|G_{st}(z)| \leq N^\epsilon(\frac{1}{N\eta})^{1/4}\Big\}
\end{eqnarray}holds with probability $1-O(N^{-D})$.
\subsubsection{Basic tools}
%\begin{Lemma}\label{GFI}(Green function identities)\\
%(i) For any $i\in \mI_M$,and $\mu \in \mI_N$, we have \begin{eqnarray}\label{gfi1}\begin{aligned}
%&G_{ii}=\frac{1}{-z-(YG^{(i)}Y^*)_{ii}},\\
%&G_{\mu\mu}=\frac{1}{-1-(Y^*G^{(\mu)}Y)_{\mu\mu}}.\end{aligned}\end{eqnarray}
%(ii)For $i\neq j \in \mI_M$\begin{eqnarray}\label{gfi2}G_{ij}=-G_{ii}(YG^{(i)})_{ij}=-G_{jj}(G^{(j)}Y^*)_{ij}=G_{ii}G_{jj}^{(i)}(YG^{(ij)}Y^*)_{ij}.\end{eqnarray}
%For $\mu\neq \nu \in \mI_N$\begin{eqnarray}\label{gfi3}G_{\mu\nu}=-G_{\mu\mu}(Y^*G^{(\mu)})_{\mu\nu}=-G_{\nu\nu}(G^{(\nu)}Y)_{\mu\nu}=G_{\mu\mu}G^{(\mu)}_{\nu\nu}(Y^*G^{(\mu\nu)}Y)_{\mu\nu}.\end{eqnarray}
%(iii) For $i\in \mI_M$ and $\mu \in \mI_N$, we have \begin{eqnarray}\label{gfi4}\begin{aligned}&G_{i\mu}=-G_{\mu\mu}(G^{(\mu)}Y)_{i\mu}=-G_{ii}(YG^{(i)})_{i\mu},\\&G_{\mu i}=-G_{\mu\mu}(Y^*G^{(\mu)})_{\mu i}  =-G_{ii}(G^{(i)}Y^*)_{\mu i}.\end{aligned}\end{eqnarray}
%(iv) For $r\in \mI$ and $s,t \in  \mI \backslash \{k\}$, we have \begin{eqnarray}\label{gfi5}G_{st}^{(r)}=G_{st}-\frac{G_{sr}G_{rt}}{G_{rr}}.\end{eqnarray}(v) All of the identities above hold if we replace $G$ by $G^{(T)}$.\end{Lemma}

\begin{Lemma}\label{sqrot}
Under Assumptions 1,2,3, there exists a small constant $\delta>0$, independent of $N$, such that $$\rho_0(x)\sim \sqrt{\lambda_r-x}, \; \mbox{for any} \; x \in [\lambda_r-\delta,\lambda_r] .$$
Furthermore,\begin{eqnarray}\Im s(z)\sim\bigg\{
\begin{array}{ll}{\frac{\eta}{\sqrt{\kappa+\eta}},  } &{\text{if}\; E\geq \lambda_r+\eta } \\{ \sqrt{\kappa+\eta},}  &{\text{if} \; E \in [\lambda_r-\delta, \lambda_r+\eta]}\end{array}
 \end{eqnarray}
\end{Lemma}

\begin{Lemma}\label{db432}
There exists sufficiently small  constant $\tc > 0$ independent of $N$, such that for $ z \in \mD(\tc,\tau)$, \begin{eqnarray}\label{delb1}\begin{aligned}&|1+cs|\geq c_0, \quad  |1+\ts| \geq c_0\\& \text{and} \quad
|-z+\frac{d_i^2}{1+cs}-z\ts| \geq c_0, \quad \text{for any} \; i=1,\cdots,M \\
\end{aligned}
\end{eqnarray}
where $c_0$ is a positive constant.
\end{Lemma}
\begin{proof}
In the Appendix A in \cite{loubaton2011almost}, it has been proved that \begin{eqnarray}\label{lb1cm}\Re(1+cs) \geq 1/2\end{eqnarray} for $z \in \mathbb{C}^+ \cup \mathbb{R}^+.$ Therefore, $|1+cs| \geq 1/2.$ The second inequality in \eqref{delb1} can be shown by similar method for the last one, so we mainly focus on the last inequality in this lemma.

Recall $w$ defined in \eqref{Linkfun}. Notice that \begin{eqnarray}\label{n8ag5}
\frac{d_i^2}{1+cs(\lambda_r)}-\lambda_r(1+cs(\lambda_r))+1-c =
\frac{d_i^2-w(\lambda_r)}{1+cs(\lambda_r)}<-C\hat{c},
\end{eqnarray}
where in the last step, we use Assumption \ref{assump3} and \eqref{lb1cm}.
First we consider the case that $z$ is in a small disk with center $\lambda_r$ and radius $\epsilon$, denoted by $C(\lambda_r,\epsilon)$.
By analytic property of $s$, there exists small constant $\tc$, such that for any $z \in C(\lambda_r,\tc)\cap \mathbb{C}^+$,\begin{eqnarray}\label{4bao4}|s(z)-s(\lambda_r) |\leq C'\hat{c},\end{eqnarray}where the choice of $C'$ and the associated $\tc$ should guarantee the second inequality of the following,
\begin{eqnarray}
\begin{aligned}\label{4bao5}
&\left| |\frac{d_i^2}{1+cs(z)}-z(1+cs(z))+1-c | -|\frac{d_i^2}{1+cs(\lambda_r)}-\lambda_r(1+cs(\lambda_r))+1-c |\right|\\
&\leq (4d_i^2c+\lambda_rc)|s(z)-s(\lambda_r)|+|z-\lambda_r|\sup_{z\in C(\lambda_r,\tc)\cap \mathbb{C}^+}|1+cs(z)|\leq \frac{C\hat{c}}{2}.
\end{aligned}
\end{eqnarray}
Then it follows that \begin{eqnarray}\label{gab21}
 |\frac{d_i^2}{1+cs(z)}-z(1+cs(z))+1-c |>\frac{C\hat{c}}{2} \quad \text{for} \quad z \in  C(\lambda_r,\tc)\cap \mathbb{C}^+.
\end{eqnarray}
Next, if $E >\lambda_r$, $s(x)$ is negative and increasing on $(\lambda_r, 1/\tau)$, with the fact that (\ref{lb1cm}),  then
\begin{eqnarray}\begin{aligned}
\frac{d_i^2}{1+cs(E)}-E(1+cs(E))+1-c &< \frac{d_i^2}{1+cs(\lambda_r)}-\lambda_r(1+cs(\lambda_r))+1-c \\&<-C\hat{c}
\end{aligned}
\end{eqnarray}
By the uniform continuity of $s(z)$ on $\mathcal{D}(\tc,\tau)$, we can find a sufficiently small constant $\eta_0$, such that $$\sup_{E\in [\lambda_r+\tc/2, 1/\tau]}|s(E+i\eta)-s(E)| \leq C''\eta_0 \quad \text{for} \quad \eta \leq \eta_0 $$
 Following similar argument from (\ref{4bao4}) to (\ref{gab21}), we have
 \begin{eqnarray}\label{gab22}\begin{aligned}
  |\frac{d_i^2}{1+cs(z)}-&z(1+cs(z))+1-c |>\frac{C\hat{c}}{2} \\ &\; \text{for} \; z \in  \{ z| E\in  [\lambda_r+\tc/2,1/\tau], \eta \in [0,\eta_0]\}.
\end{aligned} \end{eqnarray}
Last,  since $\Im s >0$, and $\Im (zs)>0$,  \begin{eqnarray}\label{gab23}
 \Im (\frac{d_i^2}{1+cs(z)}-z(1+cs(z))+1-c )\leq -\eta_0 \quad \text{for}\quad z\in \{ z|\eta>\eta_0\}
 \end{eqnarray}
 Combining (\ref{gab21})(\ref{gab22})(\ref{gab23}) together, we conclude the last inequality.
\end{proof}

\subsubsection{A prior estimates on domain $\Omega(z)$}
We start from finding estimates for diagonal entries of G.  Using \eqref{gfidiag}, write
\begin{eqnarray}\label{34jga}
\begin{aligned}
\frac{1}{G_{ii}}&=-z-(\bY G^{(i)}\bY^*)_{ii}\\
&=-z-d_{i}^2G_{\bbi\bbi}^{(i)}-\sum_{\mu,\nu}x_{i\mu}G_{\mu\nu}^{(i)}x_{i\nu}-2 d_i\sum_{\nu}G_{\bbi\nu}^{(i)}x_{i \nu}\\
&=-z-d_{i}^2G_{\bbi\bbi}^{(i)}-z\tilde{s}_N^{(i)}-(1-E_i)\sum_{\mu,\nu}x_{i\mu}G_{\mu\nu}^{(i)}x_{i\nu}-2d_i\sum_{\nu}G_{\bbi\nu}^{(i)} x_{i \nu}\\
&=-z-d_{i}^2G_{\bbi\bbi}^{(i)}-z\ts +z(\ts-\ts_N)+\mT_i ,
\end{aligned}
\end{eqnarray}
where \begin{eqnarray}\label{bc38a}\mT_i:=z(\ts_N-\ts_N^{(i)})-(1-E_i)\sum_{\mu,\nu}x_{i\mu}G_{\mu\nu}^{(i)}x_{i\nu}-2 d_i\sum_{\nu}G_{\bbi\nu}^{(i)}x_{i \nu},\end{eqnarray}
and where $\ts_N^{(T)} := \frac{1}{N}\sum G_{\mu\mu}^{(T)}$.
Furthermore,
\begin{eqnarray}\label{rg943}
\begin{aligned}
\frac{1}{G_{\bbi\bbi}^{(i)}}&=-1-(\bY^{(i)*}G^{(i\mu_i)}\bY^{(i)})_{\bbi\bbi}\\
&=-1-\sum_{j,k}^{(i)}x_{j \bbi}G_{jk}^{(i \bbi)}x_{k \bbi}\\
&=-1-c\frac{1}{M}Tr G^{(i \bbi)}-(1-E_{\bbi})\sum_{j,k}^{(i)}x_{j \bbi}G_{jk}^{(i \bbi)}x_{k \bbi}\\
&=-1-cs+c(s-s_N)+\mS_i,
\end{aligned}
\end{eqnarray}
where \begin{eqnarray}\label{9g8a9}
\mS_i:=c(s_N-s_N^{(i \bbi)})-(1-E_{\bbi})\sum_{j,k}^{(i)}x_{j \bbi} G_{jk}^{(i \bbi)}x_{k \bbi}.
\end{eqnarray}

Define control quantities:
\begin{eqnarray}\begin{aligned}
&\Lambda_o:=\max_{s\neq t, (s,t)  \text{is not index pair}}|G_{st}|, \quad  \Lambda:=|s_N-s|, \quad \Psi_\Lambda:=\sqrt{\frac{\Lambda+\Im s}{N\eta}}+\frac{1}{N\eta}
\end{aligned}
\end{eqnarray}
and event $\Omega(z)$ by
\begin{eqnarray}\label{omey1}
\Omega(z):=\{\Lambda_o\leq N^{-\tau/10}\}\cap\{\Lambda\leq N^{-\tau/10}\}.
\end{eqnarray}

%Under condition 1... for any $\zeta>0$, there exists some constant $C_\zeta$ such that the event \begin{eqnarray}
%\bigcap_{z\in \mathcal{D}}\left\{\max_{i\neq j}|G_{ij}|  \prec \frac{1}{(N\eta)^{1/4}}   \right\}
%\end{eqnarray}
%mholds with $\zeta$-high probability.
\end{Th}

\begin{Lemma}\label{lgii}
For any $t\in \mI$, the following statement holds with high probability on $\Omega(z),$
\begin{eqnarray}\label{diap1}
\quad  C^{-1} \leq |G_{tt}| \leq C
\end{eqnarray}

\end{Lemma}
\begin{proof}
We only prove the case that $t\in \mI_M$. For the case $t \in \mI_N$, the proof is similar.
First, we bound $\mS_i$ defined above. By Lemma \ref{ldb}  we have
\begin{eqnarray}\label{920en}
\begin{aligned}
|(1-\E_{\bbi})\sum_{j,k}^{(i)}x_{j \bbi} G_{jk}^{(i \bbi)}x_{k \bbi}|&\leq \sum_j^{(i)}(|x_{j \bbi}|^2-\frac{1}{N})|G_{jj}^{(i \bbi)}|+|\sum_{j\neq k}^{(i)} x_{j \bbi}^*G_{jk}^{(i \bbi)}x_{k \bbi}|\\
&\prec(\frac{1}{N^2}\sum_{j,k}^{(i)}|G_{jk}^{(i \bbi)}|^2)^{1/2}\leq (\frac{1}{N^2\eta}\sum_{j}^{(i)}\Im G_{jj}^{(i \bbi)})^{1/2}\\
&\leq \Big(\frac{\sum_j^{(i)}\Im (G_{jj}^{(i \bbi)}-G_{jj})+M\Im s_N}{N^2\eta}\Big)^{1/2}\\
&\leq \Big( \frac{\Lambda_o^2}{N\eta}+\frac{c\Lambda+c\Im s}{N\eta}\Big)^{1/2}
\end{aligned}
\end{eqnarray}
where in the last step we use (\ref{gfiminus}). And the above term can also be bounded by $\Psi_\Lambda$ if we use interlacing relationship of eigenvalues of Hermitian matrices in the last step above.
Therefore, on $\Omega$ we have \begin{eqnarray}\label{ew23a}|\mS_i|= o(1),\end{eqnarray} and
\begin{eqnarray} \label{ew32b}
|\frac{1}{G_{\bbi\bbi}^{(i)}}|=|-1-cs|+o(1)
\end{eqnarray}
Moreover we have \begin{eqnarray}\label{sn43j}\Im G_{\bbi\bbi}^{(i)}\prec \Im s +\Lambda+\Psi_\Lambda, \end{eqnarray} which can be deduced from (\ref{rg943}),(\ref{920en}) and (\ref{ew32b}). %We will need this in the Lemma \ref{lemma123}.\\

Then, we consider $1/G_{ii}$. Since
\begin{eqnarray*}
\begin{aligned}
|(1-\E_i)\sum_{\mu,\nu}x_{i\mu}G_{\mu\nu}^{(i)}x_{i \nu}|&\prec
(\frac{1}{N^2}\sum_{\mu,\nu}|G_{\mu\nu}^{(i)}|^2)^{1/2}\prec\frac{1}{N\eta}+\sqrt{\frac{\Im \ts+\Lambda}{N\eta}},
\end{aligned}
\end{eqnarray*}
and\begin{eqnarray}\label{0ab3m}
|\sum_{\nu}G_{\bbi\nu}^{(i)}x_{i \nu}|\prec(\frac{1}{N}\sum_{\nu}|G_{\bbi\nu}^{(i)}|^2)^{1/2}\leq (\frac{\Im G_{\bbi\bbi}^{(i)}}{N\eta})^{1/2}\prec \Psi_\Lambda,
\end{eqnarray}
where we use \eqref{sn43j} in the last step,  we then have\begin{eqnarray}
 |\frac{1}{G_{ii}}|=|-z-d_i^2G_{\bbi\bbi}^{(i)}-z\ts|+o(1)=|-z+\frac{d_i^2}{1+cs}-z\ts|+o(1)
\end{eqnarray}
holds with high probability on $\Omega(z).$ Using Lemma \ref{db432} we conclude that $G_{ii}$ is bounded from both above and below.
\end{proof}

\begin{Lemma}\label{lemma123} For all $z \in \mathcal{D}(\tc,\tau)$, we have for $i\in \mI_M$, \begin{eqnarray}
(1(\Omega)+1(\eta\geq 1))(\Lambda_o+|\mS_i|+|\mT_i|)\prec \Psi_\Lambda
\end{eqnarray}and
\begin{eqnarray}\label{nxw34}
(1(\Omega)+1(\eta\geq 1))\left (|G_{ii}-\frac{b(z)}{d_i^2-w(z)}|+|G_{i \bbi}-\frac{d_i}{d_i^2-w(z)}|+|G_{\bbi\bbi}-\frac{\tb(z)}{d_i^2-w(z)}|  \right) \prec \Psi_\Lambda+\Lambda.
\end{eqnarray}
For $\mu\in [\![2M+1,M+N]\!]$, \begin{eqnarray}\label{nxw35}
(1(\Omega)+1(\eta\geq 1))|G_{\mu\mu}+\frac{1}{b(z)}| \prec \Psi_\Lambda+\Lambda.
\end{eqnarray}
\end{Lemma}
\begin{Rem}
By the definition of $\prec$, the result can be directly extended to maximum of variables over countable indices, such as $\max_i |\mS_i|\prec \Psi_{\Lambda}$, and  $\max_i |G_{ii}-b(z)/(d_i^2-w(z))|\prec \Psi_{\Lambda}+\Lambda$.
\end{Rem}
\begin{proof}
We consider the estimation of quantities in this lemma under the event $\Omega$ first and we omit the indicator function for simplicity.
%From (\ref{920en}), we see that $$|\mS_i|\prec \Psi_\Lambda+\Lambda_o.$$
We start with $G_{ij}$ for $i \neq j \in \mI_M$.  Applying \eqref{gfiij}, we have
\begin{eqnarray}\label{gna02}
\begin{aligned}
G_{ij}&=G_{ii}G_{jj}^{(i)}(YG^{(ij)}Y^*)_{ij}\\
&=G_{ii}G_{jj}^{(i)}\big(d_{i}G_{\bbi \bbj}^{(ij)}d_{j}+d_{i}\sum_{\nu}G_{\bbi\nu}^{(ij)}x_{j \nu}+d_{j}\sum_{\nu}x_{i\nu}G_{\nu \bbj}^{(ij)}+\sum_{\mu,\nu}x_{i\mu}G_{\mu\nu}^{(ij)}x_{j \nu}\big).
\end{aligned}
\end{eqnarray}
For $G_{\bbi \bbj}^{(ij)}$, according to (ii) and (v) of Lemma \ref{gfi} it can be written as
\begin{eqnarray}
\begin{aligned}
G_{\bbi \bbj}^{(ij)}&=G_{\bbi\bbi}^{(ij)}G_{\bbj \bbj}^{(ij \bbi)}(Y^{*(ij)}G^{(ij \bbi \bbj)}Y^{(ij)})_{\bbi \bbj}\\&=G_{\bbi \bbi}^{(ij)}G_{\bbj \bbj}^{(ij \bbi)}\sum_{k,l}^{(ij)}x_{k \bbi}^{(ij)}
G_{kl}^{(ij \bbi \bbj)}x_{l \bbj}^{(ij)}.
\end{aligned}
\end{eqnarray}
By Lemma \ref{ldb}, $|\sum_{k,l}^{(ij)}x_{k \bbi}^{(ij)}
G_{kl}^{(ij \bbi \bbj)}x_{l \bbj}^{(ij)}|\prec \Psi_\Lambda $. We also have that $G_{\bbi\bbi}^{(ij)}$ and $G_{\bbj \bbj}^{(ij \bbi)}$ are of constant order from (\ref{ew32b}) and (\ref{gfiminus}). Then the first term inside the bracket of (\ref{gna02}) is stochastically dominated by $\Psi_\Lambda$. For the second term, $$|\sum_{\nu}G_{\bbi\nu}^{(ij)}x^*_{\nu j}|\prec \sqrt{\frac{\Im G_{\bbi\bbi}^{(ij)}}{N\eta}}\prec \Psi_\Lambda+\sqrt{\frac{\Lambda_o^2}{N\eta}},$$
where in the last step we use (\ref{sn43j}) and  (\ref{gfiminus}). The last term in \eqref{gna02} is also bounded by $\Psi_\Lambda$.
Therefore, $|G_{ij}|\prec\Psi_\Lambda +(N\eta)^{-1/2}\Lambda_o\prec \Psi_\Lambda+N^{-\tau/4}\Lambda_o.$ The proof of $|G_{\mu\nu}|\prec \Psi_\Lambda + N^{-\tau/4}\Lambda_o$ for $\mu\neq \nu \in \mI_N$ is similar. \\

%To bound $G_{i\mu}$, by (\ref{gfi4}), we have \begin{eqnarray}
%G_{i\mu}=-G_{ii}(d_{i\mu_i} G_{\mu_i\mu}^{(i)}+\sum_{\nu}x_{i\nu}G_{\nu\mu}^{(i)}).
%\end{eqnarray}
%If $\mu\neq \mu_i$, by (\ref{gfi5}), $G_{\mu_i \mu}^{(i)}\prec \Psi_\Lambda$. The summation above is dominated by $\Psi_\Lambda$ by similar argument to (\ref{0ab3m}). Thus $G_{i\mu} \prec \Psi_\Lambda$ if $(i,\mu)$  is not paired index. By similar argument, the conclusion also holds true for $G_{\mu i}$.\\
To bound $G_{\mu i}$ for $\mu \in \mI_N$ and $i\in \mI_M$, by (\ref{gfiimu}) we have \begin{eqnarray}\label{ofda1}
G_{\mu i}=-G_{ii}(d_{i}G_{\mu \bbi}^{(i)}+\sum_\nu G_{\mu\nu}^{(i)}x_{i \nu}).\end{eqnarray}
Using \eqref{gfiminus} and \eqref{gfiimu} we find \begin{eqnarray}\label{ofda2}
G_{\mu \bbi}^{(i)} = G_{\mu \bbi}-\frac{G_{\mu i}G_{i \bbi}}{G_{ii}},
\end{eqnarray}
\begin{eqnarray}\label{ofda3}
G_{i \bbi}=-G_{ii}(d_{i} G_{\bbi\bbi}^{(i)}+\sum _\nu x_{i\nu}G_{\nu \bbi}^{(i)}).
\end{eqnarray}
Substituting (\ref{ofda3}) into (\ref{ofda2}), and then together with (\ref{ofda1}), we get \begin{eqnarray}\label{ofda4}
(1+d_{i}^2G_{ii}G_{\bbi\bbi}^{(i)} + d_{i}G_{ii}\sum_\nu x_{i\nu}G_{\nu \bbi}^{(i)})G_{\mu i}=-d_{i}G_{ii}G_{\mu \bbi}-G_{ii}\sum_{\nu}G_{\mu\nu}^{(i)}x_{\nu i}^*.
\end{eqnarray}
By equation (\ref{34jga}),
\begin{eqnarray}
1+d_{i}^2G_{ii}G_{\bbi\bbi}^{(i)}=-z(1+\ts)G_{ii}+z(\ts-\ts_N)G_{ii} + G_{ii}\mT_{i}.
\end{eqnarray}
Therefore, if $\mu \neq \bbi$, the absolute value of coefficient of $G_{\mu i}$ in (\ref{ofda4}) is bounded from below by a constant based on (\ref{delb1}), (\ref{omey1}), (\ref{diap1}) and similar argument to (\ref{0ab3m}). The right hand side in (\ref{ofda4}) is stochastically controlled by $\Psi_\Lambda$. Then we get that $|G_{\mu i}|\prec \Psi_\Lambda.$ Similarly, it can be shown that $|G_{i\mu}|\prec \Psi_\Lambda,$ for $\mu\neq \bbi$.  We conclude that $$\Lambda_o\prec \Psi_\Lambda.$$
$|\mS_i|\prec \Psi_\Lambda $ has been shown in Lemma \ref{lgii}, and for $|\mT_i|$, the proof is similar.

What remains is the estimate \eqref{nxw34} and \eqref{nxw35}. Recall $w$  and $\tb$ defined in \eqref{Linkfun} and \eqref{deftb890}. Using the expansion \eqref{seco1} below  it follows easily that
 $|G_{ii}-\frac{b}{d_i^2-w}|\prec\Psi_{\Lambda}+\Lambda$. By similar strategy, we can control $G_{\bbi\bbi}$ and $G_{\mu\mu}$. To estimate $G_{i\bbi}$, we use \eqref{ofda3}. Since $|G_{ii}-\frac{b}{d_i^2-w}|\prec \Psi_{\Lambda}+\Lambda$ and $|G_{\bbi\bbi}^{(i)}-\frac{1}{-b}|\prec\Psi_{\Lambda}$ where the latter is inferred from \eqref{rg943}, together with \eqref{0ab3m},  we can get \begin{eqnarray*}
\left|G_{i \bbi}- \frac{d_i}{d_i^2-w}\right|\prec \Psi_\Lambda+\Lambda.
\end{eqnarray*}

The proof under the event $1(\eta\geq 1)$ can be proceeded similar, but in some steps we use trivial bound  $\|G^{(T)}\|\leq 1$. We omit these steps and conclude the proof.
\end{proof}

Denote $[v]=\frac{1}{M}\sum_i G_{ii} - s$. We now put $[v]$ into a form that we may control it, see Lemma \ref{queq1} for a rough control that is sufficient for the proof of Theorem \ref{Wkloc}.  Recall $ b(z)= 1+cs(z)$ and for simplicity we omit the dependence on $z$. Let $$t_i :=(-z+\frac{d_i^2}{b}-z\ts)^{-1}.$$ Define \begin{eqnarray}\label{defR23}R_2(z):=\frac{1}{M}\sum_{i=1}^M c \left(\frac{d_i^2}{b^2}+z\right)t_i^2, \quad  R_3(z):=\frac{1}{M}\sum_{i=1}^M\left( -c^2\frac{d_i^2}{b^3}t_i^2+(\frac{d_i^2 c}{b^2}+cz)^2 t_i^3 \right)\end{eqnarray}
\begin{Lemma}\label{queq1}
We have on $\Omega(z)$ or $1(\eta\geq 1)$, with high probability,$$|(1-R_2)[v]-R_3[v]^2| \leq C(\max_i\{{|\mT_i|+|\mS_i|+|\Lambda|^3+|\mS_i||\Lambda|}\})$$
\end{Lemma}
\begin{proof}
Using (\ref{34jga}), (\ref{rg943}) and trivial equality $$\frac{1}{x-y}=\frac{1}{x}+\frac{y}{x^2}+\frac{y^2}{x^3}+\frac{y^3}{(x-y)x^3},$$
we have that \begin{eqnarray}\begin{aligned}
\frac{1}{G_{ii}}&=-z+\frac{d_i^2}{b-\mS_i+c(s-s_N)}-z\ts_N+\mT_i\\
&=-z+\frac{d_i^2}{b}-z\ts+\frac{d_i^2(\mS_i+c(s-s_N))}{b^2}+
\frac{d_i^2(\mS_{i}+c(s-s_N))^2}{b^3}\\ &\relphantom{EEE}+\frac{d_i^2(\mS_{i}+c(s-s_N))^3}{b^3[(b-\mS_i -c(s-s_N)]}+z(\ts-\ts_N)+\mT_i.
\end{aligned}\end{eqnarray}
Denote $$D_i=\frac{d_i^2(\mS_i+c(s-s_N))}{b^2}+
\frac{d_i^2(\mS_{i}+c(s-s_N))^2}{b^3}+\frac{d_i^2(\mS_{i}+c(s-s_N))^3}{b^3[(b-\mS_i -c(s-s_N)]}$$
Then it follows that \begin{eqnarray}\label{seco1}
\begin{aligned}
G_{ii}&=t_i-
\left[z(\ts-\ts_N)+D_i+\mT_i\right] t_i^2+[z(\ts-\ts_N)+D_i+\mT_i]^2 t_i^3+G_{ii}^{-1}[z(\ts-\ts_N)+D_i+\mT_i]^3 t_i^3
\end{aligned}
\end{eqnarray}
Since $|-z+\frac{d_i^2}{b}-z\ts|$ is bounded from below known from Lemma \ref{db432}, we can derive a second order equation of $[v]$ from above equation by averaging over $G_{ii}, i =1 \cdots, M$. Recall that $|[v]|=\Lambda$, and  use lemma \ref{lemma123},   we have following with high probability:
\begin{eqnarray}
|(1-R_2)[v]-R_3[v]^2| \leq O(\max_i\{|\mT_i|+|\mS_i|+|\Lambda|^3+|\mS_i||\Lambda|\})
\end{eqnarray}
%Young's inequality imply $$2|\mS_i||\Lambda| \leq \frac{\Lambda^2}{\log N}  + \log N |\mY_i|^2.$$
%Therefore, we can conclude the proof from Lemma  \ref{lemma123}.
\end{proof}

\begin{Lemma}\label{coebd}
Let $\alpha(z) \equiv \alpha:=|1-R_2|$, we have following estimates for $R_2$ and $R_3$ in \eqref{defR23}:
\begin{enumerate}
\item $\alpha(z)\sim \sqrt{\kappa+\eta}$ \text{for} $z \in \mathcal{D}(\tc,\tau).$
\item $|R_3(z)|$ is uniformly bounded above for $z\in \mathcal{D}(\tc,\tau)$. And there exist small constants $c'$ and $\epsilon_0$, such that  $|R_3(z)|\geq c'$ whenever $z\in \mathcal{D}(\tc,\tau)$ satisfies $|z-\lambda_r| < \epsilon_0$.
\end{enumerate}
\end{Lemma}

\begin{proof}
We first introduce some notations for simplicity of presentation:
\begin{eqnarray}
\begin{aligned}
&g(b):=\frac{1}{M}\sum_{i=1}^M\frac{cd_i^2| t_i|^2}{|b|^2}
&G(b):=\frac{1}{M}\sum_{i=1}^M c| t_i|^2
\end{aligned}
\end{eqnarray}
Both $g(b)$ and $G(b)$ are bounded from above and below, which can be inferred from  Lemma \ref{db432}.

For statement 1,  we first consider upper bound for $\alpha(z)$, we first show that for $z \in \mD(\tc,\tau)$, \begin{eqnarray}\label{geb49}
|1- g(b)-zG(b)| = O(\sqrt{\kappa+\eta}).
\end{eqnarray}
Using \eqref{lslaw} where we replace $\rho$ by $\hat{\rho}$, we obtain $$c\Im s=\ g(b)\Im b+G(b) \Im(zb),$$ then it follows that \begin{eqnarray}\begin{aligned}|1- g(b)-zG(b)| &=|1-\frac{c\Im s-G(b)\Im(zb)}{\Im b}-zG(b)|=|\frac{G(b)\Im(zb)}{\Im b}-zG(b)|\\&=|\frac{(E\Im b+\eta \Re b)G(b)}{\Im b}-\frac{(E+i\eta) \Im bG(b)}{\Im b}|\\&=|\frac{\eta G(b)\bar{b}}{\Im b}|.
\end{aligned}\end{eqnarray}
Then Lemma \ref{sqrot} implies (\ref{geb49}).
Next, we consider \begin{eqnarray}\begin{aligned}
g(b)+zG(b)-R_2&=\frac{1}{M}\sum_{i=1}^M(\frac{cd_i^2|t_i|^2}{|b|^2}+cz|t_i|^2 )-\frac{1}{M}\sum_{i=1}^M (\frac{c d_i^2 t_i^2}{b^2}+cz t_i^2)\\&=\frac{1}{M}\sum_{i=1}^M c d_i^2|t_i|^2 (\frac{1}{|b|^2}-\frac{1}{b^2})+\frac{1}{M}\sum_{i=1}^M c(\frac{d_i^2}{b^2}+z)(|t_i|^2-t_i^2)\\
&=A_1+A_2
\end{aligned}
\end{eqnarray}
 Using a simple fact \begin{eqnarray}
\frac{1}{|z|^2}-\frac{1}{z^2}=\frac{2(\eta+E i)\eta}{|z|^4
},
\end{eqnarray}
it follows that $$|A_1| =O(\Im b), \quad  |A_2| \leq C\Im (\frac{d_i^2}{b}-zb) =O(\Im b).$$
Then \begin{eqnarray}\label{geb50}
|g(b)+zG(b)-R_2|= O(\Im b)=O(\sqrt{\kappa+\eta}).
\end{eqnarray}
We then have $|\alpha(z)|\leq C(\sqrt{\kappa+\eta})$ from \eqref{geb49} and \eqref{geb50} .\\

For the lower bound of $\alpha(z)$, we first consider the case that $\lambda_r-\tc \leq E \leq \lambda_r$, and $0 < \eta < \eta_0$, where $\eta_0$ is determined by the way in Lemma \ref{db432}. We first collect some basic control on real and imaginary part of some terms used later.  Based on the proof of Lemma \ref{db432}, and recall that $\omega=zb^2-(1-c)b$. we have \begin{eqnarray}\label{nmb21}
\max\{ \Re\frac{d_i^2-\omega}{b}, \Re(d_i^2-\omega)\} < -C\hat{c}.
\end{eqnarray}
We also have \begin{eqnarray}\label{nmb22}\begin{aligned}
\Im(d_i^2-w)&=\Im(\frac{d_i^2-\omega}{b}b)=2(\Re\frac{d_i^2-\omega}{b}\Im b+\Im \frac{d_i^2-\omega}{b} \Re b)\\ &\leq -C\hat{c}\Im b,
\end{aligned}
\end{eqnarray}
and \begin{eqnarray}\label{nmb23}
\Im (\frac{b}{d_i^2-\omega})^2 = -C\Im (\frac{d_i^2-\omega}{b})^2=-2C\Re \frac{d_i^2-\omega}{b}\Im \frac{d_i^2-\omega}{b}<0.
\end{eqnarray}
Now, for $\alpha(z)$, it can be bounded from below by
\begin{eqnarray}\label{ba4dg}
|\alpha(z)| = |1-R_2(z)|\geq |-\frac{1}{M}\sum_{i=1}^M \Im \frac{d_i^2+zb^2}{(d_i^2-\omega)^2}|.
\end{eqnarray}
For the first imaginary part term, we have\begin{eqnarray}
-\frac{1}{M}\sum_{i=1}^M \Im\frac{d_i^2}{(d_i-\omega)^2}=\frac{1}{M}\sum_{i=1}^M\frac{2d_i^2 \Re(d_i^2-\omega)\Im(d_i^2-\omega)}{|d_i^2-\omega|^4}\geq C\hat{c}^2  \Im b \geq C\sqrt{\kappa+\eta},
\end{eqnarray}
where in the second step, we use (\ref{nmb21}) and (\ref{nmb22}), and in the last step, we use Lemma \ref{sqrot}. For the second part in (\ref{ba4dg}), we have \begin{eqnarray}
-\frac{1}{M}\sum_{i=1}^M \Im\frac{zb^2}{(d_i-\omega)^2}=-\frac{1}{M}\sum_{i=1}^M E\Im (\frac{b}{d_i^2-\omega})^2 - \eta \Re (\frac{b}{d_i^2-\omega})^2>-C\eta
\end{eqnarray}
where in the last step, we use (\ref{nmb23}) and $\Re(\frac{b}{d_i^2-\omega})^2 =O(1)$.  Hence, we can choose a small constant $\eta_0^\prime$ smaller than $\eta_0$ such that $|\alpha(z)|\geq C\sqrt{\kappa+\eta}$ for all $z$ satisfying  $\lambda_r-\tc \leq E \leq \lambda_r$, and $0 < \eta < \eta_0^\prime$.\\
\indent For other cases, we are going to obtain the lower bound from \begin{eqnarray}
|1-R_2(z)|\geq 1-g(b)-|z|G(b)=\frac{[1+(1-c)G(b)]\eta G(b)}{\Im b[1-g(b)+|z|G(b)]},
\end{eqnarray}
where the second step is obtained from the argument between (2.3) and (2.5) in \cite{dozier2007empirical}. We have $0<1-g(b)+|z|G(b)\leq C$ inferred from (2.5) in \cite{dozier2007empirical}. Recall that $c\leq 1$, $G(b)$ is bounded from below, then, if $\lambda_r-\tc \leq E \leq \lambda_r$, and $\eta> \eta_0^\prime$, we have \begin{eqnarray*}
|1-R_2(z)|\geq C\frac{\eta}{\sqrt{\kappa+\eta}}\geq C\sqrt{\kappa+\eta}
\end{eqnarray*}
and if $\lambda_r \leq E \leq \tau^{-1}$, regardless of what's the order of $\eta$, we always have \begin{eqnarray*}
|1-R_2(z)|\geq C\sqrt{\kappa+\eta}
\end{eqnarray*}

For statement 2, the upper bound can be easily derived from Lemma \ref{db432}. For the lower bound, we have that\begin{eqnarray}\label{gab83}
R_3(\lambda_r)=\frac{c^2}{M}\sum_{i=1}^M \Big(\frac{3\lambda_r d_i^2}{b(\lambda_r)^2}-\frac{d_i^2(1-c)}{b(\lambda_r)^3}+\lambda_r^2\Big) (t_i(\lambda_r))^3
\end{eqnarray}
By \eqref{n8ag5} and recall that $b(\lambda_r)>0$ and $c\leq 1$, we have \begin{eqnarray*}\begin{aligned}
\lambda_r^2 > \big( \frac{d_i^2}{b(\lambda_r)^2}+\frac{1-c}{b(\lambda_r)}\big)^2 & \geq \frac{d_i^4}{b(\lambda_r)^4}+\frac{(1-c)^2}{b(\lambda_r)^2} \geq 2\frac{d_i^2(1-c)}{b(\lambda_r)^3}. \end{aligned}\end{eqnarray*}
Combining above with (\ref{gab83}), and use Lemma \ref{db432}, we have $|R_3(\lambda_r)|$ is bounded from below. Since $R_3(z)$ is an analytic function of $z$ in the neighborhood of $\lambda_r$, this proves the statement 2 in the lemma.
\end{proof}

We also need the following estimate for $\Lambda$ when $\eta\sim 1$ to get the bootstrapping procedure started.
\begin{Lemma}\label{etlar}
We can find  $\eta_b$  sufficiently large, such that for any $z$ in $\mD$ satisfying $\eta \geq \eta_b$, and any $\epsilon>0$, $$\Lambda\leq \frac{N^\epsilon}{\sqrt{N}}$$ holds with high probability.
\end{Lemma}

\begin{proof}
Using Lemma \ref{queq1} and  Lemma \ref{coebd}, we have $\sqrt{\eta}\Lambda
\leq C\Lambda^2 +CN^\epsilon\Psi_\Lambda$ holds with high probability, where C is a large constant independent of choice of $z$.
Using trivial estimates $$|G_{st}^{(\mathbb{T})}|\leq \frac{1}{\eta}, \quad |s_N^{(\mathbb{T})}|\leq \frac{1}{\eta},   \quad |s|\leq \frac{1}{\eta},$$
it follows that $\Lambda \leq 2/\eta.$
If we can find sufficiently large $\eta$ such that $2C/\eta \leq \sqrt{\eta}/2$, then
$\sqrt{\eta}\Lambda\leq 2C\Lambda/\eta+CN^\epsilon \Psi_\Lambda\leq \sqrt{\eta}\Lambda/2+CN^{\epsilon} \Psi_\Lambda$, which implies the lemma.
\end{proof}

\subsubsection{Proof of Theorem \ref{Wkloc}}
%The proof follows similar arguments in Section E of \cite{lee2013local}.
 Let $\beta(z):=(N\eta)^{-1/4}.$  Fix $\epsilon \in (0, \tau/20)$. For any fixed $E$, since $\sqrt{\kappa+\eta}$ is increasing and $\beta(E+i\eta)$ is decreasing in $\eta$, the equation \begin{eqnarray}\label{97fb4}
 \sqrt{\kappa+\eta}=N^{2\epsilon}\beta(E+i\eta) \end{eqnarray} has a unique solution $\tilde{\eta}$. Note that $\tilde{\eta} \ll 1$.

 We have that \begin{eqnarray}
 \Psi_\Lambda^2\leq 2(\frac{\Lambda+\Im s}{N\eta}+\frac{1}{(N\eta)^2})\leq 2(\beta^4\Lambda +\beta^4\alpha +\beta^8)
 \end{eqnarray}
 which implies \begin{eqnarray}
 \Psi_\Lambda \leq 2\sqrt{\beta^4\Lambda+\beta^4\alpha+\beta^8}=2\beta\sqrt{\beta^2\Lambda+\beta^2\alpha+\beta^6}\leq \beta^2\Lambda+\beta^2\alpha+2\beta^2.
 \end{eqnarray}
 Before proceeding the main proof, we need a dichotomy based on whether $\eta \geq \tilde{\eta}$ or $\eta<\tilde{\eta}$ first.
To make the following statement clear, we specify some constants to control the $\alpha(z)$ and $R(z)$ in Lemma (\ref{coebd}). Specifically, we suppose $C_1\sqrt{\kappa+\eta}\leq \alpha(z)\leq C_2\sqrt{\kappa+\eta}$, and $ |R_3(z)|\leq C_3$ for any $z\in\mD(\tc,\tau).$Then depending on the relative size of $\alpha$ and $\beta$, we consider the bound for $\Lambda$ separately.\\
Case 1.  $\eta\geq \tilde{\eta}$. On $\Omega(z)$, with probability $1-N^{-D}$, where $D$ is any positive constant,  by Lemma \ref{queq1}, we have that for sufficiently large $N$, $$\alpha\Lambda \leq C_3\Lambda^2 + N^\epsilon( \beta^2\Lambda+\beta^2\alpha+2\beta^2)$$
Using $N^\epsilon \beta \ll \alpha$, it follows that $$\alpha \Lambda \leq C_3\Lambda^2+C_4\alpha \beta.$$
Thus, we either have \begin{eqnarray}\label{3q5on}
\Lambda \geq\frac{\alpha}{2C_3} \quad  or \quad \Lambda \leq 2C_4\beta.
\end{eqnarray}
Case 2. $\eta<\tilde{\eta}$.  In this case,note that $\sqrt{\kappa+\eta} < N^{2\epsilon}\beta(E+i\eta)$=o(1), then $\kappa = o(1)$, and it follows that $R_3(z) \geq c^\prime.$
Using Lemma \ref{queq1}, $$c^\prime \Lambda^2 \leq \alpha\Lambda +CN^\epsilon (\beta^2\Lambda+\beta^2\alpha+2\beta^2)\leq N^{2\epsilon}\beta\Lambda + CN^\epsilon(\beta^2 \Lambda + \beta^2\alpha+2\beta^2)\leq \frac{c^\prime\Lambda^2}{2}+4N^{4\epsilon}\beta^2,$$which implies \begin{eqnarray}\label{4q3jb}
\Lambda \leq \sqrt{\frac{8}{c^\prime}}N^{2\epsilon}\beta.
\end{eqnarray}

After obtaining the dichotomy, we come to the main proof.
Define discrete set $\hat{D}=\mD(\tc,\tau)\cap (N^{-3}\mathbb{Z}^2).$We will first show that for $z\in \hat{D}$, $\Lambda \prec (N\eta)^{-1/4}$. For fixed $E$, let $z_0 = E+\eta_b i$, and define $z_k:= E+\eta_k i$ with $\eta_k = \eta_b-kN^{-3}.$ Define the event $\Omega_k := \Omega(z_k)\cap\{\Lambda(z_k)\leq C^{(k)}N^{2\epsilon}\beta(z_k)\}$, where \begin{eqnarray}C^{(k)}:=\Bigg\{
\begin{array}{ll}
{C_1/4C_3,}  &{\text{if} \quad \eta_k\geq \tilde{\eta}}, \\{ \sqrt{8/c^\prime},}  &{\text{if} \quad \eta_k <\tilde{\eta}}.
\end{array}
\end{eqnarray}
Next we will show that for any $D>0$ and any $k$, \begin{eqnarray}\label{vb2w5}P(\Omega_k^c)\leq 3kN^{-D}.\end{eqnarray}
We use induction method. When $k=1$, this has been proved by Lemma \ref{lemma123} and Lemma \ref{etlar}. Assume that (\ref{vb2w5}) holds for some $k\geq 1$, then $$P(\Omega_{k+1}^c) \leq P(\Omega_k\cap\Omega(z_{k+1})\cap \Omega_{k+1}^c)+P(\Omega_k\cap (\Omega(z_{k+1})^c)+P(\Omega_k^c):=B+A+P(\Omega_k^c),$$
where \begin{eqnarray}\begin{aligned}
&A:=P(\{\Omega_k\cap \{\Lambda(z_{k+1})>N^{-\frac{\tau}{10}})\}\}\cup\{\Omega_k \cap \{\Lambda_o(z_{k+1})>N^{-\frac{\tau}{10}}\}\})\\
&B:=P(\Omega_k\cap\Omega(z_{k+1})\cap\{\Lambda(z_{k+1})>C^{(k+1)}N^{2\epsilon}\beta(z_{k+1})\}).
\end{aligned}
\end{eqnarray}
We estimate $A$ first. By Lipschitz continuity of resolvent map: $z \mapsto G(z),z \in \mathbb{C}^{+}$, we have \begin{eqnarray}
|G_{st}(z_{k+1})-G_{st}(z_k)| \leq |z_{k+1}-z_k|\sup_{z\in \mD}|G_{st}^\prime(z)|\leq N^{-1}.
\end{eqnarray}
Thus $$\Lambda(z_{k+1})\leq \Lambda(z_k)+N^{-1}\leq CN^{2\epsilon}\beta_k < N^{\frac{\tau}{10}}N^{-\frac{\tau}{4}}< N^{-\frac{\tau}{10}},$$
and similarly, $\Lambda_o(z_{k+1}) < N^{-\frac{\tau}{10}}$. Thus $A\leq N^{-D}$.  To estimate $B$, recall that $\tilde{\eta}$ is solution to equation in (\ref{97fb4}), suppose that $\eta_k>\eta_{k+1}>\tilde{\eta}$, then \begin{eqnarray}
\Lambda(z_{k+1})\leq \Lambda(z_k)+N^{-1} \leq \frac{C_1}{4C_3}N^{2\epsilon} \beta_k +N^{-1} < \frac{\alpha}{2C_3}.
\end{eqnarray}
By the dichotomy argument in (\ref{3q5on}), we have that with probability $1-N^{-D}$, $\Lambda(z_{k+1})\leq 2C_4 \beta$.  If  $\eta_k<\tilde{\eta}$, (\ref{4q3jb}) implies $\Lambda(z_{k+1})\leq \sqrt{8/c^\prime} N^{2\epsilon}\beta_{k+1}.$ Thus $B\leq N^{-D}$.
By the inductive assumption, (\ref{vb2w5}) concludes.\\
\indent To finish the proof, we need to extend the result for $z \in \hat{D}$ to $z \in \mD(\tc,\tau)$. Similar to argument for the proof of Corollary 3.24 in \cite{lee2013local}, we can also show that there exists a constant C,
\begin{eqnarray}\label{3fonc} P(\bigcup_{z\in \mD(\tc,\tau)}\Omega(z)^c) + P(\bigcup_{z\in \mD(\tc,\tau)}\{\Lambda(z) > CN^{\epsilon}\beta(z)\}) \leq N^{-D^\prime}\end{eqnarray}
for arbitrary large $D^\prime$. This can be proceeded by Lipschitz continuity of  resolvent map: $z \mapsto G(z),z \in \mathbb{C}^{+}$ and error parameter $\beta$. Therefore, we conclude the proof for $\Lambda$ part. We can also conclude the proof of $\Lambda_o$ part by using Lemma \ref{lemma123},  (\ref{3fonc}) and lattice argument.
\qed

\subsection{Fluctuation Lemma and Strong Local Law}
In this subsection we first prove a fluctuation lemma (see Lemma \ref{fluctuation} below). Similar results for other random matrix models appear in \cite{erdHos2013averaging, erdHos2013local,lee2013local}, e.t.c., see also Section 6 of \cite{benaych2016lectures} whose argument we follow. However, there exist slight differences between our cases with previous results mainly because the particularity of  $\bY_{i\bbi}$, thus we need to analyze carefully. Then combining the fluctuation lemma with the Weak local law in Theorem \ref{Wkloc}, we finish the proof of Theorem \ref{Lolaw}.
\subsubsection{Fluctuation Lemma}
Recall $H$ defined in \eqref{HRplusX}. For $s\in \mI$, we define operations $P_{s}$ and $Q_s$ by
$$P_s[\, \cdot \,]:= E[\, \cdot\, | H^{(s)}], \quad \quad  Q_s:=1-P_s.$$
\begin{Lemma}\label{fluctuation}
Suppose that for $z\in \mD(\tc,\tau)$, we have $\Lambda_o \prec \Upsilon$, where $N^{-1/2}\leq \Upsilon \leq N^{-d}$ for some $d>0$.
%Then for any large even constant $p$,we have
%\begin{eqnarray}\begin{aligned}&E|\frac{1}{M}\sum_{i=1}^M Q_i(\frac{1}{G_{ii}})|^p \leq (Cp)^{4p} \Upsilon^{2p},\\&E|\frac{1}{M}\sum_{i=1}^M g_i^1\mY_i|^p \leq (Cp)^{4p} \Upsilon^{2p},\\&E|\frac{1}{M}\sum_{i=1}^M g_i^2\mZ_i|^p \leq (Cp)^{4p} \Upsilon^{2p}.\end{aligned}\end{eqnarray}
Then for $z\in \mD(\tc,\tau)$, we have  \begin{eqnarray}\label{fluct}
\begin{aligned}&\frac{1}{M}\sum_{i=1}^M Q_i(\frac{1}{G_{ii}}) =O_\prec (\Upsilon^2),\\&\frac{1}{M}\sum_{i=1}^M g_{1i}\mT_i =O_\prec (\Upsilon^2),  \\& \frac{1}{M}\sum_{i=1}^M g_{2i}\mS_i =O_\prec (\Upsilon^2),
\end{aligned}
\end{eqnarray}

where $(g_{1i})$ and $(g_{2i})$ are uniformly bounded constants.
\end{Lemma}
\begin{proof}
The first estimate in \eqref{fluct} can be proved by similar method used in the proof of Proposition 6.1 in \cite{benaych2016lectures}. One can also refer to the proofs of Lemma 4.9 in \cite{bloemendal2014isotropic} or Lemma 7.1 in \cite{pillai2014universality}. Recall $\mT_i$ introduced in \eqref{bc38a}. For the second estimate in \eqref{fluct}, we find \begin{eqnarray}
\mT_i = Q_i(\frac{1}{G_{ii}})+z(\tilde{s}_N-\tilde{s}_N^{(i)})= Q_i(\frac{1}{G_{ii}})-cz(s_N-s_N^{(i)}).
\end{eqnarray}
Then we have
\begin{eqnarray}
\frac{1}{M}\sum_{i=1}^M \mT_i =\frac{1}{M}\sum_{i=1}^M Q_i(\frac{1}{G_{ii}})+O(\Lambda_o^2)
=\frac{1}{M}\sum_{i=1}^M Q_i(\frac{1}{G_{ii}})+O_\prec(\Upsilon^2
)
\end{eqnarray}
where in the first equality we use \eqref{gfiminus}, and in the second equality we use Lemma \ref{lgii} and $M^{-1/2}\leq \Upsilon$.
Then if $g_{1i}=1, i=1,\cdots, M$, the second estimate in \eqref{fluct} follows. For general uniformly bounded $(g_{1i})$, the conclusion also holds since we used high moment  method to estimate the $\sum_{i=1}^M Q_i(\frac{1}{G_{ii}})$, and we can deal with $\sum_{i=1}^M g_{1i} Q_i(\frac{1}{G_{ii}})$ similarly.  \\

For the third estimate, by similar reason above, we just need to consider $g_{2i}=1, i=1,\cdots, M$. Recall $\mS_i$ introduced in \eqref{9g8a9}, we have \begin{eqnarray}\label{bnc45}\frac{1}{M}\sum_{i=1}^M \mS_i=-\frac{1}{M}\sum_{i=1}^M(1-E_{\bbi})\sum_{j,k}^{(i)}x_{j\bbi}^*G_{jk}^{(i\mu_i)}x_{k\bbi}+O(\Lambda_o^2). \end{eqnarray}  In the rest of this proof, we let $\by_{\mu}$  be the $\mu$-th column of $\bY$, and $\bx_{\mu}$ is defined similarly.  Denote $\tilde{V}_{i}:= \bx_{\bbi}^*G^{(i\mu_i)}\bx_{\bbi}$ and $V_{i}:= Q_{\bbi}\tilde{V}_{i} $. Next we are going to bound $E|\frac{1}{M}\sum_{i} V_{i}|^p $. In the following, we let $p$  be an even integer and denote $V_{i_s}:=V_{i_s} $ for $s\leq p/2$ and $V_{i_s}:=\bar{V}_{i_s}$ for $s>p/2$.   For a set $A\subset \{\mu_1, \cdots,\mu_M\}$ where $\mu_i = \bbi$,  we use notations $P_A:=\prod_{\mu\in A} P_\mu$ and $Q_A:=\prod_{\mu\in A} Q_\mu$. Let ${\bf{i}}= (i_1, \cdots, i_p)$ and $[{\bf{i}}]=\{i_1,\cdots,i_p\}$. We have \begin{eqnarray}\label{b6f20}\begin{aligned}
E|\frac{1}{M}\sum_{i} V_{i}|^p = & \frac{1}{M^p}\sum_{i_1,\cdots, i_p} E\prod_{s=1}^pV_{i_s}=\frac{1}{M^p}\sum_{i_1,\cdots, i_p} E\prod_{s=1}^p\left(\prod_{r=1}^s (P_{i_r}+Q_{i_r})V_{i_s}\right)\\= &\frac{1}{M^p}\sum_{\bf{i}}\sum_{A_1,\cdots, A_p \subset [{\bf{i}}]} E \prod_{s=1}^p(P_{A_s^c}Q_{A_s}V_{i_s}).
\end{aligned}
\end{eqnarray}
Following arguments from (7.21) to (7.23) in                                                        \cite{benaych2016lectures}, we have \begin{eqnarray}\label{gam93}\sum_{s=1}^p |A_s| \geq 2|[\bf i]|.\end{eqnarray}

The next step is to show that \begin{eqnarray}\label{nb7fg}|Q_A V_i|\prec \Upsilon^{|A|}. \end{eqnarray} If $A=|1|$ (corresponding to the case $A=\{\bbi\}$), it follows from Lemma \ref{ldb} directly. For the case $|A|\geq 2$ we assume that $i=1$ and $A=\{\mu_1,\mu_2,\cdots,\mu_t\}$ with $t\geq 2$.  We have \begin{eqnarray}\begin{aligned}&Q_{\mu_2}(\bx_{\mu_1}^*G^{(1\mu_1)}\bx_{\mu_1} )\\&=Q_{\mu_2}\bx_{\mu_1}^*\left(G^{(1\mu_1\mu_2)}-\frac{G^{(1\mu_1\mu_2)}\by_{\mu_2}\by_{\mu_2}^*G^{(1\mu_1\mu_2)}}{1+\by_{\mu_2}^*G^{(1\mu_1\mu_2)}\by_{\mu_2}}\right)\bx_{\mu_1}\\&=Q_{\mu_2}\bx_{\mu_1}^*G^{(1\mu_1\mu_2)}\by_{\mu_2}\by_{\mu_2}^*G^{(1\mu_1\mu_2)}\bx_{\mu_1}G_{\mu_2\mu_2}^{(1\mu_1)}\\&=Q_{\mu_2}\left(\frac{G_{\mu_1\mu_2}^{(1)}}{G_{\mu_1\mu_1}^{(1)}}\by_{\mu_2}^* G^{(1\mu_1\mu_2)}\bx_{\mu_1}\right),\end{aligned}\end{eqnarray}
where in the second equality,  $Q_{\mu_2(}\bx_{\mu_1}^*G^{(1\mu_1\mu_2)}\bx_{\mu_1})$ vanishes because $\bx_{\mu_1}^*G^{(1\mu_1\mu_2)}\bx_{\mu_1}$ is measurable with respect to $H^{(\mu_2)}$, and in the last equality, we use $G_{\mu_1\mu_2}^{(1)}=G_{\mu_1\mu_1}^{(1)}G_{\mu_2\mu_2}^{(1\mu_1)}\by_{\mu_2}^* G^{(1\mu_1\mu_2)}\bx_{\mu_1}$.
For the term $\by_{\mu_2}^* G^{(1\mu_1\mu_2)}\bx_{\mu_1}$, by using Lemma \ref{ldb}, we have
\begin{eqnarray}
 \by_{\mu_2}^* G^{(1\mu_1\mu_2)}\bx_{\mu_1}=\sum_{j}^{(1)}  \bx_{\mu_2 j}^* G_{jk}^{(1\mu_1\mu_2)}\bx_{k\mu_1}+d_2\sum_k^{(1)} G_{2k}^{(1\mu_1\mu_2)}\bx_{k\mu_1}\prec \Upsilon.
\end{eqnarray}
Therefore $Q_{\mu_2}(\bx_{\mu_1}^*G^{(1\mu_1)}\bx_{\mu_1} )\prec \Upsilon^2.$
Furthermore, we have \begin{eqnarray}\label{3gaon}
\begin{aligned}
&Q_{\mu_3}Q_{\mu_2}(\bx_{\mu_1}^*G^{(1\mu_1)}\bx_{\mu_1} )\\=& Q_{\mu_3}Q_{\mu_2}\left(\frac{G_{\mu_1\mu_2}^{(1)}}{G_{\mu_1\mu_1}^{(1)}}\by_{\mu_2}^* G^{(1\mu_1\mu_2)}\bx_{\mu_1}\right)
\\=&Q_{\mu_2}Q_{\mu_3}\left\{\left(G_{\mu_1\mu_2}^{(1\mu_3)}+\frac{G_{\mu_1\mu_3}^{(1)}G_{\mu_3\mu_2}^{(1)}}{G_{\mu_3\mu_3}^{(1)}}\right)
\left(\frac{1}{G_{\mu_1\mu_1}^{(1\mu_3)}}-\frac{G_{\mu_1\mu_3}^{(1)}G_{\mu_3\mu_1}^{(1)}}{G_{\mu_1\mu_1}^{(1)}G_{\mu_1\mu_1}^{(1\mu_3)}G_{\mu_3\mu_3}^{(1)}}\right)\right. \\ &\left.  \times \by_{\mu_2}^*\left(G^{(1\mu_1\mu_2\mu_3)}-G^{(1\mu_1\mu_2\mu_3)}\by_{\mu_3}\by_{\mu_3}^*G^{(1\mu_1\mu_2\mu_3)}G_{\mu_3\mu_3}^{(1\mu_1\mu_2)}\right) \bx_{\mu_1}\right\}
\\=& Q_{\mu_2}Q_{\mu_3}\left\{\frac{G_{\mu_1\mu_2}^{(1\mu_3)}}{G_{\mu_1\mu_1}^{(1\mu_3)}}\by_{\mu_2}^*G^{(1\mu_1\mu_2\mu_3)}  \bx_{\mu_1}-\frac{G_{\mu_1\mu_2}^{(1\mu_3)}}{G_{\mu_1\mu_1}^{(1\mu_3)}}\frac{G_{\mu_2\mu_3}^{(1\mu_1)}}{G_{\mu_2\mu_2}^{(1\mu_1)}}\by_{\mu_3}^*G^{(1\mu_1\mu_2\mu_3)}\bx_{\mu_1}\right.
\\&\left. +\left(\frac{G_{\mu_1\mu_3}^{(1)}G_{\mu_3\mu_2}^{(1)}}{G_{\mu_3\mu_3}^{(1)}G_{\mu_1\mu_1}^{(1\mu_3)}}-\frac{G_{\mu_1\mu_2}^{(1\mu_3)}G_{\mu_1\mu_3}^{(1)}G_{\mu_3\mu_1}^{(1)}}{G_{\mu_1\mu_1}^{(1)}G_{\mu_1\mu_1}^{(1\mu_3)}G_{\mu_3\mu_3}^{(1)}}-\frac{(G_{\mu_1\mu_3}^{(1)})^2 G_{\mu_3\mu_2}^{(1)} G_{\mu_3\mu_1}^{(1)}}{G_{\mu_1\mu_1}^{(1)}G_{\mu_1\mu_1}^{(1\mu_3)}(G_{\mu_3\mu_3}^{(1)})^2}\right) \by_{\mu_2}^*G^{(1\mu_1\mu_2\mu_3)}  \bx_{\mu_1} \right.
\\ & - \left. \left(\frac{G_{\mu_1\mu_3}^{(1)}G_{\mu_3\mu_2}^{(1)}}{G_{\mu_3\mu_3}^{(1)}G_{\mu_1\mu_1}^{(1\mu_3)}}-\frac{G_{\mu_1\mu_2}^{(1\mu_3)}G_{\mu_1\mu_3}^{(1)}G_{\mu_3\mu_1}^{(1)}}{G_{\mu_1\mu_1}^{(1)}G_{\mu_1\mu_1}^{(1\mu_3)}G_{\mu_3\mu_3}^{(1)}}-\frac{(G_{\mu_1\mu_3}^{(1)})^2 G_{\mu_3\mu_2}^{(1)} G_{\mu_3\mu_1}^{(1)}}{G_{\mu_1\mu_1}^{(1)}G_{\mu_1\mu_1}^{(1\mu_3)}(G_{\mu_3\mu_3}^{(1)})^2}\right)\frac{G_{\mu_2\mu_3}^{(1\mu_1)}}{G_{\mu_2\mu_2}^{(1\mu_1)}}\by_{\mu_3}^*G^{(1\mu_1\mu_2\mu_3)}\bx_{\mu_1}
\right\},
\end{aligned}
\end{eqnarray}
where in the last step, we use $\frac{G_{\mu_2\mu_3}^{(1\mu_1)}}{G_{\mu_2\mu_2}^{(1\mu_1)}}=G_{\mu_3\mu_3}^{(1\mu_1\mu_2)}\by_{\mu_2}^* G^{(1\mu_1\mu_2\mu_3)}\by_{\mu_3}.$
For the term $Q_{\mu_2}Q_{\mu_3}\frac{G_{\mu_1\mu_2}^{(1\mu_3)}}{G_{\mu_1\mu_1}^{(1\mu_3)}}\by_{\mu_2}^*G^{(1\mu_1\mu_2\mu_3)}  \bx_{\mu_1}$ in the last step of \eqref{3gaon}, it vanishes because $\frac{G_{\mu_1\mu_2}^{(1\mu_3)}}{G_{\mu_1\mu_1}^{(1\mu_3)}}\by_{\mu_2}^*G^{(1\mu_1\mu_2\mu_3)}  \bx_{\mu_1}$ is measurable with respect to $H^{(\mu_3)}$. Noting that the other terms are all fractions of $G$ entries multiplying with quadratic form of $\by_{\mu_j}^*G^{(T)}\bx_{\mu_1} (j = 2,\cdots, M) $, and the fractions have property  that the entries in the numerators are off-diagonal and the entries in the denominator are diagonal. Furthermore, the numbers of off-diagonal entries in each term are at least two,  and $\by_{\mu_j}^*G^{(T)}\bx_{\mu_1} (j = 2,\cdots, M) $ appears once. Therefore,  we have $Q_{\mu_3}Q_{\mu_2}(\bx_{\mu_1}^*G^{(1\mu_1)}\bx_{\mu_1} )\prec \Upsilon^3.$\\
\indent We may continue in this manner. At the step when $Q_{\mu_t}$ is appended,  we expand the Green entries by \eqref{gfiminus}, change the order of $Q$ operators (attach the $Q_{\mu_t}$ to expended formulas), and keep the non-vanishing terms. After each step,  the number of off-diagonal G entries in the numerator in each term increases by one. Consequently, we get $$Q_{\mu_t}\cdots Q_{\mu_2}(\bx_{\mu_1}^*G^{(1\mu_1)}\bx_{\mu_1} )\prec \Upsilon^t.$$ Therefore, we conclude \eqref{nb7fg}.\\

Using  \eqref{b6f20},\eqref{gam93} and \eqref{nb7fg}, we find \begin{eqnarray}\begin{aligned}
E|\frac{1}{M}\sum_{i} V_{i}|^p\prec C_p\frac{1}{M^p}\sum_{\bf i} \Upsilon^{2|[\bf i]|}=C_p &\sum_{u=1}^p\Upsilon^{2u}\frac{1}{M^p}\sum_{\bf i} {\bf{1}}(|[{\bf i}]| = u)\\& \leq C_p\sum_{u=1}^p\Upsilon^{2u}M^{u-p}\leq C_p(\Upsilon+M^{-1/2})^{2p}\leq C_p \Upsilon^{2p},
\end{aligned}\end{eqnarray}
where in the fourth step we use the inequality $a^m b^n\leq (a+b)^{m+n}$ for positive $a,b$. Then it follows that $$\frac{1}{M}\sum_i V_i  =O_\prec (\Upsilon^2).$$ This together with \eqref{bnc45} concludes the third estimate in \eqref{fluct}.
\end{proof}

Then we obtain an equation involving $[v]$, which is an improvement on Lemma \ref{queq1}.
\begin{Lemma}\label{sscon}
(Strong self-consistent equation)
The following equation holds with high probability, uniformly for $z\in \mD(\tc,\tau)$:
\begin{eqnarray}\label{strse}
\begin{aligned}
(1-R_2)[v]-R_3[v]^2 =& \frac{1}{M}\sum_{i=1}^M\frac{d_i^2 t_i^2\mS_i}{b^2}-\frac{1}{M}\sum_{i=1}^M t_i^2 \mT_i + \frac{1}{M}\sum_{i=1}^M(\frac{2c t_i^2 d_i^2 \mS_i }{b^3}-\frac{2c t_i^3 d_i^4 \mS_i }{b^4}-\\& \frac{2c t_i^3 z d_i^2 \mS_i }{b^2}-\frac{2c t_i^2 d_i^2 \mT_i }{b^2})[v]+
O(\Lambda^3)+O(\max_i (|\mS_i|^2)+O(\max_i|\mT_i|^2))
\end{aligned}
\end{eqnarray}
\end{Lemma}

\begin{proof}
This follows from Theorem \ref{Wkloc} and \eqref{seco1}.
\end{proof}

\subsubsection{Proof of Theorem \ref{Lolaw}}
The main step is to show that for any $\sigma \in [1/4,1]$,    $\Lambda\prec (N\eta)^{-\sigma} $ will imply $\Lambda \prec  (N\eta)^{-\frac{1}{2}-\frac{\sigma}{2}}$.
Denote \begin{eqnarray}
\begin{aligned}
g_{1i}=t_i^2 ,\quad   g_{2i}=\frac{2c t_i^2 d_i^2 }{b^2},\quad g_{3i}=\frac{d_i^2 t_i^2}{b^2}, \quad g_{4i}=\frac{2c t_i^2 d_i^2}{b^3}-\frac{2c t_i^3 d_i^4 }{b^4}- \frac{2c t_i^3 z d_i^2}{b^2}
\end{aligned}
\end{eqnarray}
Fix  $\epsilon \in (0, \tau/10)$ and $\sigma = 1/4$.  Define an event $\Xi$ as the intersection of following events:
\begin{eqnarray}\label{Eveset}\begin{aligned} &E_1 =\{ \text{the strong self-consistent equation in Lemma \ref{sscon} holds}\}, \\& E_2=\{ \Lambda +\Lambda_o \leq N^{\epsilon/2}(N\eta)^{-\sigma}\}, \\& E_3=\Big\{\Lambda_o+\max_i |\mS_i|+\max_i |\mT_i| \leq N^{\epsilon/2}\sqrt{\frac{\Im m+\Lambda}{N\eta}}\Big\},
\\& E_4=\Big\{\frac{1}{M}\left(|\sum_{i=1}^M g_{1i}\mZ_i|+ |\sum_{i=1}^M g_{2i}\mZ_i| +|\sum_{i=1}^M g_{3i}\mY_i| +|\sum_{i=1}^M g_{4i}\mY_i|\right) \leq N^\epsilon \frac{\Im m+\Lambda}{N\eta}\Big\}
\end{aligned}\end{eqnarray}
We know that $\Xi$ holds with high probability.
For the following, we fix an realization of $H \in \Xi$.
Denote $\alpha_0=N^\epsilon (N\eta)^{-\frac{1}{2}-\frac{\sigma}{2}}$. \\
\indent First consider the case $\alpha\leq \alpha_0$.  We find that \begin{eqnarray}\begin{aligned}
|[v]|^2 & \leq\left |[v]^2-\frac{(1-R_2)[v]}{R_3}\right|+|\frac{(1-R_2)[v]}{R_3}| \\ &\leq N^\epsilon \frac{\Im m+\Lambda}{N\eta}+\frac{|[v]|^2}{4 N^{\epsilon/2}}+N^{\epsilon/2}( N^{\epsilon}\frac{\Im m+\Lambda}{N\eta})^2+\frac{\alpha |[v]|}{c'}\\ & \leq N^\epsilon \frac{C\alpha +N^{\epsilon/2}(N\eta)^{-\sigma}}{N\eta} + o(1)|[v]|^2+ \frac{\alpha |[v]|}{c'},
\end{aligned}\end{eqnarray}
where in the second step we use  \eqref{strse} and the second statement in Lemma \ref{coebd}, and in the last step we use $\Im m \sim \sqrt{\kappa+\eta} \sim \alpha$ and $N^{2\epsilon}\frac{\Im m+\Lambda}{N\eta}\ll 1 $
and hence \begin{eqnarray}
(\Lambda-\frac{\alpha}{2c'})^2\leq N^\epsilon \frac{C\alpha +N^{\epsilon/2}(N\eta)^{-\sigma}}{N\eta}+\frac{\alpha^2}{4c'}.
\end{eqnarray}
Taking square root, recalling that $\alpha\leq \alpha_0$ and using the fact that $\sigma\leq 1$, we get \begin{eqnarray}
\Lambda \leq C \alpha_0.
\end{eqnarray}

Next, we consider the case $\alpha>\alpha_0.$ Suppose $|R_3(z)|\leq C_3 $ for $z \in D(\tc,\tau)$. Assume first that $\Lambda \leq
\alpha/2C_3$. By using \eqref{strse},   we get \begin{eqnarray}
\alpha \Lambda \leq C_3 \Lambda^2+N^\epsilon \frac{\Im m+\Lambda}{N\eta}+o(1)\Lambda^2\leq \frac{\alpha\Lambda}{2}+N^\epsilon \frac{C\alpha +N^{\epsilon/2}(N\eta)^{-\sigma}}{N\eta}.
\end{eqnarray}
This implies that \begin{eqnarray}
\Lambda\leq CN^{\epsilon}\frac{1}{N\eta}+\frac{\alpha_0^2}{N^{\epsilon/2}\alpha}.
\end{eqnarray}
Since $\alpha>\alpha_0
$, we have $\Lambda \ll \alpha_0 < \alpha$. Thus if  $\alpha>\alpha_0$, we have either $\Lambda > \alpha/2C_3$ or $\Lambda \ll \alpha$. By the continuity of $\Lambda(z)$ in $\eta$, and the fact that $\Lambda \ll 1\sim \alpha $ when $\eta\sim 1$, we must have $\Lambda \ll \alpha$. Thus $\Lambda \leq \alpha/2C_3$, and it follows that $\Lambda \ll \alpha_0$.

We conclude that for any $\sigma\in[1/4, 1]$,  $\Lambda \prec (N\eta)^{-\sigma}$ implies $\Lambda\prec (N\eta)^{-\frac{1}{2}-\frac{\sigma}{2}}.$ Therefore, we can start from $\sigma=1/4$ which is guaranteed by Theorem \ref{Wkloc}, and iterate the procedure finite number of times to get $\Lambda \prec (N\eta)^{-1}.$ \eqref{Lolaw2} follows from Lemma \eqref{Lolaw1} and Lemma \ref{lemma123}. This concludes the proof of Theorem \ref{Lolaw}.
\qed

\subsection{Proof of Proposition \ref{crare}}
The proof relies on the strong local law Theorem \ref{Lolaw}, and is analogous to the counter part for Wigner matrix case in \cite{benaych2016lectures} or sample covariance matrix in \cite{bao2015universality}.\\

Note that by Lemma \ref{sqrot}, it is easy to obtain that $\mu_1-\lambda_r>  -N^{-2/3+\epsilon}$ holds with high probability. Therefore, we just need to show that $\mu_1\leq \lambda_r+N^{-2/3+\epsilon}$. It is well known that the largest singular value of $\bX$ is bounded above with high probability. Together with the assumption on the boundedness of the largest singular value of $\bR$, we know that there exists $C_r$ such that $\mu_1\leq C_r$. It remains to show that, for any fixed $\epsilon>0$, there is no eigenvalue of $\bQ$ in \begin{eqnarray}
I :=[\lambda_r+N^{-2/3+\epsilon}, C_r]
\end{eqnarray} with high probability.
Recall \eqref{strse} and the event set defined in \eqref{Eveset}, we have that \begin{eqnarray*}
\Lambda \leq \frac{1}{\alpha}\left(\Lambda^2 + N^\epsilon \frac{\Im m+\Lambda}{N\eta}\right).
\end{eqnarray*}
Thus by the Theorem \ref{Lolaw}, we have \begin{eqnarray}
\Lambda \prec \frac{1}{\sqrt{\kappa+\eta}}\left(\frac{\Im m}{N\eta} +\frac{1}{(N\eta)^2}\right),
\end{eqnarray}
which is exactly the same as Lemma 9.2 in \cite{benaych2016lectures}. Therefore we can proceed as the proof of Proposition 9.1 in \cite{benaych2016lectures}. \qed

\section{Optical Theorems}\label{secoptth}
%Using \eqref{gfiimu}, we expand $G_{i\bba}G_{\bba i}$ in the lower index $\bba$ as \begin{eqnarray}\begin{aligned}\sum_{a}^{(i)} G_{i\bba}G_{\bba i}&=\sum_{a}^{(i)} G_{\bba \bba}^2 G_{ij}^{(\bba)}Y_{j\bba}Y_{k\bba} G_{ik}^{(\bba)}\\&=\gamma \sum_{a}^{(i)}  G_{\bba \bba}^2 G_{ia}^{(\bba)}d_a^2 G_{ai}^{(\bba)}+ 2\gamma \sum_{a}^{(i)}  G_{\bba\bba}^2  G_{ij}^{(\bba)} x_{j\bba}d_a G_{ia}^{(\bba)}+\gamma \sum_{a}^{(i)} G_{\bba \bba}^2 G_{ij}^{(\bba)}x_{j\bba}x_{k\bba}G_{ik}^{(\bba)}.\end{aligned}\end{eqnarray}
\subsection{Preliminaries}
We list some notations and useful expression that we will frequently use later. First we consider then expansion of term $G_{aa}$.
Using schur complement formula, we have \begin{eqnarray}\label{Gaainver}\begin{aligned}
&\frac{1}{G_{aa}}=-z-\gamma d_a^2 G_{\bba\bba}^{(a)} - \gamma \sum_{\mu,\nu}x_{a\mu} G_{\mu\nu}^{(a)}x_{a\nu}-2 \gamma d_a\sum_{\mu} G_{\bba\mu}^{(a)}x_{a\mu},\\
&\frac{1}{G_{\bba\bba}^{(a)}}=-1-\gamma\sum_{j,k}^{(a)} x_{j\bba}G_{jk}^{(a\bba)} x_{ka}.
\end{aligned}\end{eqnarray}
Recall notations $b$ in \eqref{rescb} and $\tb$ in \eqref{deftb}.
Denote \begin{eqnarray}\label{defmZmY}\begin{aligned}
& \beta_a:=\frac{1}{-\gamma d_a^2 G_{\bba\bba}^{(a)}-\tb}, \quad \mY_{a}: =\gamma \left(cm(E_+)- \sum_{j,k}^{(a)}x_{j\bba}G_{jk}^{(a\bba)}x_{k\bba}\right),\\
&\mZ_{a}: =E_+-z+\gamma \left (E_+\tm(E_+)-\sum_{\mu,\nu}x_{a\mu} G_{\mu\nu}^{(a)}x_{a\nu}-2 d_a\sum_{\mu} G_{\bba\mu}^{(a)}x_{a\mu}\right).\\
\end{aligned}
\end{eqnarray}
\begin{Lemma}\label{mYZorder}For $z$ in $\mathcal{E}_\epsilon(t)$ defined in \eqref{domainofz},
\begin{eqnarray}\label{89agi}m(E_+)-m(z)=O(\Psi),\end{eqnarray}
and \begin{eqnarray}\label{mYmZbound}\mY_a, \mZ_a \prec \Psi. \end{eqnarray}
\end{Lemma}
\begin{proof}
We show \eqref{89agi} first. Recalling that  $z= \hat{\phi}(w)$, $E_+ = \hat{\phi}(\xi)$ and $\hat{\phi}'(\xi)=0$, we have for some sufficiently small constant $\epsilon'$, \begin{eqnarray}
z-E_+ = \frac{1}{2}\hat{\phi}''(\xi)(w-\xi)^2 + O(|w-\xi |)^3, \; \mbox{for} \;  |z-E_+|\leq \epsilon_1,
\end{eqnarray}
where we use a simple fact  $\hat{\phi}'''(\xi) < C$, where $C$ is a large constant.
By calculations, we find there exists a positive constant $c_1$ such that  \begin{eqnarray}
\hat{\phi}''(\xi) = \frac{1}{M}\sum_i \frac{\frac{-4\gamma^2 d_i^2}{b}-2(1-c)\gamma^2}{(\gamma d_i^2-\xi)^3} + 2\xi \frac{\gamma^2}{M}\left (\sum_i \frac{1}{(\gamma d_i^2-\xi)^2}\right)^2>c_1>0,
\end{eqnarray}
where we use Assumptions \ref{assump2} and \ref{assump3}.  Therefore, we get $$|w(z)-w(E_+)|= \sqrt{|z-E_+|}(1+o(1))=O( \Psi). $$
Now \eqref{89agi} easily follows from this.\\
\indent Then we deal with $\mY_a$.  Write \begin{eqnarray} \mY_a= (cm-\frac{1}{N}\sum_j G_{jj})+(\frac{1}{N}\sum_j G_{jj}-\frac{1}{N}\sum_{j}^{(a)} G_{jj}^{(a\bba)} )+(\frac{1}{N}\sum_{j}^{(a)} G_{jj}^{(a\bba)}-\sum_{j,k}^{(a)}x_{j\bba}G_{jk}^{(a\bba)}x_{k\bba}).\end{eqnarray}
From \eqref{0011a} and \eqref{gfiminus}, we see that terms in the first two brackets are of order $O(\Psi)$. Using \eqref{ldb} and Ward identity, we find for the last term that $$ \left|\frac{1}{N}\sum_{j}^{(a)} G_{jj}^{(a\bba)}-\sum_{j,k}^{(a)}x_{j\bba}G_{jk}^{(a\bba)}x_{k\bba}\right|\prec \left(\frac{1}{N^2}\sum_{j,k}^{(a)}|G_{jk}^{(a\bba)}|^2\right)^{1/2}\prec \Psi,$$
Therefore we conclude $\mY_a\prec \Psi$.   $\mZ_a$ can be dealt by a similar argument to $\mY_a$ with the help of \eqref{0ab3m}.
\end{proof}

From \eqref{Gaainver}, \eqref{defmZmY}, and using \eqref{mYmZbound}, it is easy to obtain the following from Taylor's expansion, \begin{eqnarray}\label{taylorofGaa}\begin{aligned}
&G_{aa}=\beta_{a}-\beta_{a}^2 \mZ_a+\beta_{a}^3 \mZ_a^2 + O(\Psi^3),\\
&G_{\bba\bba}^{(a)}=-\frac{1}{b}-\frac{1}{b^2}\mY_a- \frac{1}{b^3}\mY_a^2 + O(\Psi^3)
\end{aligned}
\end{eqnarray}
Notice that $\xi = b\tb$,  using Taylor's expansion again, we get \begin{eqnarray}\label{taylorofbetaa}
\begin{aligned}
&\beta_a =\frac{b}{\gamma d_a^2-\xi} -\frac{\gamma d_a^2 }{(\gamma d_a^2-\xi)^2}\mY_a+\frac{\gamma d_a^2 \tb }{(\gamma d_a^2-\xi)^3}\mY_a^2+O(\Psi^3),\\
&\beta_a G_{\bba\bba}^{(a)2}=\frac{1}{(\gamma d_a^2-\xi)b} + \frac{\gamma d_a^2-2\xi}{(\gamma d_a^2-\xi)^2 b^2
}\mY_a+O(\Psi^2),\\
&\beta_a^2 G_{\bba\bba}^{(a)2}=\frac{1}{(\gamma d_a^2-\xi)^2}-\frac{2\tb}{(\gamma d_a^2-\xi)^3}\mY_a+ \frac{3\tb^2}{(\gamma d_a^2-\xi)^4}\mY_a^2+O(\Psi^3).
\end{aligned}
\end{eqnarray}

Next, we can expand $G_{\bba\bba}$ similarly. Let \begin{eqnarray}
\begin{aligned}
&\alpha_a := \frac{1}{-\gamma d_a^2 G_{aa}^{(\bba)}-b},\quad  \mU_a := E_+-z+\gamma \left( E_+\tm(E_+) -\sum_{\mu,\nu}^{(\bba)} x_{a\mu}G_{\mu\nu}^{(a\bba)}G_{a\nu}\right),\\
&\mV_a := \gamma\left(cm(E_+) -\sum_{j,k} x_{j\bba} G_{jk}^{(\bba)} x_{k\bba} -2d_a \sum_j G_{aj}^{(\bba)}x_{j\bba}\right).
\end{aligned}
\end{eqnarray}
Then it is easy to obtain that \begin{eqnarray}\label{taylorofGbba}\begin{aligned}
&G_{\bba\bba}=\alpha_{a}-\alpha_{a}^2 \mV_a+\alpha_{a}^3 \mV_a^2 + O(\Psi^3),\\
&G_{aa}^{(\bba)}=-\frac{1}{\tb}-\frac{1}{\tb^2}\mU_a- \frac{1}{\tb^3}\mU_a^2 + O(\Psi^3).
\end{aligned}\end{eqnarray}
And we have \begin{eqnarray}\label{taylorofalphaa}\begin{aligned}
&\alpha_a =\frac{\tb}{\gamma d_a^2-\xi} -\frac{\gamma d_a^2 }{(\gamma d_a^2-\xi)^2}\mU_a+\frac{\gamma d_a^2 b }{(\gamma d_a^2-\xi)^3} \mU_a^2+O(\Psi^3),\\
&\alpha_a G_{aa}^{(\bba)2}=\frac{1}{(\gamma d_a^2-\xi)\tb} + \frac{\gamma d_a^2-2\xi}{(\gamma d_a^2-\xi)^2 \tb^2}\mU_a+O(\Psi^2),\\
&\alpha_a^2 G_{aa}^{(\bba)2}=\frac{1}{(\gamma d_a^2-\xi)^2}-\frac{2b}{(\gamma d_a^2-\xi)^3}\mU_a+\frac{3b^2}{(\gamma d_a^2-\xi)^4}\mU_a + O(\Psi^3).
\end{aligned}
\end{eqnarray}
For $\alpha>2M$, we have \begin{eqnarray}\label{taylorofGalphaalpha}
G_{\alpha\alpha}=-\frac{1}{b}-\frac{\mV_\alpha}{b^2}-\frac{\mV_{\alpha}^2}{b^3}+O(\Psi^3),
\end{eqnarray}
where $$\mV_\alpha = \gamma(cm-\sum_{j,k} x_{j\alpha}G_{jk}^{(\alpha)}G_{k\alpha}).$$
Similar to Lemma \ref{mYZorder}, we have \begin{eqnarray}\label{mUVorder}
\mU_a, \mV_a \prec \Psi, \; a \in \mI_M; \quad    \mV_\alpha \prec \Psi, \; \alpha \in [\![2M+1, M+N]\!].
\end{eqnarray}

We repeat equations \eqref{R1} and \eqref{R2} here, since they play crucial role in deriving the optical theorems, \begin{eqnarray}\label{R1g}%\frac{1}{N}\sum_{i=1}^N \frac{\gamma E_+b^2+\gamma^2 d_i^2}{(\gamma d_i^2-\xi)^2}=1,
\frac{1}{N}\sum_{i=1}^M \frac{\gamma  b(E_+ b^2+\xi)}{(\gamma d_i^2-\xi)^2}=1,
\end{eqnarray}
and \begin{eqnarray}\label{R2g}
\frac{1}{\gamma^2}\frac{1}{N}\sum_{i=1}^M\frac{b^2}{(\gamma d_i^2-\xi)^2}=-\frac{1}{\gamma b^3}-\frac{1}{N}\sum_{i=1}^M\frac{(2 E_+ b- \gamma(1-c))^2}{(\gamma d_i^2-\xi)^3}-\frac{1}{N}\sum_{i=1}^M\frac{E_+}{(\gamma d_i^2-\xi)^2}
\end{eqnarray}
%We also need a variation of \eqref{R1g}:\begin{eqnarray}\label{nve4o}\frac{1}{N}\sum_{i=1}^N \frac{2\gamma E_+b^3-\gamma^2(1-c) b^2}{(\gamma d_i^2-\xi)^2}=\frac{1}{N}\sum_{i=1}^N \frac{\gamma b^2( E_+b+  \tb)}{(\gamma d_i^2-\xi)^2}=1.\end{eqnarray}

For $k \in \mathbb{N}$, let \begin{eqnarray}\label{notatvarPhi}
\begin{aligned}
&\varphi_k:=\frac{1}{N}\sum_{a}\frac{1}{(\gamma d_a^2-\xi)^k},\; \psi_k:=\frac{1}{N}\sum_{a}\frac{\gamma d_a^2}{(\gamma d_a^2-\xi)^k},\;
\varpi_2:=\sum \frac{\tb^2}{(\gamma d_a^2-\xi)^2}+\frac{1-c}{b^2}.\\
&\Phi_1 := \frac{1}{N}\sum \left(\frac{\gamma^3 d_a^2 \tb}{(\gamma d_a^2-\xi)^3}+ \frac{2\gamma^3  d_a^2 b E_+}{(\gamma d_a^2-\xi)^3}+\frac{\gamma^2 b^3 E_+^2}{(\gamma d_a^2-\xi)^3}\right),\\
&\Phi_2 := \frac{1}{N}\sum \left(\frac{\gamma^3  d_a^2 b E_+^2}{(\gamma d_a^2-\xi)^3}+\frac{2\gamma^3 d_a^2 \tb E_+}{(\gamma d_a^2-\xi)^3}+\frac{\gamma^2\tb^3}{(\gamma d_a^2-\xi)^3}\right)-\gamma^2 (1-c)\frac{1}{b^3}.
\end{aligned}
\end{eqnarray}
The above quantities $\varphi_2, \psi_2,\varpi_2,\Phi_1, \Phi_2$ will appear naturally in the derivation of optical theorems, and the following lemma will help simplify the calculations, see e.g. \eqref{optfromX22}.
\begin{Lemma}\label{Phivalue} Denote $h=E_+ b+\tb$. We have
\begin{eqnarray}\label{relofvarphi}
\varphi_2 = \frac{1}{\gamma b^2 h}, \quad \psi_2 = \frac{1}{\gamma}-E_+ b^2 \varphi_2, \quad  \varphi_3= -\left(\frac{1}{\gamma^3 h^3}+\frac{1}{\gamma b^3 h^2}+\frac{E_+}{\gamma b^2 h^3}\right) \quad \varpi_2=E_+^2 b^2 \varphi_2
\end{eqnarray}
\begin{eqnarray}\label{relofPhi}\Phi_1= -b^3 \varphi_2, \; \mbox{and}\quad   \Phi_2 = -\tb b^2 \varphi_2.
\end{eqnarray}
\end{Lemma}
\begin{proof}
From \eqref{R1g} and \eqref{xibtb}, we see that $\varphi_2 = 1/\gamma b^2 h.$
It follows that $\psi_2 = \frac{1}{\gamma}-E_+ b^2 \varphi_2$ by using a simple fact $\varphi_1 = (b-1)/\gamma b$. The third relation in \eqref{relofvarphi} is obtained from \eqref{R2g}. The last one in \eqref{relofvarphi} can be verified without difficulty. Next we verify $\Phi_1$,
\begin{eqnarray*}\begin{aligned}\Phi_1 &= \frac{1}{N}\sum_a \left(\frac{\gamma^3 d_a^2 \tb}{(\gamma d_a^2-\xi)^3}+ \frac{2\gamma^3 b d_a^2 E_+}{(\gamma d_a^2-\xi)^3}+\frac{\gamma^2 b^3 E_+^2}{(\gamma d_a^2-\xi)^3}\right)\\
&=\frac{1}{N}\sum_a \left(\frac{\gamma^2  \tb}{(\gamma d_a^2-\xi)^2}+ \frac{2\gamma^2 b E_+}{(\gamma d_a^2-\xi)^2}+\frac{\gamma^2 b \tb^2 +2\gamma^2 b^2 \tb E_+ + \gamma^2 b^3 E_+^2}{(\gamma d_a^2-\xi)^3}\right)\\
&=\frac{\gamma^2 }{N}\sum_a \left(\frac{\tb+E_+ b}{(\gamma d_a^2-\xi)^2}+ \frac{E_+ b}{(\gamma d_a^2-\xi)^2}+\frac{b(\tb+E_+ b)^2}{(\gamma d_a^2-\xi)^3}\right)\\
&=\frac{\gamma}{b^2}+\frac{\gamma^2 }{N}\sum_a \left( \frac{E_+ b}{(\gamma d_a^2-\xi)^2}+\frac{b(\tb+E_+ b)^2}{(\gamma d_a^2-\xi)^3}\right)\\
&=-b^3 \varphi_2,
 \end{aligned}\end{eqnarray*}
 where in the last step we use \eqref{R2g}.
 Similarly, we can get $\Phi_2 =-\tb b^2 \varphi_2$.
\end{proof}

%\begin{eqnarray}E_+\tm(E_+)-\frac{1}{N}\sum G_{\mu\mu} = E_+cm(E_+)-z cm(z)+zcm(z)-\frac{z}{N}\sum G_{jj}\end{eqnarray}
%\begin{eqnarray}cm(E_+)-\frac{1}{N}\sum G_{jj}=cm(E_+)-cm(z)+cm(z)-\frac{1}{N}\sum G_{jj}\end{eqnarray}
%Therefore, we have \begin{eqnarray}\label{relofmtm}\begin{aligned}E_+\tm(E_+)-\frac{1}{N}\sum G_{\mu\mu}&=z(cm(E_+)-\frac{1}{N}\sum G_{jj})+(E_+-z)cm(E_+)\\&=E_+(cm(E_+)-\frac{1}{N}\sum G_{jj})+(E_+-z)cm(E_+) + O(\Psi^3)\end{aligned}\end{eqnarray}

There are some useful relations between the quantities defined in \eqref{Xset1} and \eqref{Xset2}. We list those that will be used frequently later, for example, see \eqref{03tm8} below.
\begin{Lemma}\label{multiErel}
\begin{eqnarray}\label{EX44}
\E X_{44}'''=E_+ \E X_{44}'' +O(\Psi^5)= E_+^2 \E X_{44}'+O(\Psi^5) = E_+^3 \E X_{44} +O(\Psi^5),
\end{eqnarray}
\begin{eqnarray}\label{EX43}
\E X_{43}''=E_+  \E X_{43}'+O(\Psi^5) = E_+^2 \E X_{43} +O(\Psi^5),
\end{eqnarray}
\begin{eqnarray}\label{EX42}
\E X_{42}'=E_+ \E X_{42} + O(\Psi^5),\quad  (E_+-z)\E X_{22}'=E_+(E_+ -z)\E X_{22}+O(\Psi^5).
\end{eqnarray}
\end{Lemma}

\begin{Rem}
Actually, the relations depend on how many Greek lower indices in the quantity. When the number of Greek letters in the lower indices decreases by one, meanwhile the number of Latin indices increases by one,  the expectation increases by a factor of $E_+$ with an negligible error of order $O(\Psi^5).$
\end{Rem}

The following Lemma is also useful in the calculations to derive optical theorems, for example, see \eqref{bal3g} below.
\begin{Lemma}\label{kao4o}
We have
\begin{eqnarray}\label{zo9g2}
\begin{aligned}
&\E \frac{1}{N^2}\sum_{\alpha,\beta}^{(\bba)}G_{i\alpha}^{(a\bba)}G_{\alpha\beta}^{(a\bba)}G_{\beta i}^{(a\bba)}=\E  X_{33}''- \frac{ \gamma E_+^2(E_+ b+\tb)}{\gamma d_a^2-\xi}\E X_{44}+O(\Psi^5)\\
&\E \frac{1}{N^2}\sum_{j,\beta}G_{ij}^{(a\bba)}G_{j\beta}^{(a\bba)}G_{\beta i}^{(a\bba)}=\E  X_{33}'- \frac{ \gamma E_+(E_+ b+\tb)}{\gamma d_a^2-\xi}\E X_{44}+O(\Psi^5)\\
&\E \frac{1}{N^2}\sum_{j,k}G_{ij}^{(a\bba)}G_{jk}^{(a\bba)}G_{k i}^{(a\bba)}=\E  X_{33}- \frac{ \gamma (E_+ b+\tb)}{\gamma d_a^2-\xi}\E X_{44}+O(\Psi^5)\\
&\E \frac{1}{N} (cm-cm_N^{(a\bba)}) \sum_\alpha G_{i\alpha}^{(a\bba)}G_{\alpha i}^{(a\bba)}=\E X_{32}'+\frac{\gamma E_+ (E_+ b+\tb)}{\gamma d_a^2-\xi}(\E \tX_{44}-2\E X_{43})+O(\Psi^5)\\
&\E \frac{1}{N} (cm-cm_N^{(a\bba)}) \sum_j G_{ij}^{(a\bba)}G_{ji}^{(a\bba)}=\E  X_{32}+\frac{\gamma  (E_+ b+\tb)}{\gamma d_a^2-\xi}(\E \tX_{44}-2\E X_{43})+O(\Psi^5)
\end{aligned}
\end{eqnarray}
\end{Lemma}
The proof of Lemma \ref{multiErel} and \ref{kao4o} are postponed to the end of Section 7.

\subsection{Optical theorem from $X_{22}$}
We start from expanding $\E G_{ia}G_{ai}$ with respect to index $a$. When $a=i$, \begin{eqnarray}\label{X22spe}\begin{aligned}
\E G_{ii}^2 = \frac{b^2}{(\gamma d_i^2-\xi)^2}-\frac{2\gamma^2 b d_i^2 }{(\gamma d_i^2-\xi)^3}\E (cm-cm_N)-\frac{2\gamma b^3 E_+}{(\gamma d_i^2-\xi)^3}\E (cm-cm_N)+O(\Psi^2).
\end{aligned}\end{eqnarray}
For $a\neq i$, using \eqref{gfiij}, we write
\begin{eqnarray}\begin{aligned}
 G_{ia}G_{ai}&=\sum_{\alpha, \beta} G_{aa}^2 G_{i\alpha}^{(a)}Y_{a\alpha}Y_{a\beta} G_{i\beta}^{(a)}\\&=\gamma   G_{aa}^2 G_{i\bba}^{(a)}d_a^2 G_{\bba i}^{(a)}+ 2\gamma \sum_{\alpha}  G_{aa}^2  G_{i\alpha}^{(a)} x_{a\alpha}d_a G_{i\bba}^{(a)}+\gamma \sum_{\alpha,\beta} G_{aa}^2 G_{i\alpha}^{(a)}x_{a\alpha}x_{a\beta}G_{i\beta}^{(a)}.
\end{aligned}\end{eqnarray}
Denote \begin{eqnarray}\label{90afu}
B_1:=  G_{aa}^2 G_{i\bba}^{(a)}d_a^2 G_{\bba i}^{(a)}, \quad B_2:= \sum_{\alpha}  G_{aa}^2  G_{i\alpha}^{(a)} x_{a\alpha}d_a G_{i\bba}^{(a)}, \quad B_3:= \sum_{\alpha,\beta} G_{aa}^2 G_{i\alpha}^{(a)}x_{a\alpha}x_{a\beta}G_{i\beta}^{(a)}.
\end{eqnarray}
\subsubsection{Expansion of $B_1$}
 Consider $B_1$ first. Using \eqref{gfiimu}, we write
\begin{eqnarray}
\begin{aligned}
B_1 = \gamma \sum_{p,q}^{(a)} G_{aa}^2 G_{\bba\bba}^{(a)2} d_a^2  G_{ip}^{(a\bba)}x_{p\bba}x_{q\bba}G_{iq}^{(a\bba)}
\end{aligned}
\end{eqnarray}
Using \eqref{taylorofGaa} and \eqref{taylorofbetaa}, we decompose $B_1$ into summation of six terms defined below: \begin{eqnarray}\label{B1123456}\begin{aligned}
&B_{11}:=\sum_{p,q}^{(a)}  \frac{\gamma d_a^2}{(\gamma d_a^2-\xi)^2} G_{ip}^{(a\bba)}x_{p\bba}x_{q\bba}G_{iq}^{(a\bba)},  \quad B_{12}:=\sum_{p,q}^{(a)} \frac{-2\gamma \tb d_a^2 \mY_a}{(\gamma d_a^2-\xi)^3} G_{ip}^{(a\bba)}x_{p\bba}x_{q\bba}G_{iq}^{(a\bba)},\\
&B_{13}:=\sum_{p,q}^{(a)} \frac{-2\gamma b d_a^2 \mZ_a }{(\gamma d_a^2-\xi)^3}G_{ip}^{(a\bba)}x_{p\bba}x_{q\bba}G_{iq}^{(a\bba)},\quad
B_{14}:=\sum_{p,q}^{(a)}  \frac{3\gamma \tb^2 d_a^2 \mY_a^2}{(\gamma d_a^2-\xi)^4} G_{ip}^{(a\bba)}x_{p\bba}x_{q\bba}G_{iq}^{(a\bba)},\\
&B_{15}:=\sum_{p,q}^{(a)} \frac{(2\gamma^2 d_a^4+4 \gamma \xi d_a^2)\mY_a \mZ_a}{(\gamma d_a^2-\xi)^4} G_{ip}^{(a\bba)}x_{p\bba}x_{q\bba}G_{iq}^{(a\bba)}, \quad
B_{16} :=\sum_{p,q}^{(a)}  \frac{3\gamma b^2 d_a^2 \mZ_a^2}{(\gamma d_a^2-\xi)^4} G_{ip}^{(a\bba)}x_{p\bba}x_{q\bba}G_{iq}^{(a\bba)}.
\end{aligned}
\end{eqnarray}
In the following, we analyze these terms one by one. We start from $B_{11}$ that requires more calculations to expand to desired order.

\textbf{Expansion of $B_{11}$.}
By taking partial expectation $\E_{\bba}$ of $B_{11}$, we obtain \begin{eqnarray}\label{w33rf}\E_{\bba}B_{11} =\frac{1}{N}\sum_{p}^{(a)} \frac{\gamma d_a^2}{(\gamma d_a^2-\xi)^2} G_{ip}^{(a\bba)}G_{pi}^{(a\bba)}. \end{eqnarray}
%Using \begin{eqnarray}G_{ip}^{(a\bba)}G_{pi}^{(a\bba)} = \left(G_{ip}-\frac{G_{i\bba}G_{\bba p}}{G_{\bba\bba}}-\frac{G_{ia}^{(\bba)}G_{ap}^{(\bba)}}{G_{aa}^{(\bba)}}\right)^2\end{eqnarray}
Using \eqref{gfiminus}, write\begin{eqnarray}\label{na97a}
G_{ip}^{(a\bba)}G_{pi}^{(a\bba)} = \left(G_{ip}-\frac{G_{i\bba}G_{\bba p}}{G_{\bba\bba}}-\frac{G_{ia}^{(\bba)}G_{ap}^{(\bba)}}{G_{aa}^{(\bba)}}\right) G_{pi}-G_{ip}^{(a\bba)}\frac{G_{i\bba}G_{\bba p}}{G_{\bba\bba}}-G_{ip}^{(a\bba)}\frac{G_{ia}^{(\bba)}G_{ap}^{(\bba)}}{G_{aa}^{(\bba)}}.\end{eqnarray}
Substituting \eqref{na97a} into \eqref{w33rf}, we decompose $E_{\bba} B_{11}$ into summation of following five terms:
\begin{eqnarray}\label{ExpB11}\begin{aligned}
&B_{111}:=\frac{1}{N}\sum_{p} \frac{\gamma d_a^2}{(\gamma d_a^2-\xi)^2}G_{ip}G_{pi}, \quad  B_{112}:=\frac{1}{N}\sum_{p}\frac{-\gamma d_a^2}{(\gamma d_a^2-\xi)^2}\frac{G_{i\bba}G_{\bba p}G_{pi}}{G_{\bba\bba}},\\
&B_{113}:=\frac{1}{N}\sum_{p} \frac{-\gamma d_a^2}{(\gamma d_a^2-\xi)^2}\frac{G_{ia}^{(\bba)}G_{ap}^{(\bba)}G_{pi}}{G_{aa}^{(\bba)}}, \quad B_{114}:=\frac{1}{N}\sum_{p}\frac{-\gamma d_a^2}{(\gamma d_a^2-\xi)^2}\frac{G_{i\bba}G_{\bba p}G_{pi}^{(a\bba)}}{G_{\bba\bba}},\\
&B_{115}:=\frac{1}{N}\sum_{p} \frac{-\gamma d_a^2}{(\gamma d_a^2-\xi)^2}\frac{G_{ia}^{(\bba)}G_{ap}^{(\bba)}G_{pi}^{(a\bba)}}{G_{aa}^{(\bba)}}.\\
\end{aligned}\end{eqnarray}
We have changed summation index sets by removing the restrictions $p\neq a$. This is feasible because the difference term is of  negligible order $O(\Psi^5).$ Notice that in $B_{114}$ and $B_{115}$, the summation index $p$ cannot be $a$ since $a$ appears as an upper index. However, we use the full summation silently and similarly in the following for notational simplicity.%For example, in defining $B_{111}$, since \begin{eqnarray}\frac{1}{N}\sum_{a,p} \frac{\gamma d_a^2}{(\gamma d_a^2-\xi)^2}G_{ip}G_{pi}=\frac{1}{N}\sum_{\substack{a,p\\ a\neq i,  p\neq a}} \frac{\gamma d_a^2}{(\gamma d_a^2-\xi)^2}G_{ip}G_{pi}+O_\prec (\Psi^2),\end{eqnarray}

%We decompose $E_{\bba} B_{11}$ into summation of following terms:\begin{eqnarray}\begin{aligned}&B_{111}:=\frac{1}{N}\sum_{a,p} \frac{\gamma d_a^2}{(\gamma d_a^2-\xi)^2}G_{ip}G_{pi}, \quad  B_{112}:=\frac{1}{N}\sum_{a,p}\frac{-2\gamma d_a^2}{(\gamma d_a^2-\xi)^2}\frac{G_{ip}G_{i\bba}G_{\bba p}}{G_{\bba\bba}}\\&B_{113}:=\frac{1}{N}\sum_{a,p} \frac{-2\gamma d_a^2}{(\gamma d_a^2-\xi)^2}\frac{G_{ip}G_{ia}^{(\bba)}G_{ap}^{(\bba)}}{G_{aa}^{(\bba)}}, \quad B_{114}:=\frac{1}{N}\sum_{a,p} \frac{\gamma d_a^2}{(\gamma d_a^2-\xi)^2}\frac{G_{i\bba}G_{\bba p}G_{i\bba}G_{\bba p}}{G_{\bba\bba}^2}\\&B_{115} =\frac{1}{N}\sum_{a,p} \frac{\gamma d_a^2}{(\gamma d_a^2-\xi)^2}\frac{G_{ia}^{(\bba)}G_{ap}^{(\bba)}G_{ia}^{(\bba)}G_{ap}^{(\bba)}}{G_{aa}^{(\bba)2}}, B_{116}:=\frac{1}{N}\sum_{a,p}\frac{2\gamma d_a^2}{(\gamma d_a^2-\xi)^2}\frac{G_{i\bba}G_{\bba p}G_{ia}^{(\bba)}G_{ap}^{(\bba)}}{G_{\bba\bba}G_{aa}^{(\bba)}}\end{aligned}\end{eqnarray}

Notice that in $B_{111}$, the index $a$ has been decoupled from resolvent entries. Thus there is no need to further expand $B_{111}$. Next we turn to $B_{114}$, which is simpler than $B_{112}$.
%Using \eqref{gfiimu}, we find \begin{eqnarray}\label{B112}\begin{aligned}B_{112}&=\frac{1}{N}\sum_{a,p,r,s} \frac{-2\gamma d_a^2}{(\gamma d_a^2-\xi)^2}G_{ip}G_{\bba\bba}G_{ir}^{(\bba)}Y_{r\bba}Y_{s\bba}G_{ps}^{(\bba)}\\&=\frac{1}{N}\sum_{a,p} \frac{-2\gamma^2 d_a^2}{(\gamma d_a^2-\xi)^2}G_{\bba\bba}G_{ip}^{(\bba)}G_{ia}^{(\bba)}d_a^2 G_{pa}^{(\bba)}+\frac{1}{N}\sum_{a,p,r} \frac{-2\gamma^2 d_a^2}{(\gamma d_a^2-\xi)^2}G_{\bba\bba}G_{ip}^{(\bba)}G_{ir}^{(\bba)}x_{r\bba} d_a G_{pa}^{(\bba)}\\&\relphantom{EE}+\frac{1}{N}\sum_{a,p,s} \frac{-2\gamma^2 d_a^2}{(\gamma d_a^2-\xi)^2}G_{\bba\bba}G_{ia}^{(\bba)}G_{ia}^{(\bba)}d_a x_{s\bba}G_{ps}^{(\bba)}+\frac{1}{N}\sum_{a,p,r,s} \frac{-2\gamma^2 d_a^2}{(\gamma d_a^2-\xi)^2}G_{\bba\bba}G_{ia}^{(\bba)}G_{ir}^{(\bba)}x_{r\bba} x_{s\bba}G_{ps}^{(\bba)}\\&\relphantom{EE}+\frac{1}{N}\sum_{a,p,r,s} \frac{-2\gamma^2 d_a^2}{(\gamma d_a^2-\xi)^2}G_{i\bba}G_{\bba p}G_{ir}^{(\bba)}Y_{r\bba}Y_{s\bba}G_{ps}^{(\bba)}\end{aligned}\end{eqnarray}
Using \eqref{gfiimu}, we find \begin{eqnarray}\label{B114}\begin{aligned}
B_{114}&=\frac{1}{N}\sum_{p,r,s} \frac{-\gamma d_a^2}{(\gamma d_a^2-\xi)^2}G_{\bba\bba}G_{ir}^{(\bba)}Y_{r\bba}Y_{s\bba}G_{sp}^{(\bba)}G_{pi}^{(a\bba)}\\
&=\frac{1}{N}\sum_{p} \frac{-\gamma^2 d_a^2}{(\gamma d_a^2-\xi)^2}G_{\bba\bba}G_{ia}^{(\bba)}d_a^2 G_{ap}^{(\bba)}G_{pi}^{(a\bba)}+\frac{1}{N}\sum_{p,r} \frac{-\gamma^2 d_a^2}{(\gamma d_a^2-\xi)^2}G_{\bba\bba}G_{ir}^{(\bba)}x_{r\bba} d_a G_{pa}^{(\bba)}G_{pi}^{(a\bba)}\\
&\relphantom{EE}+\frac{1}{N}\sum_{p,s} \frac{-\gamma^2 d_a^2}{(\gamma d_a^2-\xi)^2}G_{\bba\bba}G_{ia}^{(\bba)}d_a x_{s\bba}G_{ps}^{(\bba)}G_{pi}^{(a\bba)}+\frac{1}{N}\sum_{p,r,s} \frac{-\gamma^2 d_a^2}{(\gamma d_a^2-\xi)^2}G_{\bba\bba}G_{ir}^{(\bba)}x_{r\bba} x_{s\bba}G_{ps}^{(\bba)}G_{pi}^{(a\bba)}.
%&\relphantom{EE}+\frac{1}{N}\sum_{a,p,r,s} \frac{-2\gamma^2 d_a^2}{(\gamma d_a^2-\xi)^2}G_{i\bba}G_{\bba p}G_{ir}^{(\bba)}Y_{r\bba}Y_{s\bba}G_{ps}^{(\bba)}
\end{aligned}\end{eqnarray}
%For the first term in the right side of \eqref{B114}, we find \begin{eqnarray}\label{B1121}\begin{aligned}&\frac{1}{N}\sum_{a,p} \frac{-2\gamma^2 d_a^2}{(\gamma d_a^2-\xi)^2}G_{\bba\bba}G_{ip}^{(\bba)}G_{ia}^{(\bba)}d_a^2 G_{pa}^{(\bba)}\\&=\frac{1}{N}\sum_{a,p} \frac{-2\gamma^3 d_a^4}{(\gamma d_a^2-\xi)^2}G_{\bba\bba}G_{aa}^{(\bba)2}G_{ip}^{(\bba)}\sum_{\mu,\nu}^{(\bba)}G_{i\mu}^{(a \bba)}x_{a\mu}x_{a\nu} G_{p\nu}^{(a \bba)}\\&=\frac{1}{N}\sum_{a,p}\frac{-2\gamma^3 d_a^4}{(\gamma d_a^2-\xi)^2}\left[ \frac{1}{(\gamma d_a^2-\xi)\tb}+\frac{\gamma d_a^2-2\xi}{(\gamma d_a^2-\xi)^2\tb^2}\mU_a-\frac{1}{(\gamma d_a^2-\xi)^2}\mV_a\right]\sum_{\mu,\nu}^{(\bba)}G_{ip}^{(\bba)}G_{i\mu}^{(a \bba)}x_{a\mu}x_{a\nu} G_{p\nu}^{(a \bba)}\end{aligned}\end{eqnarray}
By \eqref{gfiij}, \eqref{taylorofGbba}, \eqref{taylorofalphaa} and \eqref{mUVorder} we find that the first term in the right side of \eqref{B114} is \begin{eqnarray}\label{B1141}
\begin{aligned}
&\frac{1}{N}\sum_{p} \frac{-\gamma^2 d_a^2}{(\gamma d_a^2-\xi)^2}G_{\bba\bba}G_{ia}^{(\bba)}d_a^2 G_{ap}^{(\bba)}G_{pi}^{(a\bba)}\\
&=\frac{1}{N}\sum_{p} \frac{-\gamma^3 d_a^4}{(\gamma d_a^2-\xi)^2}G_{\bba\bba}G_{aa}^{(\bba)2}\sum_{\mu,\nu}^{(\bba)}G_{i\mu}^{(a \bba)}x_{a\mu}x_{a\nu} G_{p\nu}^{(a \bba)}G_{pi}^{(a\bba)}\\
&=\frac{1}{N}\sum_{p,\mu,\nu}^{(\bba)}\frac{-\gamma^3 d_a^4}{(\gamma d_a^2-\xi)^2}\left[ \frac{1}{(\gamma d_a^2-\xi)\tb}+\frac{(\gamma d_a^2-2\xi)\mU_a}{(\gamma d_a^2-\xi)^2\tb^2}-\frac{\mV_a}{(\gamma d_a^2-\xi)^2}
\right]G_{i\mu}^{(a \bba)}x_{a\mu}x_{a\nu} G_{p\nu}^{(a \bba)}G_{pi}^{(a\bba)}+O(\Psi^5).
\end{aligned}
\end{eqnarray}
Then we handle the three terms in the last step of \eqref{B1141} separately.
For the first term, by taking expectation, we get  \begin{eqnarray}\label{bal3g}
\begin{aligned}
&\E \frac{1}{N}\sum_{p,\mu,\nu}\frac{-\gamma^3 d_a^4}{(\gamma d_a^2-\xi)^3 \tb}G_{i\mu}^{(a \bba)}x_{a\mu}x_{a\nu} G_{p\nu}^{(a \bba)} G_{pi}^{(a\bba)} \\
&=  \E \frac{1}{N^2}\sum_{p,\mu}\frac{-\gamma^3 d_a^4}{(\gamma d_a^2-\xi)^3 \tb}G_{i\mu}^{(a\bba)}G_{\mu p}^{(a \bba)}G_{p i}^{(a \bba)} \\
&=\frac{-\gamma^3 d_a^4}{(\gamma d_a^2-\xi)^3 \tb}\E X_{33}'+ \frac{ 3\gamma^4 d_a^4 E_+(E_+ b+\tb)}{(\gamma d_a^2-\xi)^4\tb}\E X_{44}+O(\Psi^5).
\end{aligned}
\end{eqnarray}
where in the last step, we use the second equation in Lemma \ref{kao4o}.
%We denote \begin{eqnarray}\begin{aligned}&h_1:=\frac{1}{N}\sum_a \frac{\gamma^4 d_a^4 E_+}{(\gamma d_a^2-\xi)^4},\quad h_2:= \frac{1}{N}\sum_a \frac{\gamma^4 d_a^4 b E_+^2}{(\gamma d_a^2-\xi)^4\tb}, \quad h_3:=\frac{1}{N}\sum_a\frac{\gamma^3 d_a^2 \tb^2}{(\gamma d_a^2-\xi)^4}, \\&h_4:= \frac{1}{N}\sum_a \frac{\gamma^4 d_a^4 E_+^2}{(\gamma d_a^2-\xi)^3\tb^2}, \quad  h_5:=\frac{1}{N}\sum_a\frac{\gamma^3 d_a^2 E_+}{(\gamma d_a^2-\xi)^3}, \quad h_6:=\frac{1}{N}\sum_a \frac{\gamma^3 d_a^2 b E_+^2}{(\gamma d_a^2-\xi)^3 \tb},\\\end{aligned}\end{eqnarray}
%Then the expectation of the first term on the right side of \eqref{B1121} is \begin{eqnarray}\begin{aligned}
%&\E\frac{1}{N}\sum_{a,p}\frac{-2\gamma^3 d_a^4}{(\gamma d_a^2-\xi)^3\tb} \sum_{\mu,\nu}^{(\bba)}G_{ip}^{(\bba)}G_{i\mu}^{(a \bba)}x_{a\mu}x_{a\nu} G_{p\nu}^{(a \bba)}\\&=\frac{1}{N}\sum_{a}\frac{-2\gamma^3 d_a^4}{(\gamma d_a^2-\xi)^3 \tb}\E X_{33}' +  (6h_1+6h_2+4h_4)\E X_{44}+2h_4\E X_{44}'\end{aligned}\end{eqnarray}
We find the second term in \eqref{B1141} equals \begin{eqnarray}\label{03tm8}\begin{aligned}
&\frac{1}{N}\sum_{p}\frac{-\gamma^3 d_a^4(\gamma d_a^2-2\xi)}{(\gamma d_a^2-\xi)^4\tb^2}\left[E_+-z+\gamma \left( E_+\tm -\sum x_{a\alpha}G_{\alpha\beta}^{(a\bba)}G_{a\beta}\right)\right]\sum_{\mu,\nu}^{(\bba)}G_{ip}^{(\bba)}G_{i\mu}^{(a \bba)}x_{a\mu}x_{a\nu} G_{p\nu}^{(a \bba)}\\
&=\frac{1}{N}\sum_{p}\frac{-\gamma^3 d_a^4(\gamma d_a^2-2\xi)}{(\gamma d_a^2-\xi)^4\tb^2}\gamma \left( E_+\tm -\sum x_{a\alpha}G_{\alpha\beta}^{(a\bba)}G_{a\beta}\right)\sum_{\mu,\nu}^{(\bba)}G_{ip}^{(a\bba)}G_{i\mu}^{(a \bba)}x_{a\mu}x_{a\nu} G_{p\nu}^{(a \bba)}+O(\Psi^5).
\end{aligned}\end{eqnarray}
 Using \eqref{relofmtm}, \eqref{EX44} and \eqref{EX43}, its expectation equals
\begin{eqnarray}\label{h9agm}
\left(\frac{-\gamma^4 d_a^4 E_+^2}{(\gamma d_a^2-\xi)^3\tb^2}+\frac{\gamma^4 d_a^4 b E_+^2}{(\gamma d_a^2-\xi)^4 \tb}\right)(\E X_{43}-2\E X_{44})+O(\Psi^5).
\end{eqnarray}
Here we address that we will use Lemma \ref{multiErel} frequently later to express the terms by using $\E X_{42}$,$\E X_{43}$ and $\E X_{44}$ instead of their variations defined in \eqref{Xset2}.
The last term in \eqref{B1141} equals \begin{eqnarray}\label{h9agn}
\frac{\gamma^4 d_a^4 E_+}{(\gamma d_a^2-\xi)^4} \E X_{43}+O(\Psi^5)
\end{eqnarray}
Combing \eqref{bal3g}, \eqref{h9agm} and \eqref{h9agn}, the expectation of the first term in \eqref{B114} is
\begin{eqnarray}\label{B1141v}\begin{aligned}
&\E\frac{1}{N}\sum_{p} \frac{-\gamma^2 d_a^2}{(\gamma d_a^2-\xi)^2}G_{\bba\bba}G_{ia}^{(\bba)}d_a^2 G_{ap}^{(\bba)}G_{pi}^{(a\bba)}\\
=&\frac{-\gamma^3 d_a^4}{(\gamma d_a^2-\xi)^3 \tb}\E X_{33}'+\left(\frac{-\gamma^4 d_a^4 E_+^2}{(\gamma d_a^2-\xi)^3\tb^2}+\frac{\gamma^4 d_a^4 b E_+^2}{(\gamma d_a^2-\xi)^4 \tb}+\frac{\gamma^4 d_a^4 E_+}{(\gamma d_a^2-\xi)^4}\right)\E X_{43}\\&\relphantom{EEEE}+\left(\frac{ \gamma^4 d_a^4 b E_+^2}{(\gamma d_a^2-\xi)^4\tb} + \frac{ 3\gamma^4 d_a^4  E_+}{(\gamma d_a^2-\xi)^4} + \frac{2\gamma^4 d_a^4 E_+^2}{(\gamma d_a^2-\xi)^3\tb^2}\right)\E X_{44}+O(\Psi^5).
\end{aligned}\end{eqnarray}

For the second term in \eqref{B114}, expanding the $G_{\bba\bba}$ by using \eqref{taylorofGbba} and \eqref{taylorofalphaa}, we get\begin{eqnarray}\label{B1142}\begin{aligned}
& \frac{1}{N}\sum_{p,r} \frac{-\gamma^2 d_a^2}{(\gamma d_a^2-\xi)^2}G_{\bba\bba}G_{ir}^{(\bba)}x_{r\bba} d_a G_{pa}^{(\bba)}G_{pi}^{(a\bba)}\\
&= \frac{1}{N}\sum_{p,r} \frac{-\gamma^2 d_a^2}{(\gamma d_a^2-\xi)^2} \left[\frac{\tb}{\gamma d_a^2-\xi} -\frac{\gamma d_a^2 }{(\gamma d_a^2-\xi)^2}\mU_a-\frac{\tb^2}{(\gamma d_a^2-\xi)^2}\mV_a \right]G_{ir}^{(\bba)}x_{r\bba} d_a G_{pa}^{(\bba)}G_{pi}^{(a\bba)}+O(\Psi^5).
\end{aligned}\end{eqnarray}
The expectations of the first and second term in \eqref{B1142} vanish by taking partial expectation $\E_{\bba}$ first.
For the third term above, after taking expectation, we obtain
\begin{eqnarray}\begin{aligned}
&\E \frac{1}{N}\sum_{p,r} \frac{\gamma^3 d_a^2 \tb^2}{(\gamma d_a^2-\xi)^4}\left(cm-\sum x_{j\bba}G_{jk}^{(a\bba)}x_{k\bba}-2d_a \sum G_{aj}^{(\bba)}x_{j\bba}\right)G_{ir}^{(\bba)}x_{r\bba} d_a G_{pa}^{(\bba)}G_{pi}^{(a\bba)}\\
&=\frac{1}{N^{5/2}}\sum_{p,r} \frac{-2\gamma^3 d_a^3 \tb^2}{(\gamma d_a^2-\xi)^4}\E x_{r\bba}^3\E G_{ir}^{(\bba)}G_{rr}^{(a\bba)}G_{pa}^{(\bba)}G_{pi}^{(a\bba)}+\frac{1}{N^2}\sum_{p,r} \frac{-2\gamma^3 d_a^4 \tb^2}{(\gamma d_a^2-\xi)^4}\E G_{ir}^{(\bba)}G_{ra}^{(\bba)}G_{ap}^{(\bba)}G_{pi}^{(a\bba)}.
\end{aligned}\end{eqnarray}
The first term above is of order $O(N^{-1/2}\Psi^3)$, which is negligible. For the second term, by expanding Green entries that have lower index $a$, we obtain that \begin{eqnarray}\label{xyeig}\begin{aligned}
&\frac{1}{N^2}\sum_{p,r} \frac{-2\gamma^3 d_a^2 \tb^2}{(\gamma d_a^4-\xi)^4}\E G_{ir}^{(\bba)}G_{ra}^{(\bba)}G_{ap}^{(\bba)}G_{pi}^{(a\bba)}\\&=\frac{1}{N^2}\sum_{p,r,\mu,\nu} \frac{-2\gamma^4 d_a^4 \tb^2}{(\gamma d_a^2-\xi)^4}\E G_{ir}^{(\bba)} G_{aa}^{(\bba)2}G_{r\mu}^{(a\bba)}x_{a\mu}x_{a\nu}G_{p\nu}^{(a\bba)}G_{pi}^{(a\bba)}\\
&=\frac{1}{N^3}\sum_{p,r,\mu} \frac{-2\gamma^4 d_a^4 }{(\gamma d_a^2-\xi)^4}\E G_{ir}^{(a\bba)}G_{r\mu}^{(a\bba)}G_{\mu p}^{(a\bba)}G_{pi}^{(a\bba)}\\
&= \frac{-2\gamma^4 d_a^4 E_+ }{(\gamma d_a^2-\xi)^4}\E X_{44}+O(\Psi^5).
\end{aligned}\end{eqnarray}
Therefore, we conclude that the expectation of the second term in \eqref{B114} is \begin{eqnarray}\label{B1142v}
 \frac{-2\gamma^4 d_a^4 E_+ }{(\gamma d_a^2-\xi)^4}\E X_{44}+O(N^{-1/2}\Psi^3).
\end{eqnarray}

For the third term in \eqref{B114}, the estimation is essentially the same as the argument from \eqref{B1142} to \eqref{xyeig}, and we obtains that \begin{eqnarray}\label{B1143v}
\E \frac{1}{N}\sum_{p,s} \frac{-\gamma^2 d_a^2}{(\gamma d_a^2-\xi)^2}G_{\bba\bba}G_{ia}^{(\bba)}d_a x_{s\bba}G_{ps}^{(\bba)}G_{pi}^{(a\bba)}= \frac{-2\gamma^4 d_a^4 E_+ }{(\gamma d_a^2-\xi)^4}\E X_{44}+O(\Psi^5).
\end{eqnarray}

For the last term in \eqref{B114}, we have
\begin{eqnarray}\label{B1144}\begin{aligned}
&\frac{1}{N}\sum_{p,r,s} \frac{-\gamma^2 d_a^2}{(\gamma d_a^2-\xi)^2}G_{\bba\bba}G_{ir}^{(\bba)}x_{r\bba} x_{s\bba}G_{ps}^{(\bba)}G_{pi}^{(a\bba)}\\
&=\frac{1}{N}\sum_{p,r,s} \frac{-\gamma^2 d_a^2}{(\gamma d_a^2-\xi)^2}\left[\frac{\tb}{\gamma d_a^2-\xi} -\frac{\gamma d_a^2 }{(\gamma d_a^2-\xi)^2}\mU_a-\frac{\tb^2}{(\gamma d_a^2-\xi)^2}\mV_a \right]G_{ir}^{(\bba)}x_{r\bba} x_{s\bba}G_{ps}^{(\bba)}G_{pi}^{(a\bba)}+O(\Psi^5).
\end{aligned}\end{eqnarray}
The expectation of the first term of the right side of  \eqref{B1144} is \begin{eqnarray}\label{B11441v}
\begin{aligned}
&\frac{1}{N^2}\sum_{p,r} \frac{-\gamma^2 d_a^2 \tb}{(\gamma d_a^2-\xi)^3}\E G_{ir}^{(\bba)}G_{rp}^{(\bba)}G_{pi}^{(a\bba)}\\
&=\frac{1}{N^2}\sum_{p,r} \frac{-\gamma^2 d_a^2 \tb}{(\gamma d_a^2-\xi)^3}\E G_{ir}^{(\bba)}G_{rp}^{(\bba)}G_{pi}^{(\bba)}+\frac{1}{N}\sum_{p,r} \frac{\gamma^2 d_a^2 \tb}{(\gamma d_a^2-\xi)^3}\E G_{ir}^{(\bba)}G_{rp}^{(\bba)}\frac{G_{pa}^{(\bba)}G_{ai}^{(\bba)}}{G_{aa}^{(\bba)}}\\
&= \frac{-\gamma^2 d_a^2 \tb}{(\gamma d_a^2-\xi)^3}\E X_{33}+3\left(\frac{\gamma^4 d_a^4 E_+}{(\gamma d_a^2-\xi)^4}+ \frac{\gamma^3 d_a^2\tb^2}{(\gamma d_a^2-\xi)^4}\right) \E X_{44}-\frac{\gamma^3 d_a^2 E_+}{(\gamma d_a^2-\xi)^3}\E X_{44}+O(\Psi^5).
\end{aligned}
\end{eqnarray}
where in the second equality, we use the fact that
  \begin{eqnarray}\label{nn3ca}
\E\frac{1}{N^2}\sum_{j,k} G_{ij}^{(\bba)}G_{jk}^{(\bba)}G_{ki}^{(\bba)}=\E X_{33}-3\left(\frac{\gamma^2 d_a^2 E_+}{(\gamma d_a^2-\xi)\tb}+ \frac{\gamma \tb}{\gamma d_a^2-\xi}\right) \E X_{44}+O(\Psi^5).
\end{eqnarray}
To derive \eqref{nn3ca}, we write\begin{eqnarray}\label{0agj4}\begin{aligned}
&\E\frac{1}{N^2}\sum G_{ij}^{(\bba)}G_{jk}^{(\bba)}G_{ki}^{(\bba)}=\E\frac{1}{N^2}\sum \left(G_{ij}-\frac{G_{i\bba}G_{\bba j}}{G_{\bba\bba}}\right) G_{jk}G_{ki}\\&\relphantom{EEEEEEE}-\E\frac{1}{N^2}\sum G_{ij}^{(\bba)}\frac{G_{j\bba}G_{\bba k}}{G_{\bba\bba}}G_{ki}-\E\frac{1}{N^2}\sum G_{ij}^{(\bba)}G_{jk}^{(\bba)}\frac{G_{ka}G_{\bba i}}{G_{\bba\bba}}.
\end{aligned}\end{eqnarray}
Similar to \eqref{89afl} below, we calculate \begin{eqnarray}\begin{aligned}
\E\frac{1}{N^2}\sum \frac{G_{i\bba}G_{\bba j}}{G_{\bba\bba}} G_{jk}G_{ki}=\frac{\gamma^2 d_a^2 E_+}{(\gamma d_a^2-\xi)\tb}\E X_{44}+ \frac{\gamma \tb}{\gamma d_a^2-\xi}\E X_{44}+O(\Psi^5).
\end{aligned}\end{eqnarray}
The same estimation also hold for the second and third term in \eqref{0agj4}. Therefore, \eqref{nn3ca} holds true.\\
\indent The expectation of the second term of the right side of \eqref{B1144}
is \begin{eqnarray}\label{B11442v}
\E\frac{1}{N}\sum_{p,r,s} \frac{\gamma^3 d_a^4}{(\gamma d_a^2-\xi)^4}\mU_aG_{ir}^{(\bba)}x_{r\bba} x_{s\bba}G_{ps}^{(\bba)}G_{pi}^{(a\bba)}=\frac{\gamma^4 d_a^4 E_+}{(\gamma d_a^2-\xi)^4}\E X_{43}+O(\Psi^5).
\end{eqnarray}
And the expectation of the third term of the right side of \eqref{B1144} is
\begin{eqnarray}\label{B11443v}\begin{aligned}
&\E \frac{1}{N}\sum_{p,r,s} \frac{\gamma^2 d_a^2 \tb^2}{(\gamma d_a^2-\xi)^4}\mV_a G_{ir}^{(\bba)}x_{r\bba} x_{s\bba}G_{ps}^{(\bba)}G_{pi}^{(a\bba)}\\& =-\frac{2\gamma^3 d_a^2 \tb^2}{(\gamma d_a^2-\xi)^4}\E X_{44}+\frac{\gamma^3 d_a^2 \tb^2}{(\gamma d_a^2-\xi)^4}\E X_{43}+O(N^{-1/2}\Psi^3).
\end{aligned}\end{eqnarray}
Then, combining \eqref{B11441v},\eqref{B11442v} and \eqref{B11443v} together, we have that the expectation of the last term of \eqref{B114} is \begin{eqnarray}\label{B1144v}\begin{aligned}
&\E\frac{1}{N}\sum_{p,r,s} \frac{-\gamma^2 d_a^2}{(\gamma d_a^2-\xi)^2}G_{\bba\bba}G_{ir}^{(\bba)}x_{r\bba} x_{s\bba}G_{ps}^{(\bba)}G_{pi}^{(a\bba)}\\
&=\frac{-\gamma^2 d_a^2 \tb}{(\gamma d_a^2-\xi)^3}\E X_{33}+\left(\frac{3\gamma^4 d_a^4 E_+}{(\gamma d_a^2-\xi)^4}+ \frac{\gamma^3 d_a^2\tb^2}{(\gamma d_a^2-\xi)^4}-\frac{\gamma^3 d_a^2 E_+}{(\gamma d_a^2-\xi)^3}\right)\E X_{44}\\
&\relphantom{EEE}+\left(\frac{\gamma^3 d_a^2 \tb^2}{(\gamma d_a^2-\xi)^4}+\frac{\gamma^4 d_a^4 E_+}{(\gamma d_a^2-\xi)^4}\right)\E X_{43}+O(N^{-1/2}\Psi^3).
\end{aligned}\end{eqnarray}
Therefore, combining \eqref{B1141v},\eqref{B1142v},\eqref{B1143v} and \eqref{B1144v}, we have  \begin{eqnarray}\label{B114v}\begin{aligned}
&\E B_{114}=\frac{-\gamma^2 d_a^2 \tb}{(\gamma d_a^2-\xi)^3}\E X_{33}+\frac{-\gamma^3 d_a^4}{(\gamma d_a^2-\xi)^3 \tb}\E X_{33}'+\left(\frac{-\gamma^4 d_a^4 E_+^2}{(\gamma d_a^2-\xi)^3\tb^2}+\frac{\gamma^4 d_a^4 b E_+^2}{(\gamma d_a^2-\xi)^4 \tb}+\frac{2\gamma^4 d_a^4 E_+}{(\gamma d_a^2-\xi)^4}+\frac{\gamma^3 d_a^2 \tb^2}{(\gamma d_a^2-\xi)^4}\right)\E X_{43}\\&\relphantom{EEEE}+\left(\frac{ \gamma^4 d_a^4 b E_+^2}{(\gamma d_a^2-\xi)^4\tb} + \frac{ 2\gamma^4 d_a^4  E_+}{(\gamma d_a^2-\xi)^4} + \frac{2\gamma^4 d_a^4 E_+^2}{(\gamma d_a^2-\xi)^3\tb^2}+\frac{\gamma^3 d_a^2\tb^2}{(\gamma d_a^2-\xi)^4}-\frac{\gamma^3 d_a^2 E_+}{(\gamma d_a^2-\xi)^3}\right)\E X_{44}+O(N^{-1/2}\Psi^3).
\end{aligned}\end{eqnarray}

Then we consider $B_{112}$ in \eqref{ExpB11}. Using \eqref{gfiminus}, we find
\begin{eqnarray}\label{B112}
\E B_{112}=\E B_{114}+\frac{1}{N}\sum_{p}\frac{-\gamma d_a^2}{(\gamma d_a^2-\xi)^2}\frac{G_{i\bba}G_{\bba p}G_{p\bba}G_{\bba i}}{G_{\bba\bba}^2}+\frac{1}{N}\sum_{p}\frac{-\gamma d_a^2}{(\gamma d_a^2-\xi)^2}\frac{G_{i\bba}G_{\bba p}G_{pa}^{(\bba)}G_{a i}^{(\bba)}}{G_{\bba\bba}G_{aa}^{(\bba)}}
\end{eqnarray}
To calculate the expectation of the second term on the right side of \eqref{B112}, we consider
\begin{eqnarray}\label{miy11}\begin{aligned}
&\E \frac{1}{N}\sum_{p}\frac{G_{i\bba}G_{\bba p}G_{p\bba}G_{\bba i}}{G_{\bba\bba}^2}\\&=\E\frac{1}{N}\sum_{p}G_{\bba\bba}^2\gamma^2 \left(G_{ia}^{(\bba)}d_a^2 G_{pa}^{(\bba)}+\sum G_{ir}^{(\bba)}x_{r\bba}d_a G_{pa}^{(\bba)}+\sum G_{ia}^{(\bba)}d_a x_{s\bba}G_{ps}^{(\bba)}+\sum G_{ir}^{(\bba)}x_{r\bba}x_{s\bba}G_{ps}^{(\bba)}\right)^2
\end{aligned}\end{eqnarray}
By expanding the square, we find that
\begin{eqnarray}\label{miy12}\begin{aligned}
&\E \frac{\gamma^2}{N}\sum_{p}G_{\bba\bba}^2G_{ia}^{(\bba)}d_a^2 G_{pa}^{(\bba)}G_{ia}^{(\bba)}d_a^2 G_{pa}^{(\bba)}= \frac{\gamma^4 d_a^4 E_+^2}{(\gamma d_a^2-\xi)^2 \tb^2}(2\E X_{44}+\E \tX_{44}),\\
&\E \frac{\gamma^2}{N}\sum_{p}G_{\bba\bba}^2 \sum_{r,s} G_{ir}^{(\bba)}x_{r\bba}d_a G_{pa}^{(\bba)}G_{is}^{(\bba)}x_{s\bba}d_a G_{pa}^{(\bba)}=\frac{\gamma^3 d_a^2 E_+}{(\gamma d_a^2-\xi)^2}\E \tX_{44},\\
&\E \frac{\gamma^2}{N}\sum_{p}G_{\bba\bba}^2\sum_{r,s} G_{ia}^{(\bba)}d_a x_{r\bba}G_{pr}^{(\bba)}G_{ia}^{(\bba)}d_a x_{s\bba}G_{ps}^{(\bba)}=\frac{\gamma^3 d_a^2 E_+}{(\gamma d_a^2-\xi)^2}\E \tX_{44},\\
&\E \frac{\gamma^2}{N}\sum_{p}G_{\bba\bba}^2\sum_{r,s,j,k} G_{ir}^{(\bba)}x_{r\bba}x_{s\bba}G_{ps}^{(\bba)}G_{ij}^{(\bba)}x_{j\bba}x_{k\bba}G_{pk}^{(\bba)}=\frac{\gamma^2 \tb^2 }{(\gamma d_a^2-\xi)^2}(2\E X_{44}+\E \tX_{44}),\\
&\E \frac{2\gamma^2}{N}\sum_{p}G_{\bba\bba}^2G_{ia}^{(\bba)}d_a^2 G_{pa}^{(\bba)}\sum G_{ir}^{(\bba)}x_{r\bba}x_{s\bba}G_{ps}^{(\bba)}=\frac{2\gamma^3 d_a^2 E_+}{(\gamma d_a^2-\xi)^2}\E X_{44}\\
&\E \frac{2\gamma^2}{N}\sum_{p}G_{\bba\bba}^2\sum G_{ir}^{(\bba)}x_{r\bba}d_a G_{pa}^{(\bba)} G_{ia}^{(\bba)}d_a x_{s\bba}G_{ps}^{(\bba)}=\frac{2\gamma^3 d_a^2 E_+}{(\gamma d_a^2-\xi)^2}\E X_{44}.
\end{aligned}\end{eqnarray}
And for the other terms that contain one or three  entries of $X$ have negligible expectation. It follows that \begin{eqnarray}\label{miy13}\begin{aligned}
&\E \frac{1}{N}\sum_{p}\frac{-\gamma d_a^2}{(\gamma d_a^2-\xi)^2}\frac{G_{i\bba}G_{\bba p}G_{p\bba}G_{\bba i}}{G_{\bba\bba}^2}\\
&=\left(\frac{-\gamma^5 d_a^4 E_+^2}{(\gamma d_a^2-\xi)^4 \tb^2}+\frac{-\gamma^3 d_a^2\tb^2}{(\gamma d_a^2-\xi)^4}+\frac{-2\gamma^4 d_a^4 E_+}{(\gamma d_a^2-\xi)^4}\right)(2 \E X_{44}+\E \tX_{44})+O(N^{-1/2}\Psi^3).
\end{aligned}\end{eqnarray}
The expectation of the last term in \eqref{B112} equals \begin{eqnarray}\label{cmle4}\begin{aligned}
&\E\frac{1}{N}\sum_{p}\frac{-\gamma d_a^2}{(\gamma d_a^2-\xi)^2}\frac{G_{i\bba}G_{\bba p}G_{pa}^{(\bba)}G_{a i}^{(\bba)}}{G_{\bba\bba}G_{aa}^{(\bba)}}\\
&=\E\frac{1}{N}\sum\frac{-\gamma^2 d_a^2}{(\gamma d_a^2-\xi)^2}G_{\bba\bba}G_{aa}^{(\bba)}G_{ir}^{(\bba)}Y_{r\bba}Y_{s\bba}G_{ps}^{(\bba)}G_{p\mu}^{(a\bba)}x_{a\mu}x_{a\nu}G_{\nu i}^{(a\bba)}\\
&=\left(\frac{2\gamma^4 d_a^4 E_+^2}{(\gamma d_a^2-\xi)^3\tb^2}+\frac{\gamma^3 d_a^2 E_+}{(\gamma d_a^2-\xi)^3}\right) \E X_{44}+\frac{\gamma^4 d_a^4 E_+^2}{(\gamma d_a^2-\xi)^3\tb^2}\E \tX_{44}+O(\Psi^5).
\end{aligned}
\end{eqnarray}
Therefore we have  \begin{eqnarray}\label{B112v}\begin{aligned}
\E B_{112}&=\E B_{114}+\left(\frac{-2\gamma^4 d_a^4 b E_+^2}{(\gamma d_a^2-\xi)^4 \tb}+\frac{-2\gamma^3 d_a^2\tb^2}{(\gamma d_a^2-\xi)^4}+\frac{-4\gamma^4 d_a^4 E_+}{(\gamma d_a^2-\xi)^4}+\frac{\gamma^3 d_a^2 E_+}{(\gamma d_a^2-\xi)^3} \right)\E X_{44}\\
&\relphantom{EEEE}+\left(\frac{-\gamma^4 d_a^4 b E_+^2}{(\gamma d_a^2-\xi)^4 \tb}+\frac{-\gamma^3 d_a^2\tb^2}{(\gamma d_a^2-\xi)^4}+\frac{-2\gamma^4 d_a^4 E_+}{(\gamma d_a^2-\xi)^4} \right)\E \tX_{44}+O(N^{-1/2}\Psi^3).
\end{aligned}
\end{eqnarray}

\vspace{1cm}
We now turn to $B_{115}$ in \eqref{ExpB11}, which is similar to the first term in \eqref{B114}.  Write \begin{eqnarray}
B_{115}=\frac{1}{N}\sum_{p} \frac{-\gamma^2 d_a^2}{(\gamma d_a^2-\xi)^2}G_{aa}^{(\bba)}\sum_{\mu,\nu}^{(\bba)}G_{i\mu}^{(a \bba)}x_{a\mu}x_{a\nu} G_{p\nu}^{(a \bba)}G_{pi}^{(a\bba)}.
\end{eqnarray}
Using the second expansion in  \eqref{taylorofGbba},  and dealing in a similar manner to the expansion of the first term in \eqref{B114}, we have  \begin{eqnarray}\label{B115v}\begin{aligned}
\E B_{115}&=\frac{1}{N}\sum_{p} \frac{-\gamma^2 d_a^2}{(\gamma d_a^2-\xi)^2}\left(-\frac{1}{\tb}-\frac{1}{\tb^2}\mU_a\right) \sum_{\mu,\nu}^{(\bba)}G_{i\mu}^{(a \bba)}x_{a\mu}x_{a\nu} G_{p\nu}^{(a \bba)}G_{pi}^{(a\bba)}+ O(\Psi^5)\\
&=\frac{\gamma d_a^2}{(\gamma d_a^2-\xi)^2 \tb}\E X_{33}'-\frac{3\gamma^3 d_a^2 E_+(E_+b+\tb)}{(\gamma d_a^2-\xi)^3\tb}\E X_{44}-\frac{2\gamma^3 d_a^2 E_+^2}{(\gamma d_a^2-\xi)^2\tb^2}\E X_{44}+\frac{\gamma^3 d_a^2 E_+^2}{(\gamma d_a^2-\xi)^2\tb^2}\E X_{43}+O(\Psi^5).
\end{aligned}\end{eqnarray}

Next, we consider $B_{113}$. Notice that \begin{eqnarray}\label{B113}
B_{113}=B_{115}+\frac{1}{N}\sum \frac{-\gamma d_a^2}{(\gamma d_a^2-\xi)^2}\frac{G_{ia}^{(\bba)}G_{ap}^{(\bba)}G_{p\bba}G_{\bba i}}{G_{aa}^{(\bba)}G_{\bba\bba}}+\frac{1}{N}\sum \frac{-\gamma d_a^2}{(\gamma d_a^2-\xi)^2}\frac{G_{ia}^{(\bba)}G_{ap}^{(\bba)}G_{pa}^{(\bba)}G_{a i}^{(\bba)}}{G_{aa}^{(\bba)2}}
\end{eqnarray}
The expansion of the second term of \eqref{B113} was already done in \eqref{cmle4}. For the third term, we have  \begin{eqnarray}\begin{aligned}
&\E \frac{1}{N}\sum \frac{-\gamma d_a^2}{(\gamma d_a^2-\xi)^2}\frac{G_{ia}^{(\bba)}G_{ap}^{(\bba)}G_{pa}^{(\bba)}G_{a i}^{(\bba)}}{G_{aa}^{(\bba)2}}\\
&=\E \frac{1}{N}\sum \frac{-\gamma^3 d_a^2}{(\gamma d_a^2-\xi)^2}G_{aa}^{(\bba)2}G_{i\alpha}^{(a\bba)}x_{a\alpha}x_{a\beta}G_{p\beta}^{(a\bba)}G_{p\mu}^{(a\bba)}x_{a\mu}x_{x\nu}G_{i\nu}^{(a\bba)}\\
&=\frac{-\gamma^3 d_a^2 E_+^2}{(\gamma d_a^2-\xi)^2 \tb^2}(2\E X_{44}+\E \tX_{44})+O(\Psi^5).
\end{aligned}\end{eqnarray}
Therefore, we get  \begin{eqnarray}\label{B113v}
\E B_{113}=\E B_{115}+ \left(\frac{2\gamma^3 d_a^2 b E_+^2}{(\gamma d_a^2-\xi)^3\tb}+\frac{\gamma^3 d_a^2 E_+}{(\gamma d_a^2-\xi)^3}\right) \E X_{44}+\frac{\gamma^3 d_a^2 b E_+^2}{(\gamma d_a^2-\xi)^3\tb}\E \tX_{44}+O(\Psi^5).
\end{eqnarray}
Combining \eqref{B114v}, \eqref{B112v}, \eqref{B115v}, \eqref{B113v},
we conclude  \begin{eqnarray}\begin{aligned}
\E B_{11}&= \frac{\gamma d_a^2}{(\gamma d_a^2-\xi)^2}\E X_{22}+\frac{-2\gamma^2 d_a^2 \tb}{(\gamma d_a^2-\xi)^3}\E X_{33}+  \frac{-2\gamma^2 d_a^2 b}{(\gamma d_a^2-\xi)^3 }\E X_{33}'\\&+\left(\frac{4\gamma^4 d_a^4 E_+}{(\gamma d_a^2-\xi)^4}+\frac{2\gamma^3 d_a^2 b^2 E_+^2}{(\gamma d_a^2-\xi)^4} +\frac{2\gamma^3 d_a^2 \tb^2}{(\gamma d_a^2-\xi)^4} \right)\E X_{43}+  \frac{-6\gamma^3 d_a^2 E_+}{(\gamma d_a^2-\xi)^3}\E X_{44}\\
&+\left(\frac{-2\gamma^4 d_a^4 E_+}{(\gamma d_a^2-\xi)^4}+\frac{-\gamma^3 d_a^2 b^2 E_+^2}{(\gamma d_a^2-\xi)^4} +\frac{-\gamma^3 d_a^2 \tb^2}{(\gamma d_a^2-\xi)^4} \right)\E \tX_{44}+O(N^{-1/2}\Psi^3).
\end{aligned}\end{eqnarray}

\textbf{Expansions of $B_{12}, B_{13},\cdots, B_{16}$}.  For $B_{12}$ in \eqref{B1123456}, we have \begin{eqnarray}\label{B12v}\begin{aligned}
\E B_{12}&=\E \sum_{p,q} \frac{-2\gamma^2\tb d_a^2}{(\gamma d_a^2-\xi)^3}(cm-\sum x_{j\bba}G_{jk}^{(a\bba)}x_{k\bba})G_{ip}^{(a\bba)}x_{p\bba}x_{q\bba}G_{iq}^{(a\bba)}\\
&=\E\frac{1}{N}\sum_{p,q} \frac{-2\gamma^2\tb d_a^2}{(\gamma d_a^2-\xi)^3}\left(cm-\frac{1}{N}\sum G_{jj}^{(a\bba)}\right)G_{ip}^{(a\bba)}G_{pi}^{(a\bba)}+\E \frac{1}{N^2}\sum_{p,q} \frac{4\gamma^2\tb d_a^2}{(\gamma d_a^2-\xi)^3}G_{ip}^{(a\bba)}G_{pq}^{(a\bba)}G_{qi}^{(a\bba)}\\
&= \frac{-2\gamma^2\tb d_a^2}{(\gamma d_a^2-\xi)^3}\E X_{32}+ \frac{-2\gamma^3\tb d_a^2(E_+ b+\tb)}{(\gamma d_a^2-\xi)^4}(\E \tX_{44}-2\E X_{43}) \\
&\relphantom{EEEE}+ \frac{4\gamma^2\tb d_a^2}{(\gamma d_a^2-\xi)^3}\E X_{33}-\frac{12\gamma^3\tb d_a^2(E_+ b+\tb)}{(\gamma d_a^2-\xi)^4}\E X_{44}+O(\Psi^5),
\end{aligned}\end{eqnarray}
where in the last step, we use Lemma \ref{kao4o}.

For $B_{13}$ in \eqref{B1123456}, write
\begin{eqnarray}\begin{aligned}
B_{13}&=\frac{-2\gamma b d_a^2}{(\gamma d_a^2-\xi)^3}(E_+-z) \sum _{p,q}G_{ip}^{(a\bba)}x_{p\bba}x_{q\bba}G_{iq}^{(a\bba)}\\
&\relphantom{EEEE}+\frac{-2\gamma b d_a^2}{(\gamma d_a^2-\xi)^3}\left( E_+\tm-\sum_{\mu,\nu} x_{a\mu}G_{\mu\nu}^{(a)}x_{a\nu}-2 d_a \sum_{\mu} G_{\bba\mu}^{(a)}x_{a\mu}\right) \sum_{p,q}G_{ip}^{(a\bba)}x_{p\bba}x_{q\bba}G_{iq}^{(a\bba)}.
\end{aligned}\end{eqnarray}
The expectation of the first term is \begin{eqnarray}\label{mzpg3}
\frac{-2\gamma b d_a^2}{(\gamma d_a^2-\xi)^3} (E_+-z) \E X_{22} + O(\Psi^5).
\end{eqnarray}
For the second term,  we have
\begin{eqnarray}\label{ng82b}\begin{aligned}
&\E \frac{-2\gamma b d_a^2}{(\gamma d_a^2-\xi)^3}\left( E_+\tm-\sum x_{a\mu}G_{\mu\nu}^{(a)}x_{a\nu}-2 d_a \sum G_{\bba\mu}^{(a)}x_{a\mu}\right) \sum_{p,q}G_{ip}^{(a\bba)}x_{p\bba}x_{q\bba}G_{iq}^{(a\bba)}\\
&=\E \frac{-2\gamma^2 b d_a^2}{(\gamma d_a^2-\xi)^3}(E_+\tm -\frac{1}{N}\sum G_{\mu\mu}^{(a)})\sum_{p,q} G_{ip}^{(a\bba)}x_{p\bba}x_{q\bba}G_{iq}^{(a\bba)}\\
&=\E \frac{-2\gamma^2 b d_a^2}{(\gamma d_a^2-\xi)^3}(E_+\tm -\frac{1}{N}\sum G_{\mu\mu}^{(a\bba)})\sum_{p,q} G_{ip}^{(a\bba)}x_{p\bba}x_{q\bba}G_{iq}^{(a\bba)}\\
& \relphantom{EEEE}+\E \frac{2\gamma^2 b d_a^2}{(\gamma d_a^2-\xi)^3} \frac{G_{\mu\bba}^{(a)}G_{\bba\mu}^{(a)}}{G_{\bba\bba}^{(a)}}\sum_{p,q} G_{ip}^{(a\bba)}x_{p\bba}x_{q\bba}G_{iq}^{(a\bba)}\\
&= \frac{-2\gamma^2 b d_a^2 E_+}{(\gamma d_a^2-\xi)^3}\E X_{32}+\frac{-2\gamma^3 b d_a^2 E_+  (E_+ b+\tb)}{(\gamma d_a^2-\xi)^4}(\E \tX_{44}-2\E X_{43})\\
&\relphantom{EEEE}+\frac{-2\gamma^3  d_a^2 E_+}{(\gamma d_a^2-\xi)^3}(2\E X_{44}+\E \tX_{44})+O(\Psi^5).
%&= \frac{-2\gamma^2 b d_a^2 E_+}{(\gamma d_a^2-\xi)^3}\E X_{32}+\frac{4\gamma^3 b d_a^2 E_+  (E_+ b+\tb)}{(\gamma d_a^2-\xi)^4}\E X_{43}+\frac{-4\gamma^3  d_a^2 E_+}{(\gamma d_a^2-\xi)^3}\E X_{44} \\&\relphantom{EEEE}+\left(\frac{-2\gamma^3  d_a^2 b^2 E_+^2}{(\gamma d_a^2-\xi)^3}+\frac{-2\gamma^4 d_a^4 E_+}{(\gamma d_a^2-\xi)^4}\right)\E \tX_{44}+O(\Psi^5)
\end{aligned}\end{eqnarray}
Combining \eqref{mzpg3} and \eqref{ng82b}, we have \begin{eqnarray}\label{B13v}\begin{aligned}
\E B_{13}&=\frac{-2\gamma^2 b d_a^2 E_+}{(\gamma d_a^2-\xi)^3}\E X_{32}+\frac{-2\gamma b d_a^2}{(\gamma d_a^2-\xi)^3} (E_+-z) \E X_{22}+\frac{4\gamma^3 b d_a^2 E_+  (E_+ b+\tb)}{(\gamma d_a^2-\xi)^4}\E X_{43}\\
&\relphantom{EEEE}+\frac{-4\gamma^3  d_a^2 E_+}{(\gamma d_a^2-\xi)^3}\E X_{44}+\left(\frac{-2\gamma^3  d_a^2 b^2 E_+^2}{(\gamma d_a^2-\xi)^3}+\frac{-2\gamma^4 d_a^4 E_+}{(\gamma d_a^2-\xi)^4}\right)\E \tX_{44}+O(\Psi^5)
\end{aligned}\end{eqnarray}

For $B_{14}$ in \eqref{B1123456}, we have \begin{eqnarray}\label{B14v}\begin{aligned}
\E B_{14}&=\E \sum_{p,q}^{(a)}  \frac{3\gamma^3 \tb^2 d_a^2 }{(\gamma d_a^2-\xi)^4} \left(cm(E_+)- \sum_{j,k}^{(a)}x_{j\bba}G_{jk}^{(a\bba)}x_{k\bba}\right)^2G_{ip}^{(a\bba)}x_{p\bba}x_{q\bba}G_{iq}^{(a\bba)}\\
&=\frac{3\gamma^3 \tb^2 d_a^2}{(\gamma d_a^2-\xi)^4}\left(\E X_{42}-4\E X_{43}+8\E X_{44}+2\E \tX_{44}\right)+O(\Psi^5).
\end{aligned}\end{eqnarray}

For $B_{15}$ in \eqref{B1123456}, the expectation equals \begin{eqnarray}\label{B15v}\begin{aligned}
\E B_{15} &= \E \sum_{p,q}^{(a)} \frac{(2\gamma^2 d_a^4+4 \gamma \xi d_a^2)\mY_a \mZ_a}{(\gamma d_a^2-\xi)^4} G_{ip}^{(a\bba)}x_{p\bba}x_{q\bba}G_{iq}^{(a\bba)}\\
&=\frac{(2\gamma^4 d_a^4 E_++4 \gamma^3 \xi d_a^2 E_+)}{(\gamma d_a^2-\xi)^4}\E X_{42}-\frac{(4\gamma^4 d_a^4 E_+ +  8\gamma^3 \xi d_a^2 E_+)}{(\gamma d_a^2-\xi)^4}\E X_{43}+O(\Psi^5)\\
&=\left(\frac{6\gamma^4 d_a^4 E_+}{(\gamma d_a^2-\xi)^4}-\frac{4 \gamma^3 \xi d_a^2 E_+}{(\gamma d_a^2-\xi)^3}\right)\E X_{42}-\left(\frac{12\gamma^4 d_a^4 E_+}{(\gamma d_a^2-\xi)^4}-\frac{8 \gamma^3 \xi d_a^2 E_+}{(\gamma d_a^2-\xi)^3}\right)\E X_{43}+O(\Psi^5).
\end{aligned}\end{eqnarray}
For $B_{16}$ in \eqref{B1123456}, the expectation is
\begin{eqnarray}\begin{aligned}
\E B_{16}&=\E\sum_{p,q}^{(a)}  \frac{3\gamma b^2 d_a^2 \mZ_a^2}{(\gamma d_a^2-\xi)^4} G_{ip}^{(a\bba)}x_{p\bba}x_{q\bba}G_{iq}^{(a\bba)}\\
&=\E\sum_{p,q}^{(a)} \frac{3\gamma^3 b^2 d_a^2 }{(\gamma d_a^2-\xi)^4} \left (E_+\tm(E_+)-\sum_{\mu,\nu}x_{a\mu} G_{\mu\nu}^{(a)}x_{a\nu}-2 d_a\sum G_{\bba\mu}^{(a)}x_{a\mu}\right)^2 G_{ip}^{(a\bba)}x_{p\bba}x_{q\bba}G_{iq}^{(a\bba)}+O(\Psi^5)\\
&=\E\sum_{p,q}^{(a)} \frac{3\gamma^3 b^2 d_a^2}{(\gamma d_a^2-\xi)^4}\left (E_+\tm(E_+)-\sum_{\mu,\nu}x_{a\mu} G_{\mu\nu}^{(a)}x_{a\nu}\right)^2 G_{ip}^{(a\bba)}x_{p\bba}x_{q\bba}G_{iq}^{(a\bba)}\\
&\relphantom{EEEEE}+ \E\sum_{p,q}^{(a)} \frac{3\gamma^3 b^2 d_a^2 \mZ_a^2}{(\gamma d_a^2-\xi)^4}\left(2 d_a\sum G_{\bba\mu}^{(a)}x_{a\mu}\right)^2 G_{ip}^{(a\bba)}x_{p\bba}x_{q\bba}G_{iq}^{(a\bba)}+O(N^{-1/2}\Psi^3)\\
&= \frac{3\gamma^3 b^2 d_a^2 E_+^2}{(\gamma d_a^2-\xi)^4}\E X_{42}  + \frac{24\gamma^4 d_a^4 E_+}{(\gamma d_a^2-\xi)^4}\E X_{44}+\left(\frac{6\gamma^3 b^2 d_a^2 E_+^2}{(\gamma d_a^2-\xi)^4}+\frac{12\gamma^4 d_a^4 E_+}{(\gamma d_a^2-\xi)^4}\right)\E \tX_{44}+O(N^{-1/2}\Psi^3).
\end{aligned}\end{eqnarray}
Up to now we finish the expansion of $B_1$.

\subsubsection{Expansion of $B_2$}
We turn to estimate $B_2$ in \eqref{90afu}.  Using \eqref{taylorofGaa}
,write \begin{eqnarray}\label{B2}\begin{aligned}
&\sum_{\alpha}  G_{aa}^2  G_{i\alpha}^{(a)} x_{a\alpha}d_a G_{i\bba}^{(a)}\\
&=\sum_{\alpha}  (\beta_a ^2-2\beta_a^3\mZ_a+3\beta_a^4 \mZ_a^2)  G_{i\alpha}^{(a)} x_{a\alpha}d_a G_{i\bba}^{(a)}+O(\Psi^5).
\end{aligned}\end{eqnarray}

After taking expectation, the first term vanishes. Using \eqref{taylorofbetaa}, the expectation of the second term in \eqref{B2} equals
\begin{eqnarray}\begin{aligned}
&\E \sum_{ \mu}4\beta_a^3 \gamma d_a^2 G_{i\mu}^{(a)}G_{\mu\bba}^{(a)} G_{\bba i}^{(a)}- \E \sum_{\alpha} 2\beta_a^3 \gamma(E_+ \tm- \sum x_{a\mu}G_{\mu\nu}^{(a)}x_{a\nu})G_{i\alpha}^{(a)} x_{a\alpha}d_a G_{i\bba}^{(a)}\\
& =\E \sum_{ \mu} \frac{4b^3 \gamma d_a^2}{(\gamma d_a^2-\xi)^3}G_{i\mu}^{(a)}G_{\mu\bba}^{(a)} G_{\bba i}^{(a)}+\E \sum_{ \mu}\frac{-12\gamma^3 d_a^4 b^2}{(\gamma d_a^2-\xi)^4}(cm-\sum x_{j\bba}G_{jk}^{(a\bba)}x_{k\bba})G_{i\mu}^{(a)}G_{\mu\bba}^{(a)} G_{\bba i}^{(a)}+O(N^{-1/2}\Psi^3)\\
&=\E \sum_{ \mu} \frac{4 \gamma^2 b^3  d_a^2}{(\gamma d_a^2-\xi)^3}G_{i\mu}^{(a)}
G_{\bba\bba}^{(a)2}  \sum G_{\mu r}^{(a\bba)}x_{r\bba}x_{s\bba}G_{si}^{(a\bba)}\\&\relphantom{EEE}+ \E \sum_{ \mu}\frac{-12\gamma^4 d_a^4 b^2}{(\gamma d_a^2-\xi)^4}(cm-\sum x_{j\bba}G_{jk}^{(a\bba)}x_{k\bba})G_{i\mu}^{(a)}G_{\bba\bba}^{(a)2}   \sum G_{\mu r}^{(a\bba)}x_{r\bba}x_{s\bba}G_{si}^{(a\bba)}+O(N^{-1/2}\Psi^3)\\
&= \E \sum_{ \mu}\frac{4 \gamma^2 b^3 d_a^2}{(\gamma d_a^2-\xi)^3}\left[G_{i\mu}^{(a\bba)}\left(\frac{1}{b^2}+\frac{2\mY_a}{b^3}\right)  +\frac{G_{i\bba}^{(a)}G_{\bba\mu}^{(a)}}{G_{\bba\bba}^{(a)}}\frac{1}{b^2}\right]\sum G_{\mu r}^{(a\bba)}x_{r\bba}x_{s\bba}G_{si}^{(a\bba)}+\\&\relphantom{EEE}+\E \sum_{ \mu}\frac{-12\gamma^4 d_a^4 }{(\gamma d_a^2-\xi)^4}(cm-\sum x_{j\bba}G_{jk}^{(a\bba)}x_{k\bba})G_{i\mu}^{(a\bba)} \sum G_{\mu r}^{(a\bba)}x_{r\bba}x_{s\bba}G_{si}^{(a\bba)}+O(N^{-1/2}\Psi^3)\\
&=\frac{4 \gamma^2 b^3 d_a^2}{(\gamma d_a^2-\xi)^3}\E X_{33}'+\frac{-12\gamma^3 d_a^2 b E_+(E_+ b+\tb)}{(\gamma d_a^2-\xi)^4}\E X_{44}+\frac{8\gamma^3 d_a^2 E_+}{(\gamma d_a^2-\xi)^3}\E X_{43}+\frac{-16\gamma^3 d_a^2 E_+}{(\gamma d_a^2-\xi)^3}\E X_{44}\\
&\relphantom{EEE} +\frac{-4\gamma^3 d_a^2 E_+}{(\gamma d_a^2-\xi)^3}(2\E X_{44}+\E \tX_{44})+\frac{-12\gamma^4 d_a^4 E_+ }{(\gamma d_a^2-\xi)^4}(\E X_{43}-2\E X_{44})+O(N^{-1/2}\Psi^3).
\end{aligned}\end{eqnarray}

For the third term in \eqref{B2}, after expanding square of $\mZ_a$ and taking expectation, we find that all the terms with odd number of $X$ entries have negligible expectations, and therefore
\begin{eqnarray}\begin{aligned}
&\E \sum_{\alpha} 3\beta_a^4 \mZ_a^2  G_{i\alpha}^{(a)} x_{a\alpha}d_a G_{i\bba}^{(a)}\\
&=\E \sum_{\alpha} \frac{-12 \gamma^2 d_a^2 b^4 }{(\gamma d_a^2-\xi)^4}(E_+\tm-\frac{1}{N}\sum G_{\mu\mu}^{(a)})G_{i\alpha}^{(a)}G_{\alpha \bba}^{(a)}G_{\bba i}^{(a)}+\E \sum_{\alpha} \frac{24 \gamma^2 d_a^2 b^4 }{(\gamma d_a^2-\xi)^4} G_{i\alpha}^{(a)}G_{\alpha \beta}^{(a)} G_{\beta \bba}^{(a)}G_{\bba i}^{(a)}\\
&=\frac{-12 \gamma^3 d_a^2 b^2 E_+^2 }{(\gamma d_a^2-\xi)^4}\E X_{43}+\frac{24 \gamma^3 d_a^2 b^2 E_+^2 }{(\gamma d_a^2-\xi)^4}\E X_{44}+O(N^{-1/2}\Psi^3).
\end{aligned}\end{eqnarray}
Then, we have \begin{eqnarray}\label{B2v}\begin{aligned}
\E B_2 &=\frac{4 \gamma^2 b^3 d_a^2}{(\gamma d_a^2-\xi)^3}\E X_{33}'+\left(\frac{-12\gamma^4 d_a^4 E_+ }{(\gamma d_a^2-\xi)^4}+\frac{-12 \gamma^3 d_a^2 b^2 E_+^2 }{(\gamma d_a^2-\xi)^4}+\frac{8\gamma^3 d_a^2 E_+}{(\gamma d_a^2-\xi)^3}\right)\E X_{43}\\
&\relphantom{EEE}+\left(\frac{12\gamma^4 d_a^4 E_+ }{(\gamma d_a^2-\xi)^4}+\frac{12 \gamma^3 d_a^2 b^2 E_+^2 }{(\gamma d_a^2-\xi)^4}+\frac{-12\gamma^3 d_a^2 E_+}{(\gamma d_a^2-\xi)^3}\right)\E X_{44}+\frac{-4\gamma^3 d_a^2 E_+}{(\gamma d_a^2-\xi)^3}\E \tX_{44}+O(N^{-1/2}\Psi^3).
\end{aligned}\end{eqnarray}

\subsubsection{Expansion of $B_3$}

Next we estimate the expectation of $B_3$ in \eqref{90afu}. Using \eqref{taylorofGaa} and \eqref{taylorofbetaa}, we
write $B_3$ into summation of the following six terms: \begin{eqnarray}\label{B3set}\begin{aligned}
&B_{31}:=\frac{b^2}{(\gamma d_a^2-\xi)^2}\sum_{\alpha,\beta}G_{i\alpha}^{(a)}x_{a\alpha}x_{a\beta}G_{i\beta}^{(a)}, \quad B_{32}:=\frac{-2\gamma d_a^2 b \mY_a}{(\gamma d_a^2-\xi)^3}\sum_{\alpha,\beta}G_{i\alpha}^{(a)}x_{a\alpha}x_{a\beta}G_{i\beta}^{(a)},\\
&B_{33}:=\frac{-2b^3 \mZ_a}{(\gamma d_a^2-\xi)^3}\sum_{\alpha,\beta}G_{i\alpha}^{(a)}x_{a\alpha}x_{a\beta}G_{i\beta}^{(a)}, \quad B_{34}:=\frac{\gamma^2 d_a^4+2\gamma \xi d_a^2}{(\gamma d_a^2-\xi)^4}\mY_a^2 \sum_{\alpha,\beta}G_{i\alpha}^{(a)}x_{a\alpha}x_{a\beta}G_{i\beta}^{(a)}\\
&B_{35}:=\frac{6\gamma d_a^2 b^2}{(\gamma d_a^2-\xi)^4}\mY_a \mZ_a \sum_{\alpha,\beta}G_{i\alpha}^{(a)}x_{a\alpha}x_{a\beta}G_{i\beta}^{(a)}, \quad B_{36}:=\frac{3b^4}{(\gamma d_a^2-\xi)^4}\mZ_a^2 \sum_{\alpha,\beta}G_{i\alpha}^{(a)}x_{a\alpha}x_{a\beta}G_{i\beta}^{(a)}
\end{aligned}\end{eqnarray}
We handle the above six terms one by one. For $B_{31}$, taking partial expectation $\E_a$ and using \eqref{gfiminus}, we get \begin{eqnarray}\label{EB31start}\begin{aligned}
\E_a B_{31}&=\frac{1}{N}\sum_\alpha \frac{b^2}{(\gamma d_a^2-\xi)^2}G_{i\alpha}^{(a)}G_{\alpha i}^{(a)}\\
&=\frac{1}{N}\sum_\alpha \frac{b^2}{(\gamma d_a^2-\xi)^2}\left(G_{i\alpha}G_{\alpha i}-\frac{2 G_{ia}G_{a\alpha}}{G_{aa}}G_{\alpha i}^{(a)}-\frac{G_{ia}G_{a\alpha}G_{\alpha a}G_{ai}}{G_{aa}^2}\right)
\end{aligned}\end{eqnarray}
Expanding  the second term in the lower index $a$,
we can write it into summation of four terms: \begin{eqnarray*}\begin{aligned}
&B_{311}:=\frac{-2}{N}\sum_\alpha \frac{\gamma b^2}{(\gamma d_a^2-\xi)^2}G_{aa}G_{i\bba}^{(a)}d_a^2 G_{\alpha \bba}^{(a)}G_{\alpha i}^{(a)},\quad  B_{312}:=\frac{-2}{N}\sum_{\alpha,\nu} \frac{\gamma b^2}{(\gamma d_a^2-\xi)^2}G_{aa}G_{i\bba}^{(a)}d_a x_{a\nu}G_{\alpha\nu}^{(a)}G_{\alpha i}^{(a)},\\
&B_{313}:=\frac{-2}{N}\sum_{\alpha,\mu} \frac{\gamma b^2}{(\gamma d_a^2-\xi)^2}G_{aa}G_{i\mu}^{(a)}x_{a\mu}d_a G_{\alpha \bba}^{(a)}G_{\alpha i}^{(a)}, \quad  B_{314}:=\frac{-2}{N}\sum_{\alpha,\mu,\nu} \frac{\gamma b^2}{(\gamma d_a^2-\xi)^2}G_{aa}G_{i\mu}^{(a)}x_{a\mu}x_{a\nu}G_{\alpha \nu}^{(a)}G_{\alpha i}^{(a)}
\end{aligned}\end{eqnarray*}
We have
\begin{eqnarray*}\begin{aligned}
\E B_{311}&=\E\sum \frac{-2 \gamma^2 b^2}{(\gamma d_a^2-\xi)^2}G_{aa}G_{\bba\bba}^{(a)2}d_a^2 G_{ip}^{(a\bba)}x_{p\bba}x_{q\bba}G_{\alpha q}^{(a\bba)}\left(G_{\alpha i}^{(a\bba)}+\frac{G_{\alpha\bba}^{(a)}G_{\bba i}^{(a)}}{G_{\bba\bba}^{(a)}}\right)\\
&=\E \sum \frac{-2\gamma^2 b^2 d_a^2}{(\gamma d_a^2-\xi)^2}\left(\frac{1}{(\gamma d_a^2-\xi)b}+\frac{(\gamma d_a^2-2\xi)\mY_a}{(\gamma d_a^2-\xi)^2 b^2}+\frac{-\mZ_a}{(\gamma d_a^2-\xi)^2}\right)G_{ip}^{(a\bba)}x_{p\bba}x_{q\bba}G_{\alpha q}^{(a\bba)}G_{\alpha i}^{(a\bba)}\\
&\relphantom{EEE}+\E \sum \frac{2\gamma^3  b^2 d_a^2}{(\gamma d_a^2-\xi)^3}G_{ip}^{(a\bba)}x_{p\bba}x_{q\bba}G_{\alpha q}^{(a\bba)}G_{\alpha\bba}^{(a)}G_{\bba i}^{(a)} \\
&=\frac{-2\gamma^2 b d_a^2}{(\gamma d_a^2-\xi)^3}\E X_{33}'+ \left(\frac{6\gamma^3 d_a^2 b^2 E_+^2}{(\gamma d_a^2-\xi)^4}+\frac{6\gamma^3 \xi d_a^2 E_+}{(\gamma d_a^2-\xi)^4}\right)\E X_{44}+ \frac{2\gamma^3 (\gamma d_a^4-2\xi d_a^2)E_+}{(\gamma d_a^2-\xi)^4}(2\E X_{44}-\E X_{43})\\& \relphantom{EE}+ \frac{2\gamma^3 b^2 d_a^2 E_+^2}{(\gamma d_a^2-\xi)^4}\E X_{43}+ \frac{2\gamma^3 d_a^2 E_+}{(\gamma d_a^2-\xi)^3}(2\E X_{44}+\E \tX_{44})+O(\Psi^5).
\end{aligned}\end{eqnarray*}
\begin{eqnarray*}
\E B_{312}=\E B_{313}=\frac{-4\gamma^3 b^2 d_a^2 E_+^2}{(\gamma d_a^2-\xi)^4}\E X_{44}+O(N^{-1/2}\Psi^3).
\end{eqnarray*}
\begin{eqnarray*}\begin{aligned}
\E B_{314}&=\frac{-2}{N}\sum \frac{\gamma b^2}{(\gamma d_a^2-\xi)^2}\left(\frac{b}{\gamma d_a^2-\xi}-\frac{\gamma d_a^2\mY_a}{(\gamma d_a^2-\xi)^2}-\frac{b^2 \mZ_a}{(\gamma d_a^2-\xi)^2}\right)G_{i\mu}^{(a)}x_{a\mu}x_{a\nu}G_{\alpha \nu}^{(a)}G_{\alpha i}^{(a)}+O(\Psi^5)\\
&= \frac{-2\gamma b^3}{(\gamma d_a^2-\xi)^3}\E X_{33}''+  \left(\frac{6\gamma^3 b^2 d_a^2 E_+^2}{(\gamma d_a^2-\xi)^4}+\frac{6\gamma^2 b^4 E_+^3}{(\gamma d_a^2-\xi)^4}\right) \E X_{44}\\
&\relphantom{EE}+ \frac{2\gamma^3 b^2 d_a^2 E_+^2}{(\gamma d_a^2-\xi)^4}\E X_{43}+ \frac{2\gamma^2 b^4 E_+^3}{(\gamma d_a^2-\xi)^4}(\E X_{43}-2\E X_{44})+O(N^{-1/2}\Psi^3).
\end{aligned}\end{eqnarray*}
For the last term in \eqref{EB31start} it can be handled by similar steps to  \eqref{miy11}, \eqref{miy12} and \eqref{miy13}. We obtain
\begin{eqnarray}\label{EB31end}\begin{aligned}
&\E \frac{1}{N}\sum \frac{b^2}{(\gamma d_a^2-\xi)^2}\frac{-G_{ia}G_{a\alpha}G_{\alpha a}G_{ai}}{G_{aa}^2}\\
&=- \left(\frac{\gamma^4 d_a^4 E_+}{(\gamma d_a^2-\xi)^4}+\frac{2\gamma^3 d_a^2 b^2 E_+^2}{(\gamma d_a^2-\xi)^4}+\frac{\gamma^2 b^4 E_+^3}{(\gamma d_a^2-\xi)^4}\right) (2\E X_{44}+\E \tX_{44})+O(\Psi^5).
\end{aligned}\end{eqnarray}
Then we conclude $\E B_{31}$ with negligible order $O(N^{-1/2}\Psi^3)$ from the argument between \eqref{EB31start} and \eqref{EB31end}.

We turn to find other terms defined in \eqref{B3set}. By analogous decoupling  procedure as before, we can obtain that
\begin{eqnarray}\begin{aligned}
\E B_{32}&= \frac{-2\gamma^2 b d_a^2}{(\gamma d_a^2-\xi)^3}\E X_{32}'+ \left(\frac{4\gamma^3 d_a^2 b^2 E_+^2}{(\gamma d_a^2-\xi)^4}+\frac{4\gamma^3 d_a^2 \xi E_+}{(\gamma d_a^2-\xi)^4}+\frac{4\gamma^3 d_a^2 E_+}{(\gamma d_a^2-\xi)^3}\right)\E X_{43}\\
&\relphantom{EE} +\frac{-8\gamma^3 d_a^2 E_+}{(\gamma d_a^2-\xi)^3}\E X_{44}+ \left(\frac{-2\gamma^3 d_a^2 b^2 E_+^2}{(\gamma d_a^2-\xi)^4}+\frac{-2\gamma^3 d_a^2 \xi E_+}{(\gamma d_a^2-\xi)^4}\right)\E \tX_{44}+O(\Psi^5).
\end{aligned}\end{eqnarray}
\begin{eqnarray}\begin{aligned}
\E B_{33}&= \frac{-2\gamma b^3 E_+}{(\gamma d_a^2-\xi)^3}\E X_{32}'+ \frac{4\gamma b^3}{(\gamma d_a^2-\xi)^3}\E X_{33}'' + \frac{-2\gamma b^4}{(\gamma d_a^2-\xi)^3}(E_+-z)\E X_{22}'\\&\relphantom{EE}+ \left(\frac{2\gamma^3 d_a^2 b^2 E_+^2}{(\gamma d_a^2-\xi)^4}+\frac{2\gamma^2 b^4 E_+^3}{(\gamma d_a^2-\xi)^4}  \right)(2\E X_{43}-6\E X_{44}-\E \tX_{44})+O(\Psi^5).
\end{aligned}\end{eqnarray}

\begin{eqnarray}
\E B_{34}=\frac{\gamma^4 d_a^4 E_+ + 2\gamma^3 \xi d_a^2 E_+}{(\gamma d_a^2-\xi)^4}(\E X_{42}+2\E \tX_{44})+O(\Psi^5).
\end{eqnarray}
\begin{eqnarray}
\E B_{35}=\frac{6\gamma^3 d_a^2 b^2 E_+^2}{(\gamma d_a^2-\xi)^4}(\E X_{42}-2\E X_{43})+ O(\Psi^5).
\end{eqnarray}

\begin{eqnarray}\begin{aligned}
\E B_{36}&= \frac{3\gamma^2 b^4 E_+^3}{(\gamma d_a^2-\xi)^4}\E X_{42} + \frac{-12\gamma^2 b^4 E_+^3}{(\gamma d_a^2-\xi)^4}\E X_{43}+ \left(\frac{24\gamma^2 b^4 E_+^3}{(\gamma d_a^2-\xi)^4}+\frac{24\gamma^2 b^2 d_a^2 E_+^2}{(\gamma d_a^2-\xi)^4}\right)\E X_{44} \\
&\relphantom{EE}+\left( \frac{6\gamma^2 b^4 E_+^3}{(\gamma d_a^2-\xi)^4}+\frac{12\gamma^3 b^2 d_a^2 E_+^2}{(\gamma d_a^2-\xi)^4}\right) \E \tX_{44}+O(\Psi^5).
\end{aligned}\end{eqnarray}
We conclude that  \begin{eqnarray}\begin{aligned}
\E B_3 &= \frac{b^2}{(\gamma d_a^2-\xi)^2}\E X_{22}' +\left(\frac{-2\gamma^2  b d_a^2}{(\gamma d_a^2-\xi)^3}+\frac{-2\gamma b^3 E_+}{(\gamma d_a^2-\xi)^3}\right)\E X_{32}' +\frac{-2\gamma^2 b d_a^2}{(\gamma d_a^2-\xi)^3}\E X_{33}'+\\
&+\left(\frac{-2\gamma^2 b d_a^2}{(\gamma d_a^2-\xi)^3}+\frac{4\gamma b^3}{(\gamma d_a^2-\xi)^3}\right) \E X_{33}''+ \left(\frac{\gamma^4 d_a^4 E_+}{(\gamma d_a^2-\xi)^4}+\frac{6 \gamma^3 b^2 d_a^2 E_+^2}{(\gamma d_a^2-\xi)^4}+\frac{3\gamma^2 b^4 E_+^2}{(\gamma d_a^2-\xi)^4} +\frac{2\gamma^3 \xi d_a^2 E_+}{(\gamma d_a^2-\xi)^4}\right)\E X_{42}\\
&+\left(\frac{2\gamma^4 d_a^4 E_+}{(\gamma d_a^2-\xi)^4}+\frac{-6\gamma ^2 b^4 E_+^3}{(\gamma d_a^2-\xi)^4} +\frac{4\gamma^3 d_a^2 \xi E_+}{(\gamma d_a^2-\xi)^4}\right) \E X_{43}+\left(\frac{3\gamma^4 d_a^4 E_+}{(\gamma d_a^2-\xi)^4} +\frac{6 \gamma^3 b^2 d_a^2 E_+^2}{(\gamma d_a^2-\xi)^4}+\frac{3 \gamma^2 b^4 E_+^3}{(\gamma d_a^2-\xi)^4} \right)\E \tX_{44}
\\&+\left(\frac{-2\gamma^4 d_a^4 E_+}{(\gamma d_a^2-\xi)^4}+\frac{12 \gamma^3 b^2 d_a^2 E_+^2}{(\gamma d_a^2-\xi)^4}+\frac{12\gamma^2 b^4 E_+^3}{(\gamma d_a^2-\xi)^4}+\frac{2\gamma^3 \xi d_a^2 E_+}{(\gamma d_a^2-\xi)^4}\right)\E X_{44}+\frac{-2\gamma b^4 E_+}{(\gamma d_a^2-\xi)^3}(E_+-z)\E X_{22}+O(N^{-1/2}\Psi^3).
\end{aligned}\end{eqnarray}

\subsubsection{Optical theorem from $X_{22}$}
By tedious calculations, we summarize that
\begin{eqnarray}\label{X22}\begin{aligned}
&\frac{1}{N}\E \sum_a^{(i)} G_{ia}G_{ai}=  \gamma \psi_2 \E X_{22}+\gamma b^2 \varphi_2 \E X_{22}'-  \frac{1}{N}\sum \left(\frac{2\gamma ^3\tb d_a^2}{(\gamma d_a^2-\xi)^3}+ \frac{2\gamma^3 b d_a^2 E_+}{(\gamma d_a^2-\xi)^3}\right) \E X_{32}\\
& -\frac{1}{N}\sum\left(\frac{2\gamma ^3 b d_a^2}{(\gamma d_a^2-\xi)^3}+\frac{2\gamma^2 b^3 E_+}{(\gamma d_a^2-\xi)^3}\right) \E X_{32}'+\frac{1}{N}\sum\frac{2\gamma ^3\tb d_a^2}{(\gamma d_a^2-\xi)^3}\E  X_{33}+\frac{1}{N}\sum\frac{4\gamma^3 b d_a^2 }{(\gamma d_a^2-\xi)^3}\E X_{33}'\\
&\relphantom{E}+\frac{1}{N}\sum\frac{2\gamma^2 b^3}{(\gamma d_a^2-\xi)^3} \E X_{33}''-\frac{1}{N}\sum\left(\frac{2\gamma^2 b^2 d_a^2}{(\gamma d_a^2-\xi)^3}+\frac{2\gamma b^4 E_+}{(\gamma d_a^2-\xi)^3}\right)(E_+-z) \E X_{22}\\
&\relphantom{E}+\frac{1}{N}\sum \left(\frac{\gamma^3 E_+(\gamma d_a^2+E_+ b^2
)^2}{(\gamma d_a^2-\xi)^4}+\frac{\gamma^4 d_a^2(E_+ b+\tb)^2}{(\gamma d_a^2-\xi)^4}\right) (3\E X_{42}-\E 6X_{43}+\E 12X_{44}+3\E \tX_{44}) +O(N^{-1/2}\Psi^3).
\end{aligned}\end{eqnarray}

By similar procedure, we obtain that
\begin{eqnarray}\label{X22'1}\begin{aligned}
&\frac{1}{N}\E \sum_a^{(i)} G_{i\bba}G_{\bba i}=\gamma\psi_2 \E X_{22}'+ \gamma \tb^2 \varphi_2 \E  X_{22}-\frac{1}{N}\sum \left(\frac{2\gamma^3 \tb d_a^2 E_+}{(\gamma d_a^2-\xi)^3}+\frac{2\gamma^2 \tb^3}{(\gamma d_a^2-\xi)^3}\right) \E X_{32}\\
&-\frac{1}{N}\sum\left( \frac{2\gamma^3 b d_a^2 E_+}{(\gamma d_a^2-\xi)^3}+ \frac{2\gamma^3 \tb d_a^2 }{(\gamma d_a^2-\xi)^3}\right)\E X_{32}'+\frac{1}{N}\sum \frac{2\gamma^3 d_a^2 b}{(\gamma d_a^2-\xi)^3}\E X_{33}''+\frac{1}{N}\sum \frac{4\gamma^3 d_a^2 \tb}{(\gamma d_a^2-\xi)^3}\E X_{33}'
\\&+\frac{1}{N}\sum \frac{2\gamma^2 \tb^3}{(\gamma d_a^2-\xi)^3}\E X_{33}-\frac{1}{N}\sum\left( \frac{2\gamma^2 \xi d_a^2}{(\gamma d_a^2-\xi)^3}+\frac{2\gamma^2 b^2 d_a^2 E_+}{(\gamma d_a^2-\xi)^3}\right)(E_+-z) \E X_{22}\\
&+\frac{1}{N}\sum \left(\frac{\gamma^3 (\gamma E_+ d_a^2+\tb^2)^2}{(\gamma d_a^2-\xi)^4}+\frac{\gamma^4 E_+ d_a^2 (E_+ b+\tb)^2}{(\gamma d_a^2-\xi)^4}\right)(3\E X_{42}-6\E X_{43}+12\E X_{44}+3\E \tX_{44})+O(N^{-1/2}\Psi^3).
\end{aligned}\end{eqnarray}

\indent It is easy to obtain that
\begin{eqnarray}\label{X22'spe}
\E G_{i\bbi}^2 =\frac{\gamma d_i^2}{(\gamma d_i^2-\xi)^2}-\frac{2\gamma^2 b d_i^2 E_+}{(\gamma d_i^2-\xi)^3}\E (cm-cm_N)-\frac{2\gamma^2 \tb d_i^2}{(\gamma d_i^2-\xi)^3}\E (cm-cm_N)+O(\Psi^2),
\end{eqnarray} and
\begin{eqnarray}\label{X22'2}\begin{aligned}
\frac{1}{N}\E \sum_{\alpha>2M} G_{i\alpha}G_{\alpha i}&=\frac{\gamma(1-c)}{b^2}\E X_{22}+\frac{2\gamma^2 (1-c)}{b^3}\E X_{32}-\frac{2\gamma^2 (1-c)}{b^3}\E X_{33}\\
&+\frac{\gamma^3(1-c)}{b^4}(3\E X_{42}-6\E X_{43}+12\E X_{44}+3\E \tX_{44})+O(\Psi^5)
\end{aligned}\end{eqnarray}
%Since $X_{22}' = \sum G_{i\bba}G_{\bba i} +\sum_{\alpha>2M} G_{i\alpha}G_{\alpha i}$
We remark that in \eqref{X22} and \eqref{X22'1},  terms contain $\E X_{32}'$ can be further represented as linear combination of $\E X_{32}$ and $\E(cm-cm_N)$, see \eqref{x32'vsimple} below. Terms contain $\E X_{33}'$ and $\E X_{33}''$ can be replaced with $E_+\E X_{33}$ and $E_+^2 \E X_{33}$ respectively according to \eqref{X33'fiexp} and \eqref{X33''fiexp}.

Let $$\theta_4: =\frac{1}{N}\sum_a \left(\frac{\gamma^3 E_+^2(\gamma d_a^2+E_+ b^2)^2}{(\gamma d_a^2-\xi)^4}+\frac{2\gamma^4 d_a^2 E_+ (E_+ b+\tb)^2}{(\gamma d_a^2-\xi)^4}+\frac{\gamma^3 (\gamma E_+ d_a^2+\tb^2)^2}{(\gamma d_a^2-\xi)^4}+\frac{\gamma^3(1-c)}{b^4}\right)$$
Multiplying $E_+$ on both sides of \eqref{X22} and \eqref{X22spe}, then summing with \eqref{X22'1},\eqref{X22'spe} and \eqref{X22'2} together, we find that the coefficient of $\E X_{22}$ and $\E X_{22}'$ cancel by \eqref{relofvarphi}, and obtain  \begin{eqnarray}\begin{aligned}
&\frac{E_+ b^2}{(\gamma d_i^2-\xi)^2}+\frac{\gamma d_i^2}{(\gamma d_i^2-\xi)^2}-\left(\frac{2\gamma^2 \tb d_i^2}{(\gamma d_i^2-\xi)^3}+\frac{4\gamma^2 b d_i^2 E_+}{(\gamma d_i^2-\xi)^3}+\frac{2\gamma b^3 E_+^2}{(\gamma d_i^2-\xi)^3} \right) \E (cm-cm_N)\\&-2\Phi_2 X_{32} -2\Phi_1 \E X_{32}' +(2\Phi_1 E_+ +2\Phi_2)\E  X_{33}-\frac{2b \Phi_1}{\gamma }(E_+-z)\E X_{22}\\&+\theta_4 (3\E X_{42}-6\E X_{43}+12\E X_{44}+3\E \tX_{44})=O(N^{-1/2}\Psi^3)
\end{aligned}
\end{eqnarray}
Summing over index $i$,  and using \eqref{x32'vsimple}, we have that \begin{eqnarray}\begin{aligned}
E_+ b^2 \varphi_2 &+\psi_2+(2\Phi_2+2E_+\Phi_1)(-X_{32}+X_{33})-\frac{2b \Phi_1}{\gamma }\E(cm-cm_N)-\frac{2b \Phi_1}{\gamma }(E_+-z)X_{22}\\& +\theta_4 (3X_{42}-6X_{43}+12X_{44}+3\tX_{44})=O(N^{-1/2}\Psi^3).
\end{aligned}\end{eqnarray}
Actually, through tedious calculations, we find that
Recall \eqref{defmathcalX}. Using Lemma \ref{Phivalue}, we get the optical theorem from $X_{22}$,
\begin{eqnarray}\label{optfromX22}
\frac{1}{\gamma}+\frac{2}{\gamma}\E \mathcal{X}_{3}+\frac{2b^2}{\gamma^2 h}\E \mathfrak{X}_{4}+\theta_4 \E \mathcal{X}_4=O(N^{-1/2}\Psi^3),
\end{eqnarray}
where \begin{eqnarray}
\theta_4=\gamma^3 h^4 \varphi_4-\frac{4E_+ }{h}-\frac{4\gamma^2 E_+}{b^2}-\frac{2\gamma^2 E_+^2}{b^2 h}+\frac{\gamma^3(1-c)}{b^4}.
\end{eqnarray}
The above expression of $\theta_4$ is obtained through simple but lengthy calculations with the help of Lemma \ref{Phivalue}.

Combining \eqref{X22'1},\eqref{X22'spe} and \eqref{X22'2}, we have\begin{eqnarray*}\begin{aligned}
\frac{1}{N}\sum_i \E X_{22}'=& E_+ \frac{1}{N}\sum_i  \E X_{22} + \frac{1}{N}\sum_i  \frac{2\tb}{\gamma E_+}\E[X_{32}-X_{33}]+\frac{\psi_2}{1-\gamma \psi_2}\\&-\frac{1}{N^2}\sum_{i,a}\frac{2\gamma^2 d_a^2 b (E_+b+\tb)^2}{(\gamma d_a^2-\xi)^3 E_+} \left[\E(cm-cm_N)+ (E_+-z)\E X_{22} \right]\\&+\frac{1}{1-\gamma \psi_2}\left\{\frac{1}{N}\sum_{a} \left(\frac{\gamma^3 (\gamma E_+ d_a^2+\tb^2)^2}{(\gamma d_a^2-\xi)^4}+\frac{\gamma^4 E_+ d_a^2 (E_+ b+\tb)^2}{(\gamma d_a^2-\xi)^4}\right)+\frac{\gamma^3(1-c)}{b^4}\right\}\\&\times \frac{1}{N}\sum_i   \E (3X_{42}-6X_{43}+12X_{44}+3\tX_{44})+O(N^{-1/2}\Psi^3).
\end{aligned}\end{eqnarray*}
Using Lemma \ref{Phivalue}, we have  \begin{eqnarray}\label{X22'vsimple}\begin{aligned}
&\frac{1}{N}\sum_i \E X_{22}' = E_+ \frac{1}{N}\sum_i  \E X_{22} + \frac{2\tb}{\gamma E_+}\E \mathcal{X}_{3} +\left(\frac{h}{\gamma E_+}-\frac{1}{\gamma}\right)-\frac{2b(\gamma^2 E_+^2-b\tb)}{\gamma^2 E_+ h}\E \mathfrak{X}_{4}\\
&+ \left( \frac{\gamma^3 h^4 \tb \varphi_6}{E_+} -\frac{3\tb}{h}-\frac{3\gamma^2 \tb}{b^3}-\frac{3\gamma^2 E_+\tb}{b^2 h}-\frac{E_+  b^2+\gamma^2 E_+^2 }{b h}+ \frac{h \gamma^3(1-c)}{E_+ b^4}\right)\E  \mathcal{X}_4+O(N^{-1/2}\Psi^3)
\end{aligned}\end{eqnarray}

\subsection{Optical theorem from $X_{33}$}
Recall $$X_{33}=\frac{1}{N^2}\sum_{a,b} G_{ia}G_{ab}G_{bi}.$$
Since the contributions to the sum when $a=i$  or $a=b$ are of negligible order $O(\Psi^2)$, we can assume that $i,b\neq a.$
Expanding $X_{33}$ with respect to index $a$, we have
\begin{eqnarray}\begin{aligned}\label{optx331}
  \frac{1}{N^2}\sum_{a,b}G_{ia}G_{ab}G_{bi} & =\frac{1}{N^2}\sum_{a,b,\alpha,\beta}G_{aa}^2 G_{i\alpha}^{(a)}Y_{a\alpha}Y_{a\beta}G_{b\beta}^{(a)}G_{bi}\\
  &=\frac{\gamma}{N^2}\sum_{a,b} G_{aa}^2 G_{i\bar{a}}^{(a)}d_a^2 G_{b\bar{a}}^{(a)}G_{bi}+\frac{\gamma}{N^2}\sum_{a,b,\alpha} G_{aa}^2 G_{i\alpha}^{(a)}x_{a\alpha}d_a  G_{b\bar{a}}^{(a)}G_{bi}\\ &\relphantom{EEEE} +\frac{\gamma}{N^2}\sum_{a,b,\beta} G_{aa}^2 G_{i\bar{a}}^{(a)}d_a  x_{a\beta}G_{b\beta}^{(a)}G_{bi}+\frac{\gamma}{N^2}\sum_{a,b,\alpha,\beta} G_{aa}^2 G_{i\alpha}^{(a)}x_{a\alpha} x_{a\beta} G_{b\beta}^{(a)}G_{bi}
\end{aligned}
\end{eqnarray}
By decomposing $G_{bi}$ into sum of $G_{bi}^{(\bba)}$ and $G_{b\bba}G_{\bba i}/G_{\bba\bba}$ for the first three terms,  and sum of $G_{bi}^{(a)}$ and $G_{ba} G_{ai}/G_{aa}$ for the last term, we can write $X_{33}$ into sum of eight terms:
\begin{eqnarray}\begin{aligned}
&D_1:=\frac{\gamma}{N^2}\sum_{a,b} G_{aa}^2 G_{i\bba}^{(a)}d_a^2 G_{b\bba}^{(a)}G_{bi}^{(\bba)}, \quad  D_2:=\frac{\gamma}{N^2}\sum_{a,b} G_{aa}^2 G_{i\bba}^{(a)}d_a^2 G_{b\bba}^{(a)}\frac{G_{b\bba}G_{\bba i}}{G_{\bba\bba}},\\
&D_3:=\frac{\gamma}{N^2}\sum_{a,b,\alpha} G_{aa}^2 G_{i\alpha}^{(a)}x_{a\alpha}d_a  G_{b\bba}^{(a)}G_{bi}^{(\bba)},\quad D_4:=\frac{\gamma}{N^2}\sum_{a,b,\alpha} G_{aa}^2 G_{i\alpha}^{(a)}x_{a\alpha}d_a  G_{b\bba}^{(a)}\frac{G_{b\bba}G_{\bba i}}{G_{\bba\bba}},\\
&D_5:=\frac{\gamma}{N^2}\sum_{a,b,\beta} G_{aa}^2 G_{i\bba}^{(a)}d_a  x_{a\beta}G_{b\beta}^{(a)}G_{bi}^{(\bba)}, \quad  D_6:=\frac{\gamma}{N^2}\sum_{a,b,\beta} G_{aa}^2 G_{i\bba}^{(a)}d_a  x_{a\beta}G_{b\beta}^{(a)}\frac{G_{b\bba}G_{\bba i}}{G_{\bba\bba}},\\
&D_7:=\frac{\gamma}{N^2}\sum_{a,b,\alpha,\beta} G_{aa}^2 G_{i\alpha}^{(a)}x_{a\alpha} x_{a\beta} G_{b\beta}^{(a)}G_{bi}^{(a)},\quad  D_8:=\frac{\gamma}{N^2}\sum_{a,b,\alpha,\beta} G_{aa}^2 G_{i\alpha}^{(a)}x_{a\alpha} x_{a\beta} G_{b\beta}^{(a)}\frac{G_{ba}G_{ai}}{G_{aa}}.
\end{aligned}\end{eqnarray}

We analyze above terms one by one.  $D_1$ can be written as
\begin{eqnarray}\label{D1}\begin{aligned}
 D_1&=  \frac{\gamma^2}{N^2}\sum_{a,b,j,k} G_{aa}^2 G_{\bba\bba}^{(a)2}d_a^2 G_{ij}^{(a\bba)}x_{j\bba}x_{k\bba}G_{bk}^{(a\bba)}G_{bi}^{(\bba)}\\
&=\frac{\gamma^2}{N^2}\sum \left(\frac{1}{(\gamma d_a^2-\xi)^2}+\frac{-2\tb \mY_a}{(\gamma d_a^2-\xi)^3}+\frac{-2b \mZ_a}{(\gamma d_a^2-\xi)^3} \right)d_a^2 G_{ij}^{(a\bba)}x_{j\bba}x_{k\bba}G_{bk}^{(a\bba)}G_{bi}^{(\bba)}+O(\Psi^5).
\end{aligned}\end{eqnarray}
The expectation of the first term on the right side of \eqref{D1} equals \begin{eqnarray}\begin{aligned}
&\E \frac{1}{N^3}\sum_{a,b,j} \frac{\gamma^2 d_a^2}{(\gamma d_a^2-\xi)^2}G_{ij}^{(a\bba)}G_{jb}^{(a\bba)}G_{bi}^{(\bba)} \\
&=\E \frac{1}{N^3}\sum_{a,b,j} \frac{\gamma^2 d_a^2}{(\gamma d_a^2-\xi)^2}G_{ij}^{(a\bba)}G_{jb}^{(a\bba)}G_{bi}^{(a\bba)}+\E \frac{1}{N^3}\sum_{a,b,j} \frac{\gamma^2 d_a^2}{(\gamma d_a^2-\xi)^2}G_{ij}^{(a\bba)}G_{jb}^{(a\bba)}\frac{G_{ba}^{(\bba)}G_{ai}^{(\bba)}}{G_{aa}^{(\bba)}}\\
&=\frac{1}{N}\sum_a \frac{\gamma^2 d_a^2}{(\gamma d_a^2-\xi)^2}\E X_{33} +\frac{1}{N}\sum_a \frac{-3\gamma ^3 d_a^2(E_+ b+\tb)}{(\gamma d_a^2-\xi)^3}\E X_{44}+\frac{1}{N}\sum_a \frac{-\gamma^3 d_a^2 E_+}{(\gamma d_a^2-\xi)^2\tb}\E X_{44}+O(\Psi^5),
\end{aligned}\end{eqnarray}
where we use the third equation in \eqref{zo9g2} to handle the first term in the second line above.
The expectation of the other two terms in \eqref{D1} equals\begin{eqnarray}
\frac{1}{N}\sum_a \frac{-2\gamma^3 d_a^2 \tb }{(\gamma d_a^2-\xi)^3}(\E X_{43}-2\E X_{44})+\frac{1}{N}\sum_a \frac{-2\gamma^3 d_a^2 b E_+}{(\gamma d_a^2-\xi)^3}\E X_{43}+O(\Psi^5).
\end{eqnarray}
Then we get $\E D_1$ from above easily.
For the other terms, we find \begin{eqnarray}
\E D_2=\frac{1}{N}\sum_a \frac{\gamma^4  d_a^4 E_+}{(\gamma d_a^2-\xi)^3\tb}\E X_{44}+\frac{1}{N}\sum_a \frac{2\gamma^3 d_a^2 \tb}{(\gamma d_a^2-\xi)^3}\E X_{44}+\frac{1}{N}\sum_a \frac{\gamma^3 d_a^2 \tb}{(\gamma d_a^2-\xi)^3}\E \tX_{44}+O(\Psi^5).
\end{eqnarray}
\begin{eqnarray}
\E D_3=\E D_5=\frac{1}{N}\sum_a \frac{4 \gamma^3 d_a^2 b E_+}{(\gamma d_a^2-\xi)^3}\E X_{44}+O(N^{-1/2}\Psi^3).
\end{eqnarray}
\begin{eqnarray}
\E D_4=\E D_6=\frac{1}{N}\sum_a  \frac{\gamma^3 d_a^2 b E_+}{(\gamma d_a^2-\xi)^3}(\E X_{44}+\E \tX_{44})+O(\Psi^5).
\end{eqnarray}

\begin{eqnarray}\begin{aligned}
\E D_7&=\frac{1}{N}\sum_a \frac{\gamma b^2}{(\gamma d_a^2-\xi)^2}\E X_{33}' +\frac{1}{N}\sum_a \left(\frac{-3\gamma^3 d_a^2 b E_+ }{(\gamma d_a^2-\xi)^3}+ \frac{-3\gamma^2 b^3 E_+^2}{(\gamma d_a^2-\xi)^3}\right)\E X_{44}\\
&\relphantom{EEE}+\frac{1}{N}\sum_a \left( \frac{-2 \gamma^3 d_a^2 b E_+}{(\gamma d_a^2-\xi)^3}+\frac{-2\gamma^2 b^3 E_+^2}{(\gamma d_a^2-\xi)^3}\right)\E X_{43}+\frac{1}{N}\sum_a \frac{4\gamma^2 b^3 E_+^2}{(\gamma d_a^2-\xi)^3}\E X_{44}+O(\Psi^5).
\end{aligned}\end{eqnarray}

\begin{eqnarray}
\E D_8=\frac{1}{N}\sum_a \frac{\gamma^3 d_a^2 b E_+}{(\gamma d_a^2-\xi)^3}\E_{44}+\frac{1}{N}\sum_a \frac{\gamma^2 b^3 E_+^2}{(\gamma d_a^2-\xi)^3} (2\E X_{44}+\E \tX_{44})+O(\Psi^5).
\end{eqnarray}
We conclude from above that
\begin{eqnarray*}
\E X_{33}=\gamma \psi_2 \E X_{33}+\gamma b^2 \varphi_2 \E X_{33}'+\Phi_1 (-2\E X_{43}+3\E X_{44}+\E \tX_{44})+O(N^{-1/2}\Psi^3).
\end{eqnarray*}
Using Lemma \ref{Phivalue}, it follows that\begin{eqnarray}\label{X33v}
\E X_{33}=\frac{1}{E_+}\E X_{33}'+\frac{-b}{\gamma E_+}(-2\E X_{43}+3\E X_{44}+\E \tX_{44})+O(N^{-1/2}\Psi^3).
\end{eqnarray}

By similar procedure of expanding $X_{33}$, we can derive
\begin{eqnarray}\label{X33l2Mv}
\frac{1}{N^2}\sum_{\alpha>2M,k} G_{i\alpha}G_{\alpha k}G_{ki}= \frac{\gamma(1-c)}{b^2}\E X_{33}+\frac{-\gamma^2(1-c)}{b^3}(-2\E X_{43}+3\E X_{44}+\E \tX_{44})+O(N^{-1/2}\Psi^3).
\end{eqnarray}
And \begin{eqnarray*}
\E X_{33}'= \gamma \psi_2 \E X_{33}'+\gamma \varpi_2\E X_{33} + \Phi_2 (-2\E X_{43}+3\E X_{44}+\E \tX_{44})+O(N^{-1/2}\Psi^3).
\end{eqnarray*}
It follows that \begin{eqnarray}\label{X33v'}
\E X_{33}'=E_+ \E X_{33}+\frac{-\tb}{\gamma E_+}(-2\E X_{43}+3\E X_{44}+\E \tX_{44})+O(\Psi^5).
\end{eqnarray}
From \eqref{X33v} and \eqref{X33v'}, we get the optical theorem obtained from $X_{33}$, \begin{eqnarray}\label{optfromX33}
-2\E X_{43}+3\E X_{44}+\E \tX_{44}=O(N^{-1/2}\Psi^3).
\end{eqnarray}
Substituting \eqref{optfromX33} back into \eqref{X33v'}, we get \begin{eqnarray}\label{X33'fiexp}
\E X_{33}'= E_+ X_{33}+ O(N^{-1/2}\Psi^3).
\end{eqnarray}
Similarly, by expanding  $\E X_{33}'$ with respect to the Latin lower index and an $\E X_{33}''$  with respect to a Greek lower index, together with the optical law \eqref{optfromX33}, we can get \begin{eqnarray}\label{X33''fiexp}\E X_{33}''=E_+ \E X_{33}'+O(N^{-1/2}\Psi^3).\end{eqnarray}

\subsection{Optical theorem from $X_{32}$}
Recall $$X_{32}= (cm(E_+)-cm_N)\frac{1}{N}\sum_{r} G_{ir}G_{ri}$$. Expanding with lower index $r$, through tedious calculations, we obtain
\begin{eqnarray}\label{x32fo}
\E X_{32}-\frac{1}{N}\E(cm-cm_N)G_{ii}^2 = \gamma \psi_2 \E X_{32}+\gamma b^2 \varphi_2 \E X_{32}'+\Phi_1\E (-2X_{42}+2X_{43}-2X_{44})+O(\Psi^5).
\end{eqnarray}
Expanding of the left side of the following with lower index $\alpha$, we get
\begin{eqnarray}\label{x32'l2Mv}
\frac{1}{N}\sum_{\alpha>2M}\E (cm-cm_N)G_{i\alpha}G_{\alpha i}= \frac{\gamma(1-c)}{b^2}\E X_{32}-\frac{\gamma^2(1-c)}{b^3}\E (-2X_{42}+2X_{43}-2X_{44})+O(\Psi^5).
\end{eqnarray}
Similar to \eqref{x32fo}, we can obtain
\begin{eqnarray}\label{x32'fo}
\E X_{32}'-\frac{1}{N}\E(cm-cm_N)G_{i\bbi}^2 = \gamma \psi_2\E X_{32}'+\gamma \varpi_2 \E X_{32}+\Phi_2\E (-2X_{42}+2X_{43}-2X_{44})+O(\Psi^5).
\end{eqnarray}
From \eqref{x32'fo}, we have \begin{eqnarray}\label{349b8}
\E X_{32}'= E_+ \E X_{32}+\frac{ \gamma d_i^2 (E_+b +\tb) }{ E_+  (\gamma d_i^2-\xi)^2} \E\frac{1}{N} (cm-cm_N)-\frac{\tb}{\gamma E_+}\E (-2X_{42}+2X_{43}-2X_{44})
\end{eqnarray}
Using Lemma \ref{Phivalue}, we find from \eqref{x32fo} and  \eqref{349b8} that \begin{eqnarray}\label{9n3jg}\begin{aligned}
\frac{E_+ b^2+\gamma d_i^2}{(\gamma d_i-\xi)^2} \E \frac{1}{N}(cm-cm_N) &= -(E_+\Phi_1+\Phi_2)\E (-2X_{42}+2X_{43}-2X_{44})+O(\Psi^5)\\&=\frac{1}{\gamma}\E (-2X_{42}+2X_{43}-2X_{44})+O(\Psi^5)
\end{aligned}\end{eqnarray}
Substituting \eqref{9n3jg} back into \eqref{349b8}, we have
\begin{eqnarray}\label{x32'vsimple}
\E X_{32}'= E_+ \E X_{32}+\frac{b}{\gamma d_i^2-\xi}\E \frac{1}{N}(cm-cm_N)+O(\Psi^5)
\end{eqnarray}
Moreover, if summing over $i$ and taking the average for \eqref{9n3jg}, we get the optical theorem derived from $X_{32}$,\begin{eqnarray}\label{optfromx32}
E\frac{1}{N}(cm-cm_N)=\frac{1}{N}\sum_i \E (-2X_{42}+2X_{43}-2X_{44})+O(\Psi^5).
\end{eqnarray}

\subsection{Optical theorem from $m_N X_{22}$}
In this subsection we show that $(E_+-z)\E X_{22}$ can be decomposed as linear combinations of $\E X_{42}, \E X_{44}$ and $\E \tX_{44}$, which is an optical theorem derived from $m_N X_{22}$. Recall that $m_N = \frac{1}{M}\sum_a G_{aa}$. Using  \eqref{taylorofGaa} and \eqref{taylorofbetaa}, we get
 \begin{eqnarray}\label{mnx22exp}\begin{split}
\frac{1}{N^2}\sum_{a,j} G_{aa}G_{ij}G_{ji} &=\frac{1}{N^2}\sum _{a,j}\left(\frac{b}{\gamma d_a^2-\xi}+\frac{-\gamma d_a^2 \mY_a}{(\gamma d_a^2-\xi)^2} + \frac{-b^2 \mZ_a}{(\gamma d_a^2-\xi)^2} \right. \\&+
\left. \frac{\gamma d_a^2 \tb \mY_a^2}{(\gamma d_a^2-\xi)^3}+\frac{ 2\gamma d_a^2 b \mY_a \mZ_a }{(\gamma d_a^2-\xi)^3}+\frac{b^3 \mZ_a^2}{(\gamma d_a^2-\xi)^3}\right) \times G_{ij}G_{ji}+O(\Psi^5).
\end{split}\end{eqnarray}
After taking expectation, we find the first term on the right side of \eqref{mnx22exp} is already fully expanded with index $a$.
Recall notations in \eqref{notatvarPhi}. By calculations, we find that the second to the sixth terms in \eqref{mnx22exp} are respectively \begin{eqnarray*}
\E \frac{1}{N^2}\sum_{a,j}\frac{-\gamma d_a^2 \mY_a}{(\gamma d_a^2-\xi)^2}G_{ij}G_{ji}= -\gamma \psi_2 \E X_{32}+ 4\gamma^2 \tb \psi_3 \E X_{44} - \gamma^2 (E_+ b+\tb) \psi_3 \E \tX_{44}+O(\Psi^5),
\end{eqnarray*}
\begin{eqnarray*}\begin{aligned}
\E \frac{1}{N^2}\sum_{a,j} \frac{-b^2 \mZ_a}{(\gamma d_a^2-\xi)^2} G_{ij}G_{ji} &= -b^3\varphi_2(E_+ -z)\E X_{22}-\gamma E_+ b^2 \varphi_2 \E X_{32}+8\gamma^2 E_+ b \psi_3 \E X_{44}\\&-(\gamma^2 E_+ b \psi_3 +\gamma^2 E_+^2 b^3 \varphi_3)\E \tX_{44}+O(N^{-1/2}\Psi^3),
\end{aligned}\end{eqnarray*}
\begin{eqnarray*}
\frac{1}{N^2}\sum_{a,j} \frac{\gamma d_a^2 \tb \mY_a^2}{(\gamma d_a^2-\xi)^3}G_{ij}G_{ji} = \gamma^2 \tb \psi_3(\E X_{42}+2\E \tX_{44})+O(\Psi^5),
\end{eqnarray*}
\begin{eqnarray*}
\frac{1}{N^2}\sum_{a,j} \frac{ 2\gamma d_a^2 b \mY_a \mZ_a }{(\gamma d_a^2-\xi)^3}G_{ij}G_{ji}=2\gamma^2 E_+ b \psi_3 \E X_{42}+O(\Psi^5),
\end{eqnarray*}
\begin{eqnarray*}
\frac{1}{N^2}\sum_{a,j} \frac{b^3 \mZ_a^2}{(\gamma d_a^2-\xi)^3}G_{ij}G_{ji} =\gamma^2 E_+^2 b^3\varphi_3 (\E X_{42}+2\E \tX_{44})+4\gamma^2 E_+ b \psi_3 \E \tX_{44}+O(N^{-1/2}\Psi^3).\end{eqnarray*}
Combining above with \eqref{mnx22exp}, we find \begin{eqnarray}\begin{aligned}
\E cm_N X_{22}&= b\varphi_1 \E X_{22} -b^3 \varphi_2 (E_+-z)\E X_{22} -(\gamma\psi_2+\gamma E_+ b^2 \varphi_2)\E X_{32}\\&+\Phi_1(\E X_{42}+4\E X_{44}+\E \tX_{44})+O(N^{-1/2}\Psi^3)
\end{aligned}\end{eqnarray}
Since $cm_N X_{22} = cm X_{22}- X_{32}$, using Lemma \ref{Phivalue} and $\varphi_1 = cm/ b$, we obtain
\begin{eqnarray}\label{optfrommnx22}(E_+-z)\E X_{22} =-\E X_{42} -4\E X_{44}-\E \tX_{44}+O(N^{-1/2}\Psi^3).
\end{eqnarray}

\subsection{Proof of Lemma \ref{multiErel} and Lemma \ref{kao4o}}
\noindent \textsc{\textbf{Proof of Lemma \ref{multiErel}}.} We only show that $\E X_{44}' =E_+ \E X_{44} +O(\Psi^5)$ and the others are similar. By expanding with lower index $\alpha$ using \eqref{gfiimu}, we find \begin{eqnarray}\label{91c7j}\begin{aligned}\sum_\alpha G_{i\alpha} G_{\alpha s}& = \sum_\alpha \gamma G_{\alpha\alpha}^2 G_{ij}^{(\alpha)}Y_{j\alpha}Y_{k\alpha}G_{ks}^{(\alpha)}= \sum_a \gamma G_{\bba\bba}^2 d_a^2 G_{ia}^{(\bba)}G_{as}^{(\bba)} +\sum_a \gamma G_{\bba\bba}^2 G_{ia}^{(\bba)}d_a x_{k\bba}G_{ks}^{(\bba)}\\ & +\sum_a \gamma G_{\bba\bba}^2 G_{ij}^{(\bba)}x_{j\bba}d_a G_{a s}^{(\bba)}+\sum_\alpha \gamma G_{\alpha\alpha}^2 G_{ij}^{(\alpha)} x_{j\alpha} x_{k\alpha} G_{ks}^{(\alpha)}.\end{aligned}\end{eqnarray}
Since we have  \begin{eqnarray*}\begin{aligned}
\E \frac{1}{N^3}\sum_{a,s,t} \gamma G_{\bba\bba}^2 G_{ia}^{(\bba)}d_a x_{k\bba}G_{ks}^{(\bba)}G_{st}G_{ti}&=\E \frac{1}{N^3}\sum_{a,s,t} \gamma G_{\bba\bba}^2 G_{ia}^{(\bba)}d_a x_{k\bba}G_{ks}^{(\bba)}G_{st}^{(\bba)}G_{ti}^{(\bba)}+O(\Psi^5)\\&=\E \frac{1}{N^3}\sum_{a,s,t} \frac{\gamma \tb^2}{(\gamma d_a^2-\xi)^2}G_{ia}^{(\bba)} d_a x_{k\bba}G_{ks}^{(\bba)}G_{st}^{(\bba)}G_{ti}^{(\bba)}+O(\Psi^5)=O(\Psi^5),
\end{aligned}\end{eqnarray*}
where in the first step we use \eqref{gfiminus} and in the second step we use   \eqref{taylorofGbba}, \eqref{taylorofalphaa} and \eqref{mUVorder}, it follows by substituting \eqref{91c7j} into $X_{44}'$ that
\begin{eqnarray}\label{ju641}\begin{aligned} \E \frac{1}{N^3}\sum_{\alpha,s,t} G_{i\alpha}G_{\alpha s}G_{st}G_{ti}&= \E\sum_{a,s,t}  \frac{\gamma d_a^2}{N^3}G_{\bba\bba}^2 G_{ia}^{(\bba)}G_{as}^{(\bba)} G_{st}G_{ti} + \E\sum_{\alpha,s,t} \frac{\gamma }{N^3}G_{\alpha\alpha}^2 G_{ij}^{(\alpha)} x_{j\alpha} x_{k\alpha} G_{ks}^{(\alpha)}G_{st}G_{ti}+O(\Psi^5)\end{aligned}\end{eqnarray}
Recall $\psi_2$ and $\varpi_2$ defined in \eqref{notatvarPhi}. By using \eqref{gfiij}, \eqref{gfiminus}, \eqref{taylorofGbba} and  \eqref{taylorofalphaa} the first term on the right side above equals \begin{eqnarray}\label{ju642}
\E \sum_{a,s,t,\alpha} \frac{\gamma^2 d_a^2}{N^4(\gamma d_a^2-\xi)^2}G_{i\alpha}^{(a\bba)}G_{\alpha s}^{(a\bba)}G_{st}^{(a)} G_{ti}^{(a)}=\gamma \psi_2 \E X_{44}'+O(\Psi^5).
\end{eqnarray}
Using  \eqref{taylorofGbba} and  \eqref{taylorofalphaa}  and \eqref{taylorofGalphaalpha}, the second term equals \begin{eqnarray}\label{ju643}
\E \sum  \frac{\gamma \tb^2}{N^4(\gamma d_a^2-\xi)2}G_{ij}^{(\bba)}G_{js}^{(\bba)}G_{st}^{(\bba)}G_{ti}^{(\bba)}+\E\sum_{\alpha>2M, j,s,t}  \frac{\gamma}{N^4 b^2}G_{ij}^{\alpha}G_{js}^{(\alpha)}G_{st}^{(\alpha)}G_{ti}^{(\alpha)}= \gamma \varpi_2 \E X_{44}+O(\Psi^5).
\end{eqnarray}
Combining \eqref{ju641}, \eqref{ju642} and \eqref{ju643} together with the help of \eqref{relofvarphi}, we conclude that  $\E X_{44}'=E_+ \E X_{44}+O(\Psi^5)$.\\

\noindent \textsc{\textbf{Proof of Lemma \ref{kao4o}}.} We only show the first one and the others can be proved similarly. Using \eqref{gfiminus}, we have \begin{eqnarray}\label{bu8af}\begin{aligned}
&\frac{1}{N^2}\sum_{\alpha,\beta}^{(\bba)} G_{i\alpha}^{(a\bba)}G_{\alpha\beta}^{(a\bba)}G_{\beta i}^{(a\bba)}=\frac{1}{N^2}\sum_{\alpha,\beta}^{(\bba)}  \left(G_{i\alpha}-\frac{G_{i\bba}G_{\bba \alpha}}{G_{\bba\bba}}-\frac{G_{ia}^{(\bba)}G_{a \alpha}^{(\bba)}}{G_{aa}^{(\bba)}} \right)G_{\alpha\beta} G_{\beta i} \\
&- \frac{1}{N^2}\sum_{a,\alpha,\beta}^{(\bba)}G_{i\alpha}^{(a\bba)}\left(\frac{G_{\alpha \bba}G_{\bba \beta}}{G_{\bba\bba}}+\frac{G_{\alpha a}^{(\bba)}G_{a\beta}^{(\bba)}}{G_{aa}^{(\bba)}}\right)G_{\beta i}-\frac{1}{N^2}\sum_{a,\alpha,\beta}^{(\bba)} G_{i\alpha}^{(a\bba)}G_{\alpha\beta}^{(a\bba)}\left(\frac{G_{\beta \bba}G_{\bba i}}{G_{\bba\bba}}+\frac{G_{\beta a}^{(\bba)}G_{ai}^{(\bba)}}{G_{aa}^{(\bba)}}\right).
\end{aligned}\end{eqnarray}
We find that  \begin{eqnarray}\label{89afl}\begin{aligned}
&\E \frac{1}{N^2}\sum_{\alpha,\beta}^{(\bba)} \frac{G_{i\bba}G_{\bba \alpha}}{G_{\bba\bba}}G_{\alpha\beta} G_{\beta i}\\
&=\E \frac{1}{N^2}\sum G_{\bba\bba} G_{ip}^{(\bba)}Y_{p\bba}Y_{q\bba}G_{q\alpha}^{(\bba)}G_{\alpha\beta}^{(\bba)}G_{\beta i}^{(\bba)}+O(\Psi^5)\\
&=\E\frac{1}{N^2}\sum  \gamma^2 G_{\bba\bba}G_{aa}^{(\bba)2} d_a^2 G_{i\mu}^{(\bba)}x_{a\mu}x_{a\nu}G_{\nu\alpha}^{(\bba)}G_{\alpha\beta}^{(\bba)}G_{\beta i}^{(\bba)}\\
&\relphantom{EE}+\E \frac{1}{N^2}\sum \gamma G_{\bba\bba} G_{ip}^{(\bba)}x_{p\bba}x_{q\bba}G_{q\alpha}^{(\bba)}G_{\alpha\beta}^{(\bba)}G_{\beta i}^{(\bba)}+O(\Psi^5)\\
&=\E\frac{1}{N^3}\sum \frac{\gamma^2 d_a^2}{(\gamma d_a^2-\xi)\tb}G_{i\mu}^{(\bba)}G_{\mu \alpha}^{(\bba)}G_{\alpha \beta}^{(\bba)}G_{\beta i}^{(\bba)}+\E\frac{1}{N^3}\sum \frac{\gamma \tb}{\gamma d_a^2-\xi}G_{ip}^{(\bba)}G_{p\alpha}^{(\bba)}G_{\alpha\beta}^{(\bba)}G_{\beta i}^{(\bba)}+O(\Psi^5)\\
&=\frac{\gamma^2 d_a^2 E_+^3}{(\gamma d_a^2-\xi)\tb}\E X_{44}+ \frac{\gamma \tb E_+^2}{\gamma d_a^2-\xi}\E X_{44}+O(\Psi^5)
%&=\frac{1}{N^3}\sum h(d_a) \frac{\gamma}{\tb}
\end{aligned}
\end{eqnarray}
where in the second equality we use that \begin{eqnarray}\label{32or5}\begin{aligned}
&\E \frac{1}{N^2}\sum \gamma G_{\bba\bba}G_{ia}^{(\bba)}d_a x_{q\bba}G_{q\alpha}^{(\bba)}G_{\alpha\beta}^{(\bba)}G_{\beta i}^{(\bba)}\\&=\E \frac{1}{N^2}\sum  \frac{\gamma \tb}{\gamma d_a^2-\xi}G_{ia}^{(\bba)} d_a  x_{q\bba}G_{q\alpha}^{(\bba)}G_{\alpha\beta}^{(\bba)}G_{\beta i}^{(\bba)}+O(\Psi^5)=O(\Psi^5).
\end{aligned}\end{eqnarray}
And we also have \begin{eqnarray}\label{43ca0}\begin{aligned}
\E \frac{1}{N^2}\sum_{\alpha,\beta}^{(\bba)} \frac{G_{ia}^{(\bba)}G_{a \alpha}^{(\bba)}}{G_{aa}^{(\bba)}} G_{\alpha\beta} G_{\beta i} =&\E \frac{1}{N^2}\sum  \gamma G_{aa}^{(\bba)}G_{i\mu}^{(\bba)}x_{a\mu}x_{a\nu}G_{\nu\alpha}^{(\bba)}G_{\alpha\beta}G_{\beta i}\\
&=-\E\frac{1}{N^3}\sum \frac{\gamma}{\tb}G_{i\mu}^{(\bba)}G_{\mu\alpha}^{(\bba)}G_{\alpha\beta}^{(\bba)}G_{\beta i}^{(\bba)}+O(\Psi^5)\\
&=-\frac{ \gamma E_+^3}{\tb}\E X_{44}+O(\Psi^5).
\end{aligned}\end{eqnarray}
Combining \eqref{89afl} and \eqref{43ca0}, we find \begin{eqnarray}
\E \frac{1}{N^2}\sum_{\alpha,\beta}^{(\bba)}\left( \frac{G_{i\bba}G_{\bba \alpha}}{G_{\bba\bba}}+\frac{G_{ia}^{(\bba)}G_{a \alpha}^{(\bba)}}{G_{aa}^{(\bba)}} \right)G_{\alpha\beta} G_{\beta i}= \frac{\gamma E_+^2(E_+ b+\tb)}{\gamma d_a^2-\xi}\E X_{44}+O(\Psi^5).
\end{eqnarray}
By similar steps, we obtain \begin{eqnarray}
\E \frac{1}{N^2}\sum_{\alpha,\beta}^{(\bba)}  G_{i\alpha}^{(a\bba)}\left(\frac{G_{\alpha \bba}G_{\bba \beta}}{G_{\bba\bba}}+\frac{G_{\alpha a}^{(\bba)}G_{a\beta}^{(\bba)}}{G_{aa}^{(\bba)}}\right)G_{\beta i}= \frac{\gamma E_+^2(E_+ b+\tb)}{\gamma d_a^2-\xi}\E X_{44}+O(\Psi^5),
\end{eqnarray}
and \begin{eqnarray}
\E \frac{1}{N^2}\sum_{\alpha,\beta}^{(\bba)} G_{i\alpha}^{(a\bba)}G_{\alpha\beta}^{(a\bba)}\left(\frac{G_{\beta \bba}G_{\bba i}}{G_{\bba\bba}}+\frac{G_{\beta a}^{(\bba)}G_{ai}^{(\bba)}}{G_{aa}^{(\bba)}}\right)= \frac{\gamma E_+^2(E_+ b+\tb)}{\gamma d_a^2-\xi}\E X_{44}+O(\Psi^5).
\end{eqnarray}
From \eqref{bu8af} and above three equations, we conclude the first one in the lemma. \qed

%\bibliographystyle{plain}
%\bibliography{information+noise_reference}

\end{document}